\documentclass[10pt]{article}

\usepackage{amsmath} %for dealing with mathematics,
\usepackage{amsthm} %for dealing with theorem environments,

\usepackage{subfigure} %for getting the subfigures e.g., "Figure 1a and 1b" etc.
\usepackage{url} %It provides better support for handling and breaking URLs.
\usepackage{amssymb}
\usepackage{mathtools} %to use \dcase
\usepackage{float}
\usepackage{authblk}

\newtheorem{proposition}{Proposition}
\newtheorem{problem}{Problem}

\newcommand{\R}{{\mathbb R}}

\newcommand{\ignore}[1]{}

\newcommand{\cKa}{{\mathcal K}_{a}}
\newcommand{\cKb}{{\mathcal K}_{b}}

\begin{document}

\title{Distributional barycenter problem through data-driven flows}

\author[a]{Esteban G. Tabak}
\affil[a]{Courant Institute of Mathematical Sciences, 251 Mercer Street, New York, NY 10012, USA}
            
\author[b]{Giulio Trigila}
\affil[b]{Baruch College, CUNY, 55 Lexington avenue, New York, NY 10010, USA}
\author[a]{Wenjun Zhao}

\date{}  
\maketitle

\begin{abstract}
%% Text of abstract
A new method is proposed for the solution of the data-driven optimal transport barycenter problem and of the more general distributional barycenter problem that the article introduces. The method improves on previous approaches based  on adversarial games, by slaving the discriminator to the generator, minimizing the need for parameterizations and by allowing the adoption of general cost functions. It is applied to numerical examples, which include analyzing the MNIST data set with a new cost function that penalizes non-isometric maps.      
\end{abstract}

\section{Introduction}  
Optimal transport and the related Wasserstein barycenter problem have undergone rapid development during the last ten years, with a particular focus on applications to the analysis of data and machine learning \cite{kolouri2017optimal}, ranging from gene expression \cite{schiebinger2019optimal} to economics  \cite{galichon2018optimal}. Procedures based on optimal transport have been used for density and conditional density estimation \cite{TT1, tabak2020conditional}, data augmentation  \cite{pavon2018data}, image classification \cite{kolouri2016continuous, wang2010optimal, yang2018complex}, computer vision \cite{solomon2015convolutional, wang2013linear, AHT, rabin2014adaptive}, factor discovery \cite{yang2019conditional} and data imputation \cite{TT3}.

Given two probability %density functions (pdf) 
distributions $\rho$ and $\mu$, the optimal transport problem (\cite{monge1781memoire,kantorovich1948v,santambrogio2015optimal}) seeks the map $T$ with minimal cost among those satisfying the push forward condition $\mu=T_{\#}\rho$, with a cost function determined by the application at hand. In the barycenter problem, a conditional distribution $\rho(x|z)$ is mapped to a single, unknown  distribution $\mu(y)$, which minimizes the sum over $z$ of the transportation cost from $\rho$ to $\mu$.

Some recent methodologies for the numerical solution of the data-based barycenter problem apply only to a canonical cost function, the squared Euclidean distance between points. The advantage of restricting attention to this or similar cost functions is that one can fully characterize the solution in terms of a convex potential, thus bypassing the need to actually perform a total cost minimization. However, a number of applications call for more general, field-specific cost functions. Consider for illustration the following instances:
\begin{enumerate}

\item The distributions  $\rho(x|z)$ underlying real world data are often defined on high dimensional spaces, yet they concentrate near a manifold $\mathcal{M}$ of  dimension $m$ smaller than the dimension $d$ of the ambient space.  Exploiting this geometric property of the data reduces the complexity of the map, which should be a function of $m$ rather than $d$. 
The geometry underlying the data encodes the nature of a system, so models consistent with it have a more meaningful data correlation structure.
One way to carry out this program is to use a cost function that penalizes maps $T$ moving data outside the manifold $\mathcal{M}$.  
Even in low dimensions, data often concentrates on or near a non-Euclidean sub-manifold, such as the Earth's surface for climate-related data. 

\item The introduction of a new cost function is often dictated purely by properties that one wishes to impose on the barycenter. We introduce in section \ref{sec:NumEx}, in the context of an application to the MNIST data set, a cost function favoring isometric maps. This results in a smoother, more interpretable barycenter of hand-written digits, modeled as distributions in pixel space.  
\end{enumerate}

This last example goes beyond the realm of optimal transport, as the cost function does not adopt the form of the expect value of a pairwise cost $c(x, y)$. We call such extensions of the Wasserstein barycenter problem, \emph{distributional barycenter problems}. They can be used to enforce problem-dependent desirable conditions on the conditional maps, such as proximity to prescribed priors.
 
The methodology for the solution of the data-driven distributional barycenter problem proposed in this article can be used with general cost functions. 
It improves significantly over previous approaches to the barycenter problem based on adversarial games (\cite{ELT, tabak2020conditional, yang2019conditional}). The latter have two players:  one that proposes cost-minimizing maps through time-evolving flows, and another that builds test functions to enforce the push-forward condition. The new approach slaves the test-functions to the flows, thus making the latter self-driven. Moreover, these flows are essentially non parametric, with a kernel's bandwidth as their single parameter.

\subsection{Prior work} 
While optimal transport on Riemannian manifolds has been broadly studied from an analytical perspective  (\cite{feldman2002monge, santambrogio2015optimal}), few algorithms have been proposed for its numerical solution. Most are based on a regularization of optimal transport \cite{solomon2015convolutional, tenetov2018fast} or are specific to particular manifolds \cite{yair2019optimal}. The work in \cite{lavenant2018dynamical} finds a smooth interpolation of densities on discrete surfaces using the dynamical approach of Benamou and Brenier \cite{BB}. This approach, though grounded as ours on gradient flows, uses a different flow and requires the knowledge of the densities to be transported rather than samples thereof. 

The optimal transport barycenter problem and its dual formulation were introduced in \cite{AguehBarycenters}. One of the first proposed methodologies for the numerical solution of the dual problem, a saddle point optimization problem,  appeared in \cite{carlier2015numerical}, where a modification of linear programming was adopted to compute the potentials associated to the optimal maps. Here we propose an alternative derivation of the formulation in \cite{carlier2015numerical}, better suited for the discussion leading to the algorithm proposed in section \ref{sec:F}.

\subsection{Original contribution}
The main contribution of this paper is an original methodology for the solution of the barycenter problem under general cost functions, through a time-dependent flow that pushes the marginal distributions $\rho(x|z)$ to their barycenter $\mu(y)$. Its main novel aspects are that the maps require no parameterization and that the test function enforcing that all conditional distributions be mapped to the same barycenter is slaved to the maps.
This yields a minimization problem with one constraint rather than a saddle point problem, equivalent to a minimization problem with infinitely many constraints. This reduction is achieved by proposing a specific --but sufficient--  form for the test function $F$, therefore bypassing the adversarial formulation in which a Lagrangian is minimized over maps and maximized over test functions. 

A numerical implementation based on a variation of the penalty method (\cite{nocedal2006numerical}) results in a method that builds arbitrarily complex maps through flows 
and permits the adoption of very general cost functions. In particular, a new cost function is proposed that penalizes non-isometric conditional maps, a natural way to minimize data distortion. 

\subsection{Organization of the article}
Section \ref{sec:BOT} reviews the barycenter problem and its dual. Section \ref{sec:F} proposes two specific test functions, yielding two alternative formulations, section \ref{sec:DDp1p2} develops their data driven version and section \ref{sec:penalty} introduces a penalty method for their numerical solution. Section \ref{sec:NumEx} contains numerical experiments on both synthetic data and the MNIST data set, using various test and cost functions. In particular, subsection \ref{subsec:Digits} introduces a new cost penalizing maps far from isometric, and subsection \ref{subsec:HiddenSignal} uses the barycenter problem to recover a hidden signal behind a time series defined on a sphere.

\section{Data-driven distributional barycenter problem}\label{sec:BOT}
Given a conditional probability distribution $\rho (x|z)$, the optimal transport barycenter problem seeks a target $\mu(y)$ and $z$-dependent maps $T(x, z)$  from $\rho$ to $\mu$ with minimal total transportation cost:
\begin{equation}\label{eq:DefBary}
\min_T \int \int  c(x,T(x,z)) \rho (x,z) dx dz, \quad s.t. \quad \forall z\ T\# \rho (\cdot |z) := \rho_{T}(\cdot|z) = \mu. 
\end{equation}
Examples of cost functions are $p$-norms, as in the canonical cost $c(x,y) = \frac 1 2 ||x-y||_2^2$, and the squared geodesic distance on a manifold.

We will consider the more general distributional barycenter problem
\begin{equation}\label{eq:DefDBary}
\min_T C\left(T(x, z), \rho\right), \quad s.t. \quad \forall z\ T\# \rho (\cdot |z) := \rho_{T}(\cdot|z) = \mu, 
\end{equation}
where $C$ can adopt forms different from the expected value of a pairwise cost function $c(x, T(x, z))$ of optimal transport. Examples of such more general costs include the Fermat distance introduced in \cite{sapienza2018weighted} and a cost function introduced below to penalize deviations from isometry. For concreteness and to enable comparison with prior work, we describe below our methodology in the context of regular pairwise costs $c(x, y)$, explaining afterwards how it extends, quite straightforwardly, to the general case. The only constraint on $C(T, \rho)$ is that it must admit a data-based formulation, i.e. its dependence on $\rho$ must be translatable into an expression involving only samples thereof. For the regular pairwise cost, such formulation simply replaces expected values by empirical means over the data points. 

As the pushforward condition expresses the requirement that the random variable $y=T(x,z)$ be independent of $z$, it can be rewritten without explicit reference to the unknown barycenter $\mu$. If $z$ and  $y=T(x, z)$ are independent, then $\rho_{T}(y, z)=\mu(y)\nu(z)$, so
$$
\int F(y,z)\rho_{T}(y,z) dydz= \int F(T(x,z),z)\rho(x,z) dxdz = 0 
$$
for every test function $F$ satisfying $\int F(y,z)\nu(z)dz=0$. The converse is also true, leading to the minimax formulation of the barycenter problem:
\begin{equation}\label{eq:Carl}
\begin{dcases}
\min_{T}\max_{F} \int c(x,T(x,z))\rho(x,z)dx dz + \int F(T(x,z),z)\rho(x,z)dxdz \\
\forall y \ \int F(y,z)\nu(z) dz =E_{z}[F(y,\cdot)]=0
\end{dcases}
\end{equation}
A comparison between (\ref{eq:Carl}) and the formulation in \cite{AguehBarycenters} reveals that the test function $F$ is the Lagrange multiplier $\psi(y, z)$ of the dual Kantorovich problem. 

The constraint in (\ref{eq:Carl}) can be satisfied automatically by subtracting from $F$ its expected value $E_{z}[F]$, which yields the unconstrained variational problem 
\begin{equation}\label{eq:Bary}
\min_{T}\max_{F}  L= \int c(x,T(x,z))\rho(x,z)dx dz + \int (F(y,z) - E_{z}[F])\rho_{T}(y,z)\ dydz .
\end{equation}
The first integral in (\ref{eq:Bary}) corresponds to the cost function of optimal transport, to be extended below to far more general costs.  For future reference, we will denote this integral as $L_C$, and the second integral, designed to test the fulfillment of the pushforward condition, as $L_F$:
$$ L_C = \int c(x,T(x,z))\rho(x,z)dx dz, \quad L_F = \int (F(y,z) - E_{z}[F])\rho_{T}(y,z)\ dydz. $$

As noted in \cite{TT1}, the dual Kantorovich problem is a natural starting point for a data driven formulation of optimal transport. In particular,  (\ref{eq:Bary}) has two main advantages over (\ref{eq:DefBary}): the unknown barycenter $\mu$ does not appear explicitly, and the objective function is a sum of expected values, which can be replaced by their empirical counterpart
\begin{equation}\label{eq:DataBary}
\min_{T}\max_{F} \frac{1}{N}\sum_{i} \left[c(x_{i},T(x_{i},z_{i})) +  F(T(x_{i},z_{i}),z_{i}) - \frac{1}{N}\sum_{j}F(T(x_{i},z_{i}),z_{j})\right]
\end{equation}
when only samples $\left(x_i, z_i\right)$ of $\rho(x, z)$ are available.

\section{Two choices for the test function $F(y,z)$}\label{sec:F}
This section introduces a new algorithm for the numerical solution  of the optimization problem in (\ref{eq:DataBary}). We first define an the evolution equation for $T$ through the gradient descent of $L$:

\begin{equation}\label{eq:Tdot}
 \dot T = -\left.\frac{\delta L}{\delta T}\right|_{x, z}= -\left[\nabla_y c(x, y)   + \nabla_y F(y, z) \right] \rho(x, z), \quad y = T(x, z).
\end{equation}
Notice that, for the canonical squared-distance cost, the first order optimality condition  $\dot T=0$ recovers the well-known relationship between the optimal map $T^{\ast}$ and the optimal potential $F^{\ast}$, i.e. $x = T^{\ast}(x, z) - \nabla_y F^{\ast}(y, z)$.

Thus the evolution of the map $T$ is defined in terms of the test function $F$. 
%How should we represent $F$?
Since the role of $F$ is to penalize any dependence of $\rho_{T}(y|z)$ on $z$, it is natural to think that it should be able to resolve the family of distributions $\rho_{T}(y|z)$. The following two propositions clarify this point. We will use them to reformulate the problem in (\ref{eq:DataBary}) so that the adversarial game played by $F$ and $T$ is reduced to a pure minimization algorithm over $T$.    

\begin{proposition}\label{prop:Jensen}
 If $F(y, z)=\rho_{T}(y|z)$ then the second term ($L_F$) in (\ref{eq:Bary}) is always strictly positive unless $\rho_{T}(y|z)$ is independent of $z$. 
\end{proposition}

\begin{proof}
 We can rewrite $L_F$ as
 \begin{equation}
L_F = \int\left[ F(y, z) \rho_{T}(y|z)\nu(z)dz - \int F(y, z) \bar{\rho}_{T}(y)\nu(z)dz\right] dy,
 \end{equation}
where $\bar{\rho}_{T}(y)=\int \rho_{T}(y|w)\nu(w) dw$. Substituting $F(y, z)=\rho_{T}(y|z)$ yields
\begin{equation}
L_F =  \int \left(E_{z}[\rho_{T}^{2}(y|\cdot)] - E_{z}[\rho_{T}(y|\cdot)]^{2}\right)dy.
\end{equation}
By Jensen's inequality, the integrand is strictly positive for all values of $y$ unless $\rho_{T}(y|z)$ does not depend on $z$. 
\end{proof}
This result suggests adopting $F(y,z)=\lambda\rho_{T}(y|z)$, a test function that evolves as the conditional distributions $\rho(x|z)$ are pushed forward toward their barycenter $\mu$. With this choice, the infinite dimensional maximization of (\ref{eq:Bary}) over $F$ reduces to the maximization over the scalar $\lambda$:
%Therefore, the optimization problem (\ref{eq:Bary}) becomes under Proposition \ref{prop:Jensen}:
%
\begin{problem}\label{prob:CKDE}
\begin{multline}\nonumber
 \min_T \max_\lambda \int c(x,T(x,z))\rho(x,z) dx dz + \\
 +\lambda \int \left[\rho_{T}(x|z) - \int \rho_{T}(x|w) \nu (w) dw \right] \rho_{T}(x,z) dx dz.
\end{multline}
 \end{problem}
Section \ref{sec:penalty} discusses in detail how to solve numerically Problem \ref{prob:CKDE}. Here we just point out that:
1) At all times, the information we have on $\rho_{T}(x, z)$ consists of samples thereof, i.e. the points $y^{i}=T(x^{i},z^{i})$ that have been transported by $T$, and 2) Since $L_F$ is non-negative, the maximization over $\lambda$ can be implemented through a penalty method, reducing Problem \ref{prob:CKDE} to a pure minimization problem. 

We show next that, alternatively, we can choose as test function $F(y,z)$ the product of two related functions, depending on $y$ and $z$ respectively:

\begin{proposition}\label{prop:Jensen2}

If $F(y, z)=f(y)g(z)$ where $g(z)=\int f(y) \rho_{T}(y|z)dy$, then $L_F$ in (\ref{eq:Bary}) is strictly positive unless the expected value of $f(y)$ under $\rho_{T}(y|z)$ is independent of $z$.
\end{proposition}

\begin{proof}
It is not difficult to see that, with $F$ given as above, one has  
\begin{equation}\label{eq:Test}
L_F = \int g(z)^2 \nu(z) dz - \left(\int g(z) \nu(z) dz \right)^2.
\end{equation}
By Jensen's inequality, (\ref{eq:Test}) is always non-negative, vanishing only if $g$ is independent of $z$.
\end{proof}
Under Proposition \ref{prop:Jensen2} we can relax (\ref{eq:Bary}) into 
\begin{problem}\label{prob:factor}
\begin{equation}\nonumber
 \min_{T}\max_{f} \int c(x,T(x,z))\rho(x,z) dx dz
 +\int g(z)^{2}\nu(z)dz - \left(\int g(z)\nu(z)dz\right)^{2}
\end{equation}
where $g(z)=\int f(y)\rho_{T}(y|z)dy$.
\end{problem}
\noindent
This formulation enforces the independence of $\rho_{T}(y|z)$ from $z$ in a weak sense, with test function $f(y)$. For instance, restricting $f$ to linear functions $f=\lambda y$ enforces that the conditional mean $\bar{y}(z)$ of $\rho_{T}(y|z)$  be independent of $z$. Notice that, in this case and under the canonical cost, the descent equation (\ref{eq:Tdot}) implies that the map $T$ is a $z$-dependent rigid translation, precisely the minimal family of maps able to remove conditional means.

These considerations suggest a preconditioning procedure whereby, rather than seeking the full barycenter from the start, one first limits the family of test functions and maps, yielding a less detailed but faster procedure that brings the $\rho(x|z)$ closer to each other. In particular, one can perform a preconditioning whereby only the conditional mean of $\rho(x|z)$ is removed, through a $z$-dependent rigid translation. Under the canonical cost, performing this preconditioning and subsequently computing the barycenter of the resulting push-forward distributions, results in the same barycenter that one would have found directly from the original ones. The proof, which
extends arguments in \cite{kuang2017preconditioning} to the barycenter problem, is the content of the following proposition.
%We show in section \ref{sec:feature} how Problem \ref{prob:factor} can also be essentially reduced to a pure minimization problem through a penalty method.

\begin{proposition}
Consider the following, two-stage procedure for finding the barycenter  of the conditional distributions $\rho(x|z)$ under the canonical cost
$ c(x, y) = \frac{1}{2}\|x - y\|^2$.
First restrict the maps to the $z$-dependent rigid translations
$$ w = T_1(x, z) = x + \bar{x} - \bar{x}(z), $$
which make the conditional means of the resulting random variable $W$ match.  Then find the full barycenter of the resulting conditional distributions $\mu_1(w|z)$ through a map 
$ y = T_2(w, z) $.
Then the composition of the two maps,
$$ y = T(x, z) = T_2\left(T_1(x, z), z\right) $$
solves the original barycenter problem.
\end{proposition}
\label{prop3}
\begin{proof}
Clearly the distribution $\mu(y)$ is independent of $z$, since $\mu$ is the barycenter of the $\mu_1(w|z)$. To prove optimality, it is enough \cite{AguehBarycenters,kuang2019sample} to show that 
\begin{enumerate}
\item $T$ is the gradient of a convex function:
$$ T(x, z) = \nabla_x \phi(x, z), \quad \phi(:, z) \ \hbox{convex for all $z$,} $$

\item every point $y$ is the geometrical barycenter of its pre-images under $T(x, z)$,
$$\forall y \  E_z\left[T^{-1}(y, z)\right] = y. $$

\end{enumerate}

\noindent
Since $\mu$ is the barycenter of the $\mu_1(w|z)$, both properties hold for $T_2$:
$$ T_2(w, z) = \nabla_w\psi(w, z), \quad \psi(:, z) \ \hbox{convex for all $z$,} \quad \forall y \  E_z\left[T_2^{-1}(y, z)\right] = y.$$
Then
$ T(x, z) = T_2\left( x + \bar{x} - \bar{x}(z), z\right) = \nabla_x \phi(x, z)$,
where
$ \phi(x, z) = \psi\left( x + \bar{x} - \bar{x}(z), z\right)$
is convex in $x$ for all values of $z$.
Also $ T^{-1}(y, z) = T_2^{-1}(y, z) + \bar{x}(z) - \bar{x}$,
so
$$  E_z\left[T^{-1}(y, z)\right]  = E_z\left[T_2^{-1}(y, z)\right] + \bar{x} - \bar{x} = y, $$
concluding the proof.

\end{proof}

Two natural questions arise from proposition \ref{prop3}: can one perform pre-conditioning under more general cost functions, and can one implement richer pre-conditioners that bring the $\rho(x|z)$ closer to each other than merely translating them so that their conditional means match. To answer these questions, notice that proposition \ref{prop3} allows one to start the follow-up barycenter problem directly from the $\mu_1(w|z)$ resulting from the pre-conditioning map, without any reference to the original random variable $X$. However, one does know the conditional pairing of $X$ and $W$, i.e. the map $w = T_1(x, z)$ or, in the data-driven case, the point $x_i$ that each $w_i$ originated from. It follows that one can perform pre-conditioning under any cost function $C(T, \rho)$ and with any family of test functions $F$, provided that, in the subsequent full barycenter problem, though starting from the $W = T_1(X, z)$, one computes the cost $C$ in terms of the original $X$:
$$ C_2\left(T, \mu_1\right) = C\left(T * T_1, T_1^{-1}\#\mu_1\right). $$
In the data-driven setting developed below, this formula simply translates into using $x_i$ in lieu of $w_i$ in $C$.

\section{Data-driven formulations}\label{sec:DDp1p2}
This section discusses the numerical representation of $\rho(y|z)$ and its use for implementing data-driven versions of Problems \ref{prob:CKDE} and \ref{prob:factor}. 

\subsection{Data driven Problem \ref{prob:CKDE}}
The map $y = T(x, z)$ is built from the composition of near-identity maps which yield, at each time-step of the algorithm, a current state of the map and a corresponding current conditional density $\rho_{T}(y|z)$. This conditional density, which evolves from $\rho(y|z)$ to $\mu(y)$, is known at all times through the points $y^{i}=T(x^{i},z^{i})$. A natural way to estimate $F(y, z) = \rho_{T}(y|z)$ from these samples is through a conditional kernel density estimation (CKDE) in the Nadaraya-Watson form (\cite{rosenblatt1969conditional,de2003conditional}):
\begin{equation}\label{eq:Fzcat}
F(y, z_k)= \rho_{T}(y|z_k) \approx \frac{\sum_i  \cKa (y,y_i)\cKb (z_k, z_i)}{\sum_j \cKb(z_k, z_j) } = \sum_{i=1}^{N}  \cKa (y,y_i)Z_{ik}.
\end{equation}
The kernel functions $\cKa(y,y_{i})$, nonnegative and normalized so as to integrate to one, have centers $y_{i}$ and  bandwidth $a$ --the algorithm's only free parameter, other than the choice of the kernels themselves, for which  isotropic Gaussians were adopted in all the numerical examples below. The matrix $Z \in \R^{N\times N}$ is a normalized version of similar kernels in $z$-space:
\begin{equation}\label{eq:Zik}
 Z_{ik} = \frac{\cKb(z_k,z_i)}{ \sum_{j=1}^{N} \cKb(z_k,z_j) }.
\end{equation}
With this choice for $F$, the empirical version of the term in square brackets in Problem \ref{prob:CKDE} adopts the form
\begin{equation}
 \rho_{T}(y_l|z_l) - \mathbb{E}_z \rho_{T}(y_l|z) \approx \sum _i \cKa(y_l,y_i) \left[ Z_{il} - \frac{1}{N}\sum_k Z_{ik} \right]=\sum_i \cKa(y_l, y_i) C_{il},
\end{equation}
where the $N$ by $N$ matrix
\begin{equation}\label{eq:Cil}
 C_{il} =  Z_{il} - \frac{1}{N}\sum_k Z_{ik}
\end{equation}
can be precomputed at the onset of the procedure, since the values of $z_{i}$ remain unchanged throughout. Then the complete data driven formulation of Problem \ref{prob:CKDE} adopts the simple form
\begin{equation}\label{eq:DDprob1}
\min_y \max_\lambda \sum_i c(x_i,y_i) + \lambda  \sum_{i,l} \cKa(y_l,y_i) C_{il}.
\end{equation}

\subsection{Data driven Problem \ref{prob:factor}}
\label{sec:feature}
In order to evaluate (\ref{eq:Test}) from sample points, we rewrite $g(z)$ in the form 
\begin{equation}
 g(z) = \int f(y) \rho_{T}(y|z) dy = \int f(y) \frac{\rho_{T}(y, z)}{\nu(z)} dy = \int f(y) \frac{\rho_{T}(y, w)}{\nu(w)} \delta(w-z) dy dw,
\end{equation}
and propose the mollification
\begin{equation}
\delta(w-z) \approx \cKb(w,z), \quad \nu(w) = \frac{1}{n} \sum_j \cKb\left(w,z_j\right),
\end{equation}
with a positive kernel $\cKb$ with bandwidth $b$ that integrates to one. Then
\begin{equation}\label{eq:repg}
g(z) \approx \sum_k f\left(y_k\right) \frac{\cKb(z,z_k)}{\sum_l \cKb\left(z_l,z_k\right)},
\end{equation}
which we can substitute in the test function $F(y,z)$ according to the proposal in Proposition \ref{prop:Jensen2}, i.e. $F(y,z)=\lambda f(y)g(z)$. The resulting test component $L_F$ of the Lagrangian is 
\begin{multline}
L_F = \sum_{i}\left[F(T(x_{i},z_{i}),z_{i}) - \frac{1}{N}\sum_{j}F(T(x_{i},z_{i}),z_{j})\right]= \\
\lambda\sum_{i}f(y_{i})\sum_{k}f(y_{k})\left[Z_{ki}-\frac{1}{N}\sum_{j}Z_{kj} \right]=\lambda\sum_{i,k}f(y_{i})f(y_{k})C_{ki},
\end{multline}
where $y_{i}=T(x_{i},z_{i})$ and the matrices $Z_{ik}$ and $C_{il}$ are those defined in (\ref{eq:Zik}) and (\ref{eq:Cil}). The overall data-driven version of Problem \ref{prob:factor} with fixed test function $f$ then becomes
\begin{equation}\label{eq:DDprob2}
\min_y \max_\lambda \sum_i c(x_i,y_i) + \lambda  \sum_{i,k}  f(y_{i})f(y_{k})C_{ki}.
\end{equation}
The choice of the function $f$ specifies a relaxation of the pushforward condition, with $f(y)=y^l$ ($y^l$ here stands for the $l$th component of $y$) corresponding to moving each conditional mean $\bar{x^l}(z)$ to the mean $\bar{y^l}$ of the barycenter. To match the conditional means of all components $y^l$, as well as to enforce other moments, we can choose $f$ to be a vectorial function whose entries can be chosen, for instance, as a polynomial basis:  $f(y)=[f_1(y), f_2(y),.., f_m(y)]$, corresponding to the solution of 
\begin{equation}\label{eq:DDprob2b}
\min_y \max_\lambda \sum_i c(x_i,y_i) +  \lambda \sum_{i,k,l} f_{l}(y_{i})f_{l}(y_{k})C_{ki}.
\end{equation}
Notice that we do not need an independent factor $\lambda_l$ for each $f_l$, as each term
$\sum_{i,k} f_{l}(y_{i})f_{l}(y_{k})C_{ki}$ is independently non-negative, vanishing only when the expected value of $f_l$ agrees for all values of $z$ (We have proved this in proposition 2 for the problem posed in terms in distributions, and we will prove it below for the sample-based problem.) We may, however, weight each $f_l$ differently if desired. For instance, we might want to start with most of the weight on the linear components of $f$, so as to enforce the agreement of the conditional means, then slowly increase the weight of the quadratic components, to match all conditional covariances, and add more terms, either higher order polynomials or localized features, for a more detailed fulfillment of the pushforward condition. However, we have found empirically that simply pre-conditioning first with a linear $f$ is enough to speed up the subsequent convergence of problem 1 in all of its generality, bypassing the need for the ``continuous preconditioning'' that the procedure just described would entail.

\subsection{An alternative conditional density estimator}

Formula (\ref{eq:Zik}) for the matrix $Z_{ik}$ is not the only choice that makes (\ref{eq:Fzcat}) a robust conditional density estimator. The core requirements for $Z$ are:
\begin{enumerate}
  \item The entries $Z_{ik}$ must be nonnegative and add up to zero row-wise:
  $$ Z_{ik} \ge 0, \quad \sum_i Z_{ik} = 1, $$
to guarantee that the estimated $\rho_T(y|z_k)$ is positive and integrates to one. 
  
  \item $Z_{ik}$ must be large when $z_i$ and $z_k$ are close to each other, and small when they are far away. This follows from conceptualizing (\ref{eq:Fzcat}) as a regular kernel density estimation which has the various centers $y_i$ weighted by $Z_{ik}$. Then $Z_{ik}$ must provide a measure of how relevant $y_i$ is for an estimation of $\rho_T(y|z_k)$, i.e. how close the $z_i$ associated with $y_i$ is to $z_k$. The notion of closeness is, of course, problem dependent. For instance, for categorical factors $z$, a choice for $Z_{ik}$ vanishes whenever $z_i \ne z_k$.
\end{enumerate}

The particular form (\ref{eq:Zik}) for $Z_{ik}$ satisfies these properties, and it leads to robust and accurate numerical results in all examples that we have tried. Yet the resulting matrix $Z_{ik}$ is asymmetric, as only its rows, not its columns, are normalized. For reason that the following subsection will clarify, we prefer a matrix $Z$ that is symmetric and positive definite. Since a symmetric matrix $Z$ with nonnegative entries whose rows add up to one is necessarily bi-stochastic, a natural candidate is the unique bi-stochastic matrix $\tilde{Z}$ that derives from the symmetric and positive Kernel matrix
$ K_{ik} = \cKb (z_k, z_i)$
through Sinkhorn's factorization:
$ \tilde{Z} = D K D, $
where $D$ is a diagonal matrix with positive diagonal entries. 
%In practice, $\tilde{Z}$ and $D$ are found through the Sinkhorn-Knopp iterative procedure.
 Since $K$ is positive definite, so is $\tilde{Z}$, which also satisfies the required properties for (\ref{eq:Fzcat}) to provide a consistent conditional density estimator, and is in fact better balanced than $Z$, in the sense that all points $y_i$ have the same total weight (For points whose $z_i$ is an outlier, this weight concentrates mostly in self-estimation, while for points with $z_i$ in the core of the $z$-distribution, the weights are distributed among neighboring points in $z$, not necessarily $y$.)

We have found the numerical results with $\tilde{Z}$ and $Z$ to be nearly indistinguishable. Since $\tilde{Z}$ comes with better theoretical guarantees, we use $\tilde{Z}$ in the remaining of the article and in the numerical examples, renaming it $Z$ to avoid notational clumsiness.

The kernel-based matrix $Z_i^j$ is suitable for continuous factors $z$ with a notion of distance among points. Clearly, for categorical factors $z$, it should be replaced by the simpler
$$ Z_i^j = \begin{cases}\frac{1}{N_i} & \hbox{for $z_i = z_j$} \cr
                            0 & \hbox{otherwise,} \end{cases} \quad N_i =|\{z: z = z_i\}|,$$
also bi-stochastic, which simply discriminates among classes. To avoid repeating proofs and arguments, this can be considered as a particular case of the kernel-based $Z$ with vanishing small bandwidth, so that different $z_i$ do dot interact.

\subsection{Positivity of $L_F$ in the data-driven problem}
We saw in section \ref{sec:F} that, for the two particular choices of the test function $F$ corresponding to problems 1 and 2,  the $L_F$ in (\ref{eq:Bary}) is strictly positive unless all the marginals $\rho_T(y|z)$ agree. The positivity of $L_F$ allows us to pre-multiply it by a positive scalar $\lambda$, replacing the minimax formulation by a penalized minimization. We show here that the positivity of $L_F$ also holds in its data-driven version.

\begin{proposition}
 If the kernel matrices $\cKa$ and $\cKb$, with entries $\cKa(y_{k},y_{i})$ and $\cKb(z_{k},z_{i})$ respectively, are non-negative definite, then the test component $L_F$ of both (\ref{eq:DDprob1}) and (\ref{eq:DDprob2}) is non negative.  
\end{proposition}  
\begin{proof}
Notice that it is enough to show that the matrix $C$ given by (\ref{eq:Cil}) is non-negative definite, i.e. that
\begin{equation}
 \forall x, \ x^T C x \ge 0.
 \label{Cpos}
\end{equation}
This sufficiency of (\ref{Cpos}) for (\ref{eq:DDprob2}) is obvious, as its test component is precisely a sum of terms of this form, but (\ref{Cpos}) is also sufficient for (\ref{eq:DDprob1}), since its test component is the inner product of $C$ and $\cKa$, and the inner product of two non-negative definite matrices is a non-negative number, even when only one of them is symmetric ($C$, in general, is not.)
To prove (\ref{Cpos}), write
\begin{eqnarray*}
 \sum_{i,j} x_i C_{ij} x_j &=& \sum_{i,j} x_i Z_{ij} x_j - \frac{1}{N} \sum_{k,i,j} x_i Z_{ik} \left(x_j - x_k + x_k\right) \\  
 &=& - \frac{1}{N} \sum_{k,i,j} x_i Z_{ik} \left(x_j - x_k\right) \\ 
 &=& \sum_{k,i} x_i Z_{ik} \left(x_k - \bar{x}\right) \\ 
 &=& \sum_{k,i} \left(x_i - \bar{x}\right) Z_{ik} \left(x_k - \bar{x}\right) \ge 0,
\end{eqnarray*}
as the matrix $Z$ (i.e. $\tilde{Z}$) is non-negative definite by construction.

\end{proof}

\subsection{Extension to general cost functions}

We have so far restricted the cost component $L_C$ of the objective function to the expected value of a pairwise cost $c(x, T(x, z))$, as pertains optimal transport. However, it is clear from the data-based formulations derived that the only requirement one must impose on $L_C = C(T, \rho)$ is that $\rho$ should only appear through the expected value of functions, which can be replaced by their empirical counterpart when only samples $(x_i, z_i)$ of $\rho$ are known. Thus, for instance, in lieu of the pairwise cost
$ L_C = \int c(x, T(x, z)) \rho(x, z) dx dz$,
one may propose cost functions involving two points and their images under a common factor $z$,
$$ L_C = \int c\left(x_1, T\left(x_1, z\right), x_2, T\left(x_2, z\right)\right) \rho(x_1|z) \rho(x_2|z) \nu(z) dx_1 dx_2 dz.  $$
We will propose one such cost in an example below on hand-written digits, a cost that penalizes deviations of $T(:, z)$ from an isometry. A data-driven version of a cost of this form is
$$ L_C = \frac{1}{N^2} \sum_{i,j} c\left(x_i, T\left(x_i, z_i\right), x_j, T\left(x_j, z_j\right)\right) Z(z_i, z_j), $$
involving the bi-stochastic matrix $Z_i^j = Z(z_i, z_j)$ introduced above.

Most formulas in this article are written, for concreteness, in terms of pairwise cost functions. In order to apply them to more general costs, it is enough to insert the corresponding expression for $L_C$ and its derivatives, while all the formulas concerning $L_F$ remain unaltered.

\section{A penalty method}\label{sec:penalty}

Both (\ref{eq:DDprob1}) and (\ref{eq:DDprob2}) are minimax problems of a special kind, where the maximization is carried out over a single positive scalar quantity $\lambda$ whose optimal value is unbounded (as perteins the Lagrange multiplier of a single constraint requiring a non-negative quantity, $L_F$, to vanish). More effective than maximizing $L$ over $\lambda$ is to use a penalty method, whereby
%In this section we propose an algorithm belonging to the family of penalty methods. These methods are used in constrained optimization problems like (\ref{eq:DDprob1}) and (\ref{eq:DDprob2}) whose solution is characterized by an unbounded value of $\lambda$.  
$\lambda$ is externally increased at each iteration step to as to progressively enforce the constraint. 

Among the possible strategies for controlling $\lambda$, we propose one guaranteeing that $L_F$ decreases at every step, while not making $\lambda$ grow so fast as to effectively make the minimization of $L_C$ a secondary goal. The procedure applies to both (\ref{eq:DDprob1}) and (\ref{eq:DDprob2}), for concreteness we describe it here for (\ref{eq:DDprob1}):   
\begin{enumerate}
\item Initialize $y_i=x_i$, $\lambda = \lambda_0>0$, and a maximum number of iteration $niter$. The iteration count starts from $n=0$, and the learning rate from $\eta = \eta^0$. Precompute the matrix $C$ as defined in \eqref{eq:Cil}.

If $\lambda$ is sufficiently small, $L_C$ dominates the objective function, making the minimization problem convex ($L_C$ is typically convex, at least near the identity map.) Therefore, we set $\lambda_0 = 1/\max(\mbox{abs}(\mbox{eigs} [-F_{xx}(x)])) $ resulting in a semi positive definite Hessian.

\item While $n < niter $ and $y$ has not yet converged, tentatively evolve the learning rate through the formula $\eta^{n+1} = \min \{ 2.01 \eta^n , \eta^0 \}$.
\begin{enumerate}
\item Calculate the derivatives of the objective function (\ref{sec::gradHess}.)

\item Compute $\lambda^{n+1}$ according the criteria below, whith $\alpha > 0$:
\begin{multline}
 \left\langle \nabla_y c(x,y) + \lambda \nabla_y \sum_{l} \cKa(y_l,y)C_{:l}, \nabla_y \sum_{l} \cKa(y_l,y)C_{:l} \right\rangle \geq \\
 \geq \alpha \left\langle  \nabla_y \sum_{l} \cKa(y_l,y)C_{il}, \nabla_y \sum_{l} \cKa(y_l,y)C_{il} \right\rangle ,
\end{multline}
which implies the lower bound for $\lambda$: 
\begin{equation} \lambda \geq \alpha - \frac{ \langle \nabla_y c(x,y)  , \nabla_y \sum_{l} \cKa(y_l,y)C_{:l} \rangle   }{  \langle  \nabla_y \sum_{l} \cKa(y_l,y)C_{il}, \nabla_y \sum_{l} \cKa(y_l,y)C_{:l} \rangle } = \lambda_{min}.
\label{eq::lambda}
\end{equation}
Here the inner product $\langle \rangle$ between two functions $f(x,y,:)$ and $g(x,y,:)$ is defined as $\langle f, g\rangle = \sum_i f(x_i, y_i,i) g(x_i, y_i,i)$. 
%Therefore, the inner products in (\ref{eq::lambda}) are computed by substituting $x,y$ and $C_{:l}$ with $x_{i},y_{i}$ and $C_{il}$ respectively and then summing over $i$.
If $\lambda_{min}$ is larger than $\lambda^n$ and smaller than a threshold $\lambda^{max}$, set $\lambda^{n+1} = \lambda_{min}$. Otherwise, if $\lambda_{min} > \lambda^{max}$, set $\lambda^{n+1} = \lambda^{max}$, else set $\lambda^{n+1}=\lambda^{n}$. This guarantees that $\lambda^{n+1}$ is not smaller than $\lambda^n$, so $L(y,\lambda^{n+1}) \geq L(y,\lambda^n)$ is satisfied for any $y$.

\item Update $y$ using either gradient descent:
\begin{equation} \label{eq:GD}
y^{n+1}=y^n - \eta \nabla_y L(y,\lambda^{n+1})|_{y=y^{n}},
\end{equation}
or implicit gradient descent \cite{essid2019implicit}:
\begin{equation}\label{eq:firstOrd}
 y^{n+1}=y^n - \eta \left(I+\eta \nabla_{yy} L(y,\lambda^{n+1})|_{y=y^{n}}\right)^{-1} \nabla_y L(y,\lambda^n)|_{y=y^{n}}.
\end{equation}
These update rules couple the points $y_{i}\in \R^{d}$ in different ways, as discussed in the appendix. 
\item Check that the objective function decreases after the step,
\begin{equation}\label{eq:SecondOrd}
L(y^{n+1},\lambda^{n+1}) \leq L(y^{n},\lambda^{n+1}),
\end{equation}
where the kernel centers are evaluated at $y^{n+1}$ on both sides of the inequality (see the appendix).
Otherwise decrease $\eta$ to $\eta/2$ and repeat (c) and (d) until it does.
\end{enumerate}
\end{enumerate}

The intuition behind the criteria for updating $\lambda$ is the following. The two components of the objective function push the map $T(x, z)$ in opposite directions: while $L_F$ decreases as $\rho_T$ approaches the barycenter $\mu$, $L_C$ decreases as $\rho_T$ returns to $\rho$, as the cost is typically minimal at $T(x, z) = x$. The two components are also different in nature: $L_F$ represents the hard constraint that $y = T(x, z)$ be independent of $z$, while the minimization of $L_C$ establishes a selection criterion among all maps satisfying $L_F = 0$. Because of this, one should always pick $\lambda$ large enough that the direction of gradient descent of the full Lagrangian $L$ is also a direction of descent for $L_F$. This condition reads
$$ \left<\frac{\delta}{\delta T} \left[L_C + \lambda L_F\right], \frac{\delta}{\delta T} L_F\right> \ge 0, $$
a requirement that we make more precise by establishing a threshold $\alpha > 0$:
$$ \left<\frac{\delta}{\delta T} \left[L_C + \lambda L_F\right], \frac{\delta}{\delta T} L_F\right> \ge \alpha \left<\frac{\delta}{\delta T} L_F, \frac{\delta}{\delta T} L_F\right> , $$
which is the content of (\ref{eq::lambda}).
Notice that, if the optimum is reached for the prior $\lambda^{n}$, then the first order condition yields:
$$
\nabla_y c(x,y) + \lambda^n \nabla_y \sum_l \cKa(y_l,y)C_{il} = 0 \implies \nabla_y c(x,y) = -\lambda^n  \nabla_y \sum_l \cKa(y_l,y)C_{il},
$$
so \eqref{eq::lambda} yields $\lambda_{min} = \alpha - (-\lambda^n) = \alpha + \lambda^n$. This suggests setting adaptively $\alpha = \omega \lambda^n$, whith $\omega \in (0,1)$.

\section{Numerical examples}\label{sec:NumEx}

This section presents four representative numerical examples in different dimensions and with different types of covariates, to: (1) demonstrate the ability to work with cost functions different from the canonical $L^{2}$ and the effect that different choices for the cost have, and (2) show applicability to times series analysis with data distributed on a Riemannian manifold.

\subsection{ Barycenter of three ellipses under different costs}\label{subsec:Ellipses}

A toy example shows how the choice of a cost function affects the properties of the barycenter. The data points $x_i \in \mathbb{R}^2$ are sampled from $3$ uniform densities supported on 3 ellipses with different centers and shapes, and labelled by the discrete cofactor $z_i \in \{ 0, 1, 2\}$. The major axes of the ellipses are horizontal for $z=1,2$ (in green and yellow) and  vertical for $z=0$ (in blue). All ellipses have eccentricity $e = \frac{2\sqrt{2}}{3}$, with $100$ points sampled from each. 

We apply our algorithm with cost function induced by the $p$-norm:
$$
c (x , y) = \sum_{i=1}^2 |  x_i - y_i  |^p,\quad x,y\in \mathbb{R}^2,\quad p\in \mathbb{R}, \quad p\geq 1.
$$
The test functions in $x$ space are constructed by solving Problem \ref{prob:factor}, using as features $f_{j}$  polynomials up to the second degree, i.e. $y_1$, $y_2$, $y_1^2$, $y_1 y_2$ and $y_2^2$.  To address the issue that, for $p<2$, the Hessian of the cost function is degenerate at the origin, we utilize the approximation $|x| \approx \sqrt{x^2 + \epsilon } - \sqrt{\epsilon}$ ,with $\epsilon = 0.01$. The data $x_i$ and barycenter $y_i$ (in purple), with $p\in \{ 1.2, 1.5, 2, 2.5, 3\}$ are shown in Figure \ref{fig:ellipses}. 

\begin{figure}[h!bp]
\hspace{-1.6cm}
\includegraphics[width=1.2\textwidth]{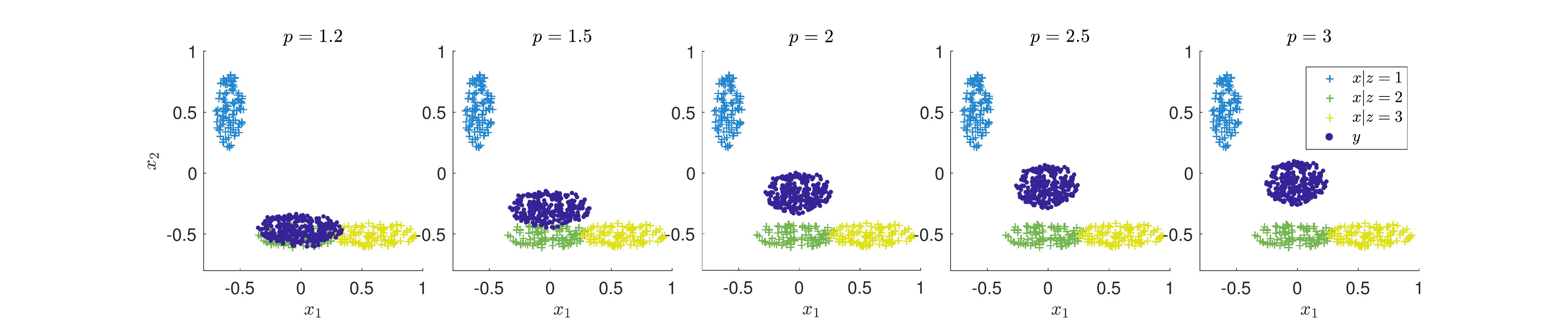}
\caption{Barycenters with $p$-norm-based costs with different values of $p$. The color in the source data refers to  the index $z$ of the cluster, while the barycenter is displayed in purple. 
 }
\label{fig:ellipses}
\end{figure}

Large values of $p$ penalize outliers, i.e. distributions that are far from the barycenter. Thus the barycenter for $p$ large must be such that no distribution is far from it. On the other hand, for $p$ close to $1$, majority rules: the average distance to the barycenter must be minimal. In our case, the ``outlier'' , both in shape and position, is the cluster $z=0$ in blue. Thus in Figure \ref{fig:ellipses}, when $p$ is small, the  barycenter (in deep purple) is closer in shape and position to the ellipses with $z=1, 2$, while, as $p$ increases, the barycenter shifts gradually from the bottom to the middle of the figure and becomes nearly isotropic, so as not to be far from any cluster, including the outlier.

\subsection{ Handwritten digits  }\label{subsec:Digits}

We use the MNIST dataset \cite{lecun-98} to display the effect of the test functions chosen for Problem 2, contrast this with the non-parametric Problem 1, and illustrate how a non-pairwise cost function can help impose desired features on the barycenter. The MNIST dataset contains handwritten digits from $0$ to $9$. For each digit, we randomly select $6$ images, which we randomly displace, and then compute their barycenter under various test functions and costs.

\subsubsection{Effect of test function }
We first demonstrate the effect of the richness of the test functions adopted, keeping as cost function the standard squared Euclidean distance. Two different sets of test functions are used: first order polynomials, which only detect the discrepancy in the conditional means, and polynomials up to $2$nd order, testing both conditional mean and covariance.
We then compare these results to the nonparametric algorithm (\ref{eq:DDprob1}), after preconditioning by subtracting the conditional mean. The results are displayed in Table \ref{table:MNIST}. Qualitative improvements can be observed when the test function becomes richer. For example, the barycentric images are noisy when only the conditional mean is aligned, the edges are clearer when second order polynomials are adopted, and the nonparametric approach outperforms both choices.

\begin{table}[h!]
\centering
\setlength{\tabcolsep}{0.5pt}
\begin{tabular}{ccccccccccc}\\ \hline
& 0 & 1 & 2 & 3 & 4 & 5 & 6 & 7 & 8 & 9 \\ 
\rotatebox{90}{}&
\includegraphics[width=0.1\textwidth]{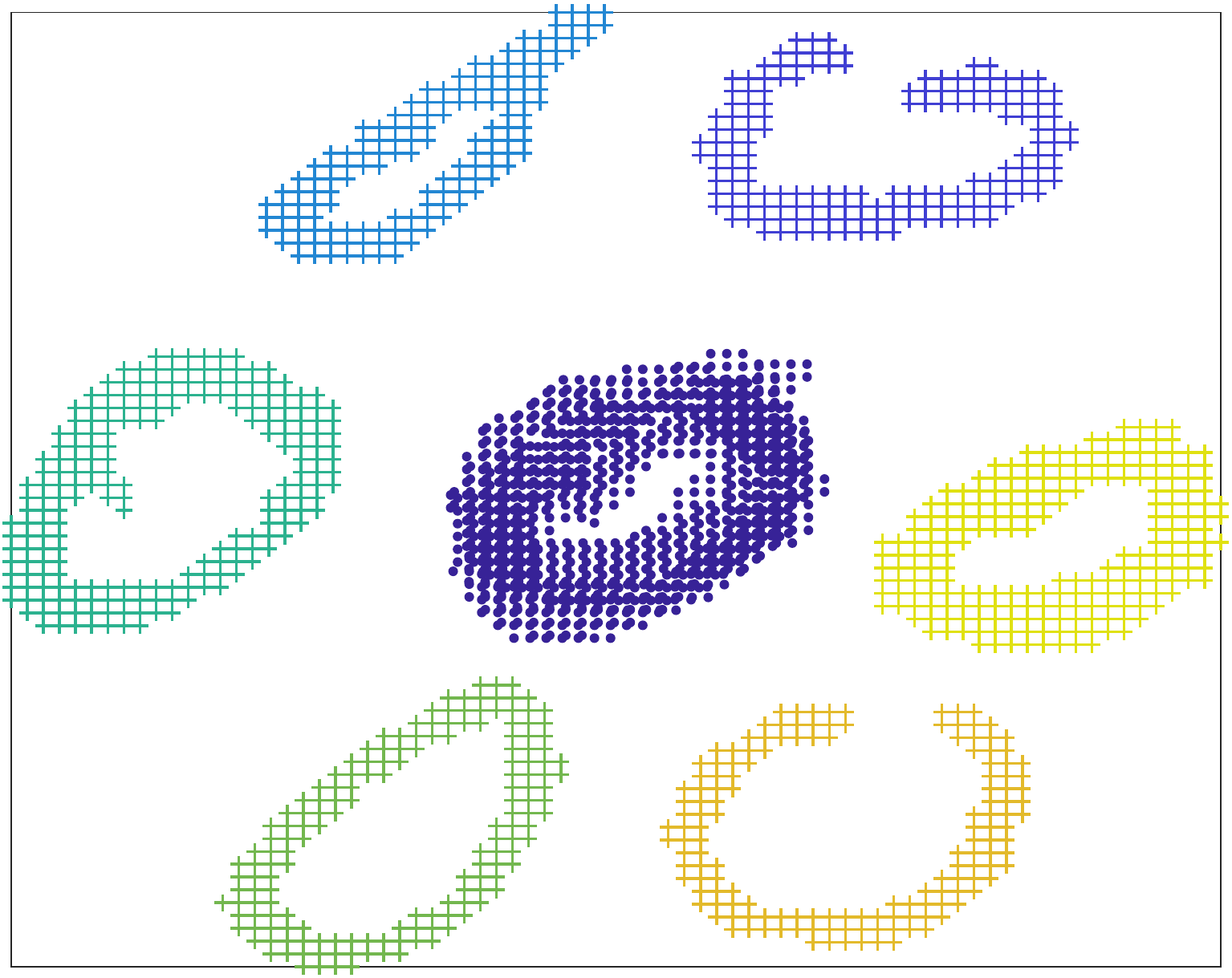} &
\includegraphics[width=0.1\textwidth]{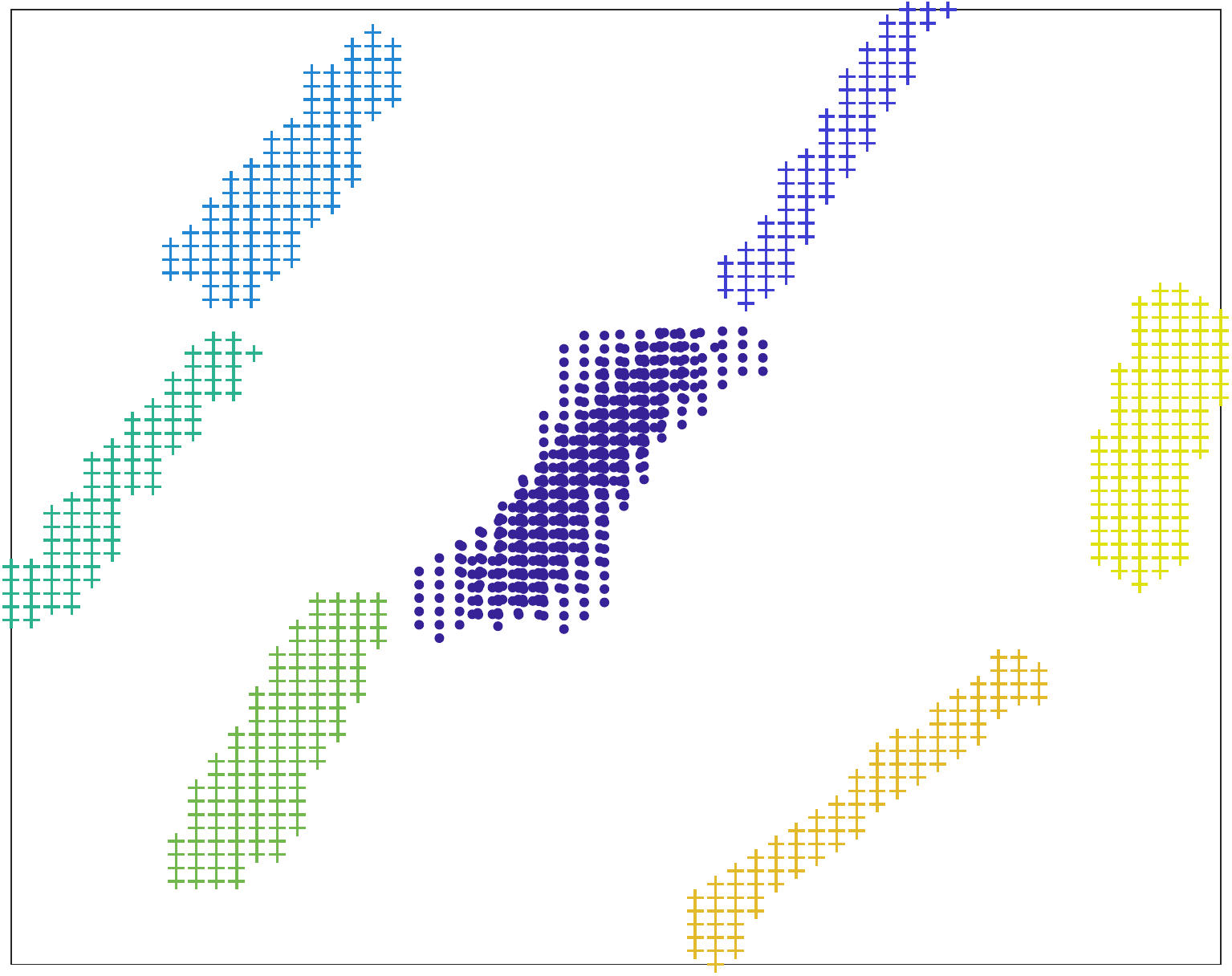} &
\includegraphics[width=0.1\textwidth]{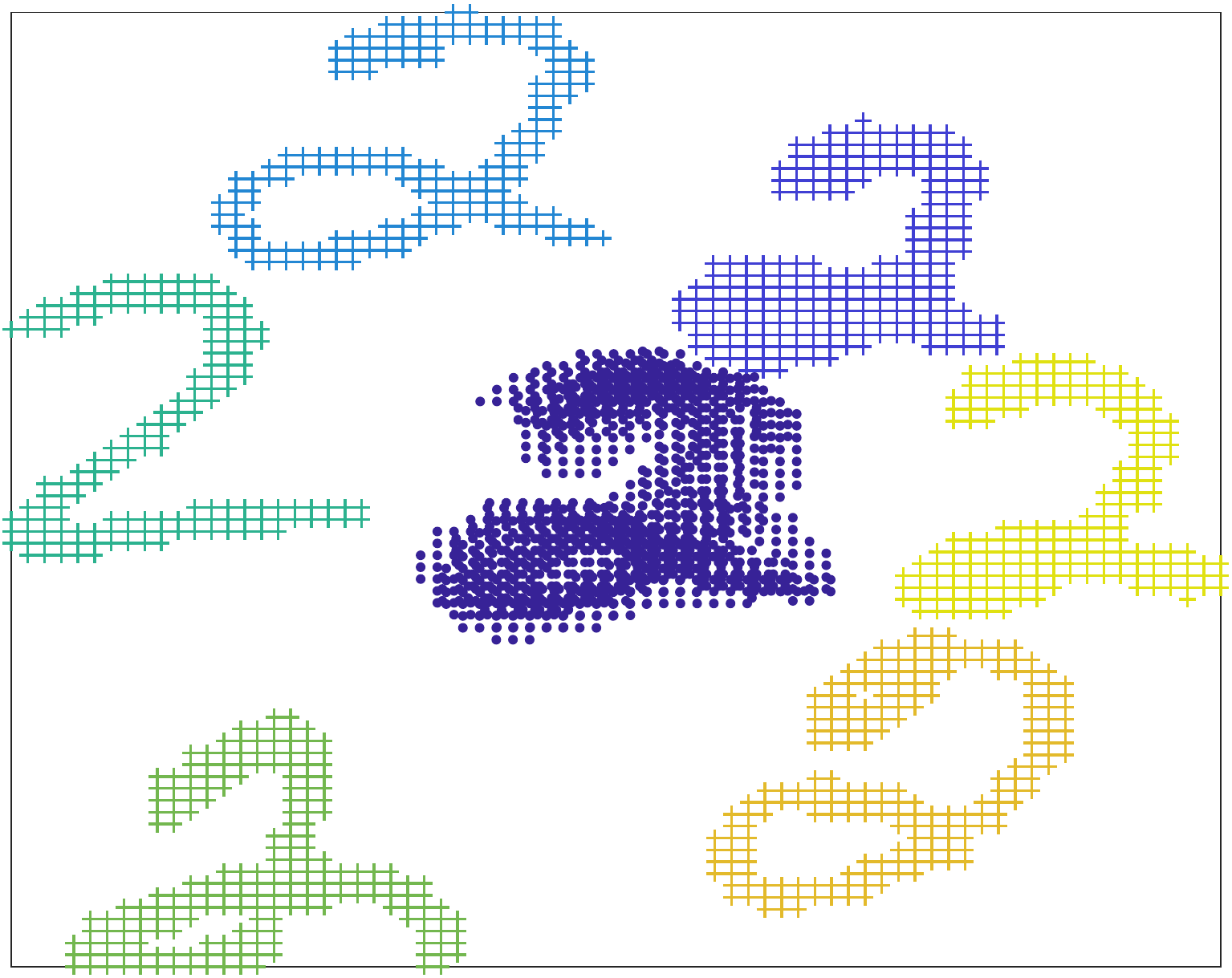} &
\includegraphics[width=0.1\textwidth]{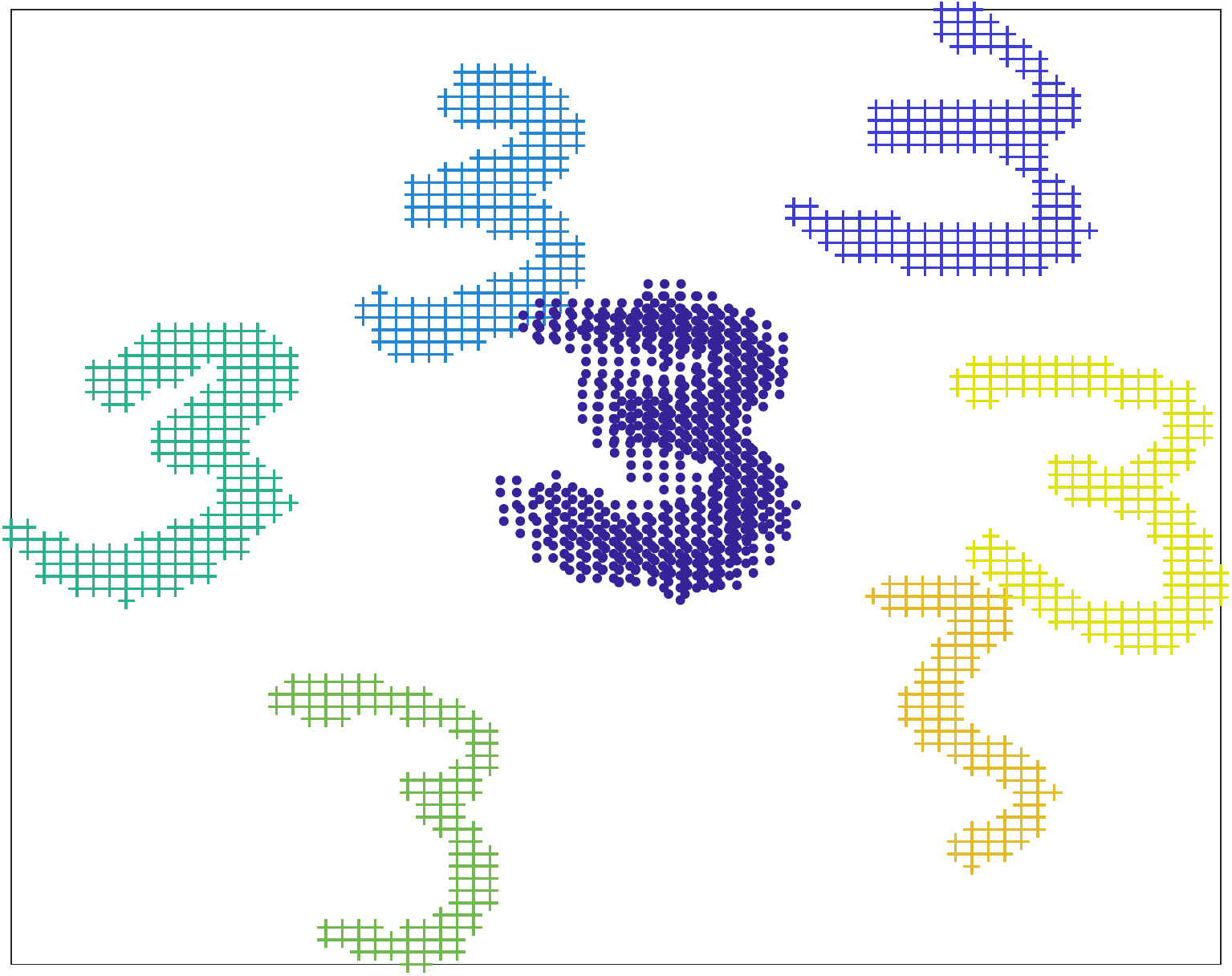} &
\includegraphics[width=0.1\textwidth]{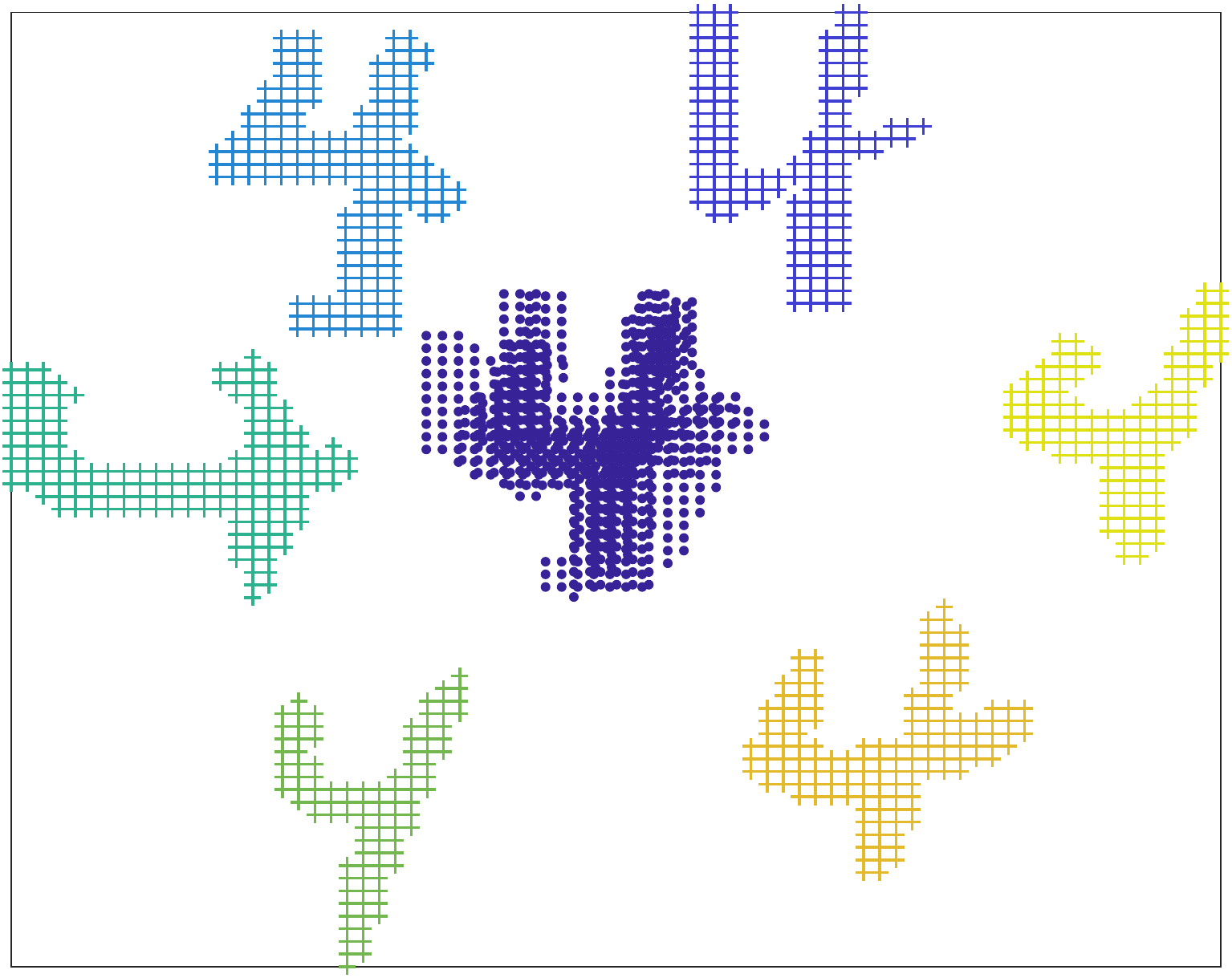} &
\includegraphics[width=0.1\textwidth]{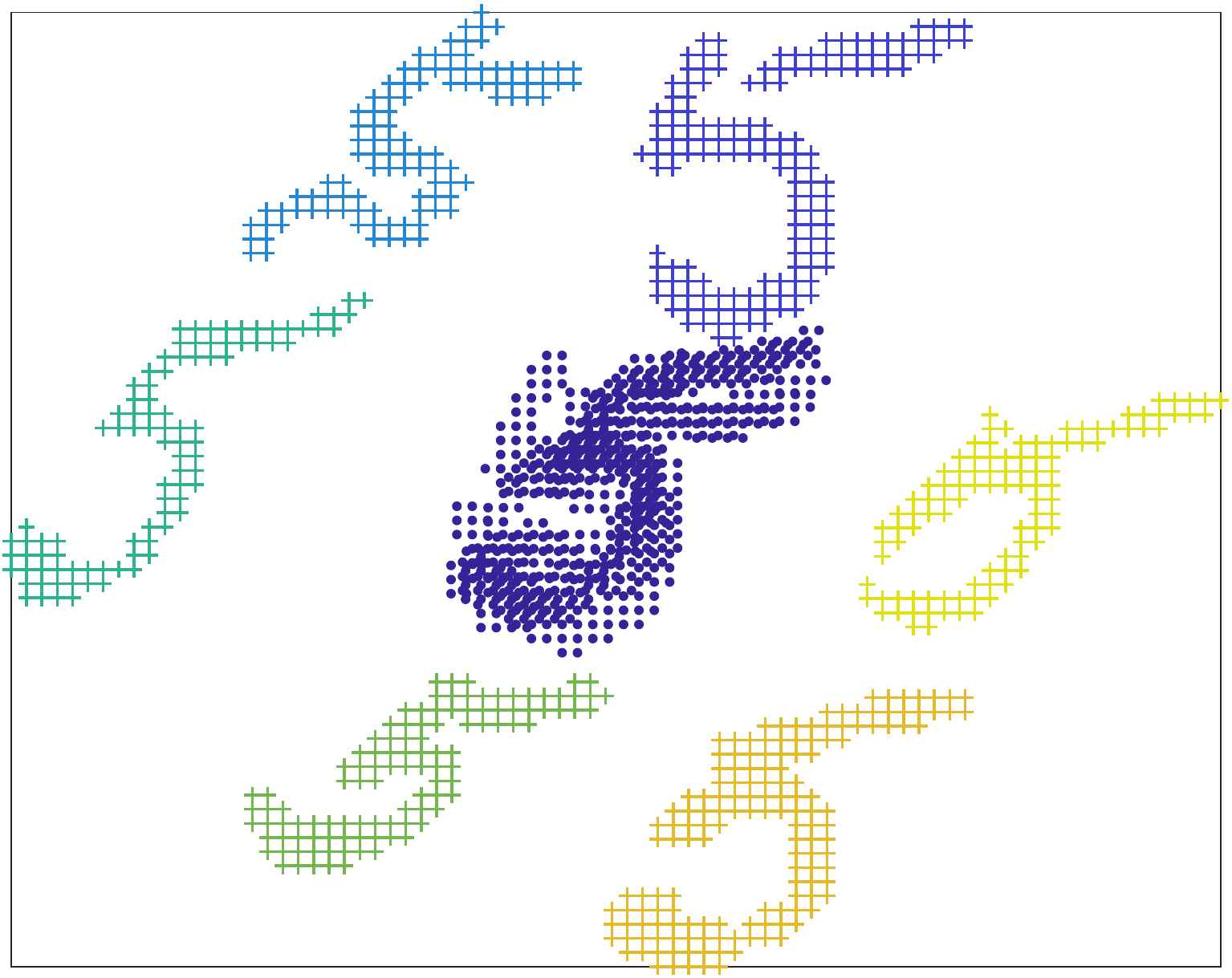} &
\includegraphics[width=0.1\textwidth]{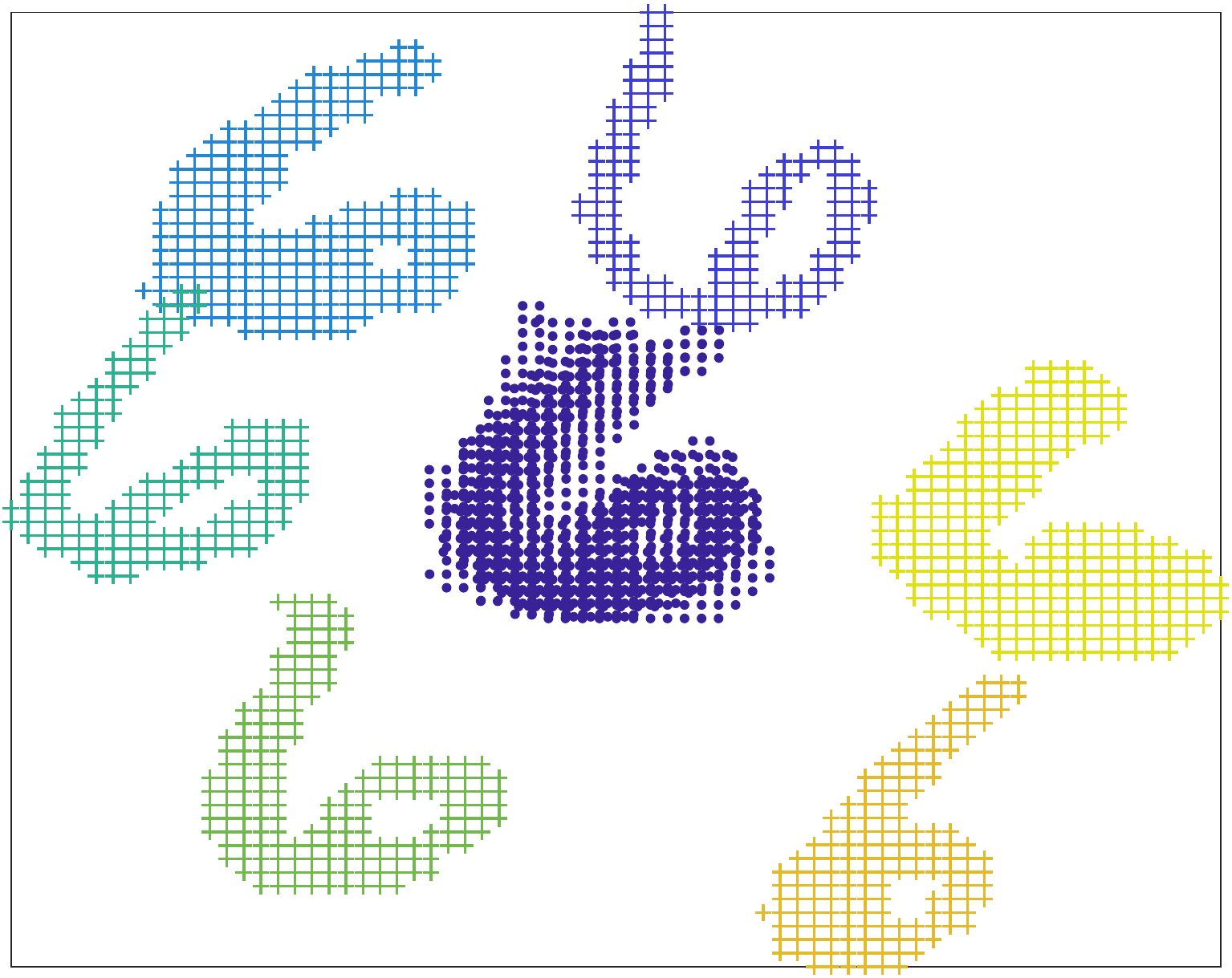} &
\includegraphics[width=0.1\textwidth]{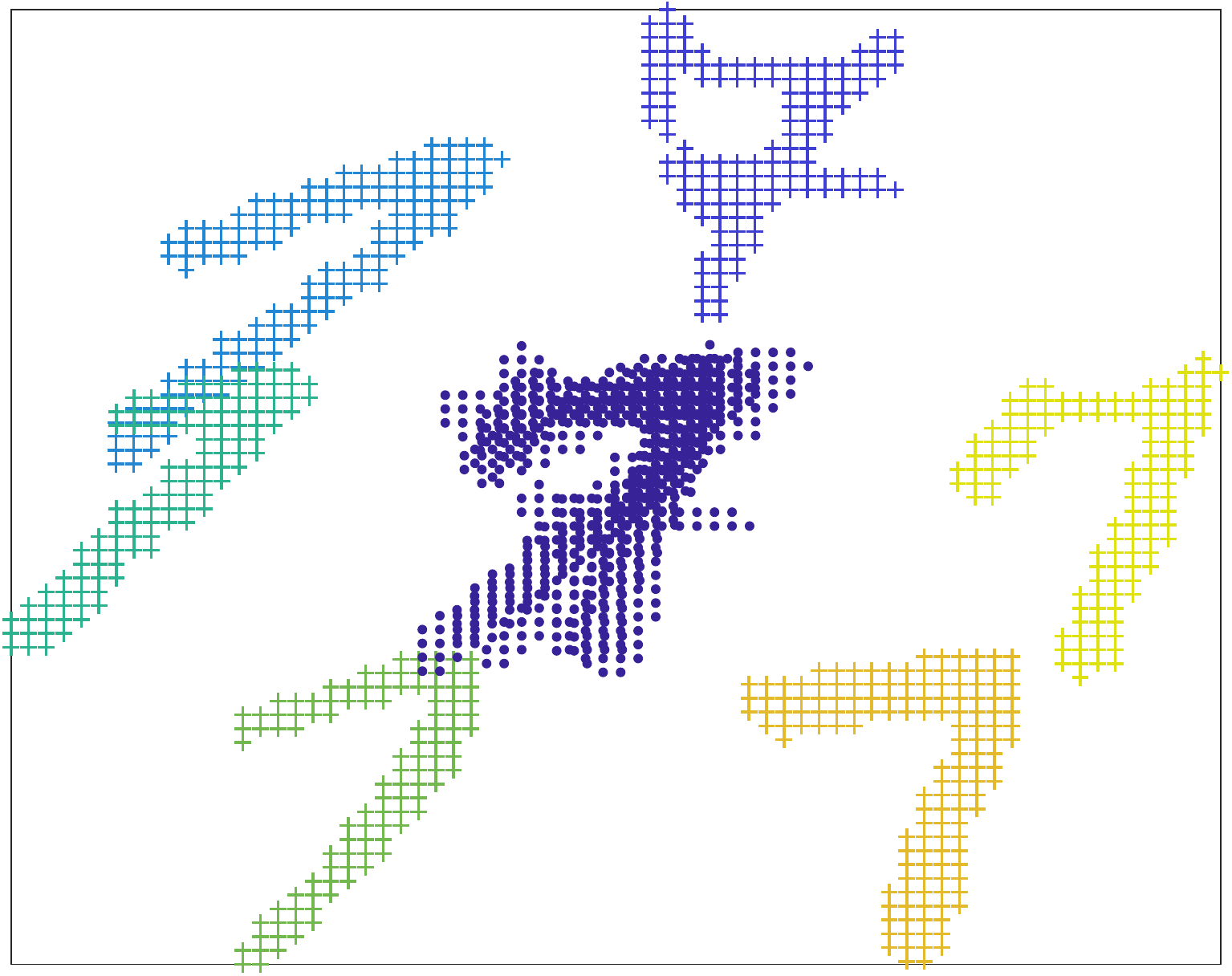} &
\includegraphics[width=0.1\textwidth]{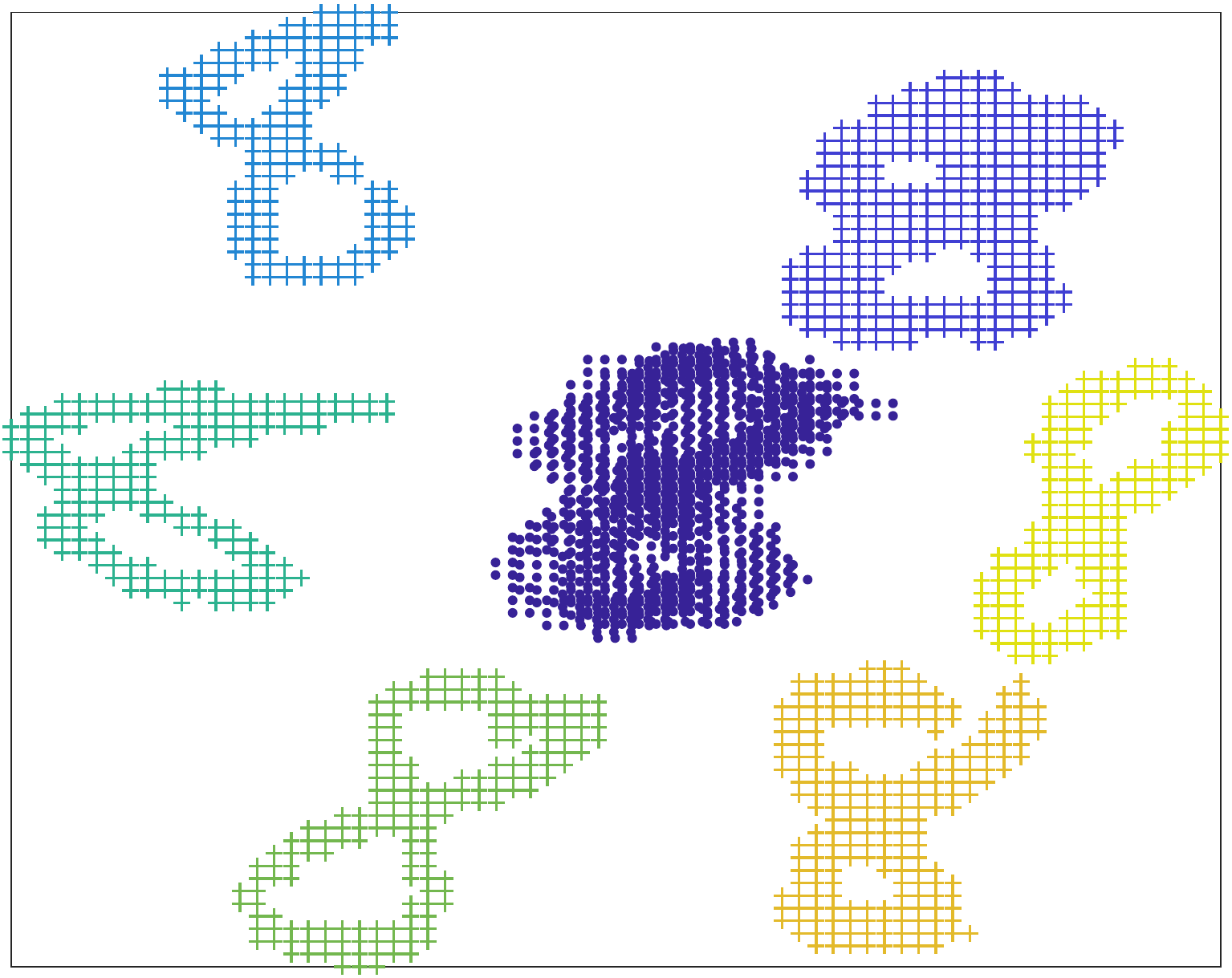} &
\includegraphics[width=0.1\textwidth]{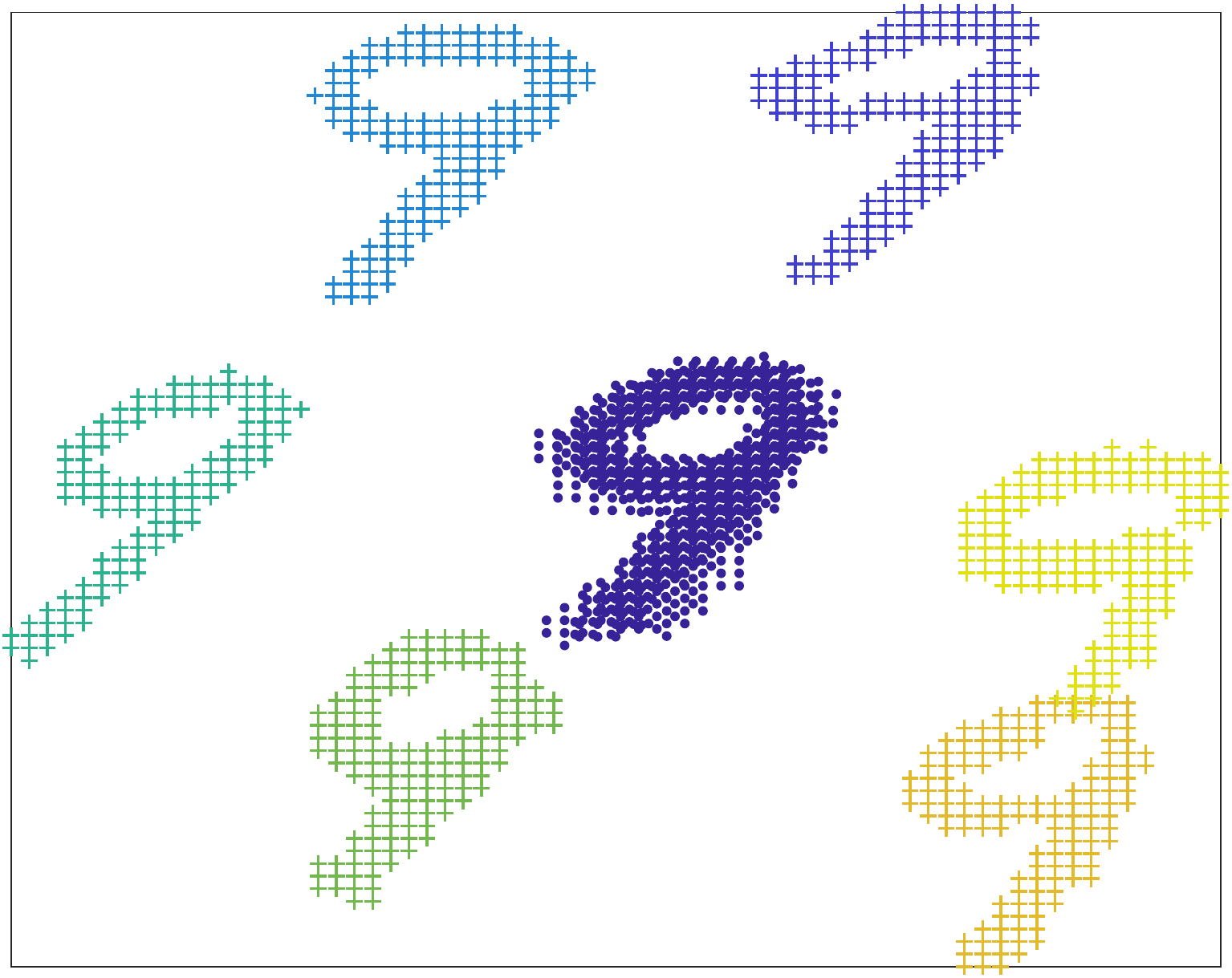}\\\hline
\rotatebox{90}{}&
\includegraphics[width=0.1\textwidth]{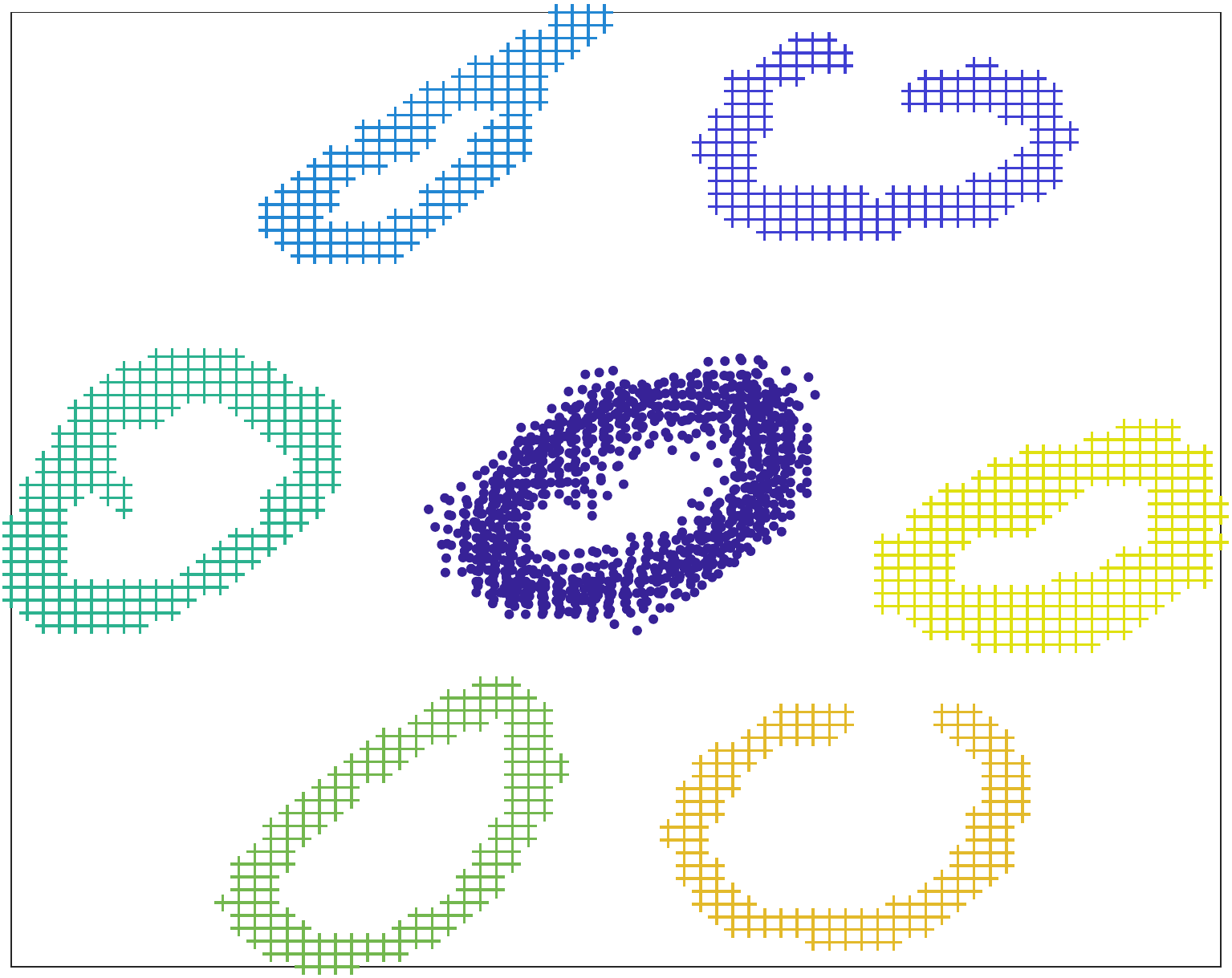} &
\includegraphics[width=0.1\textwidth]{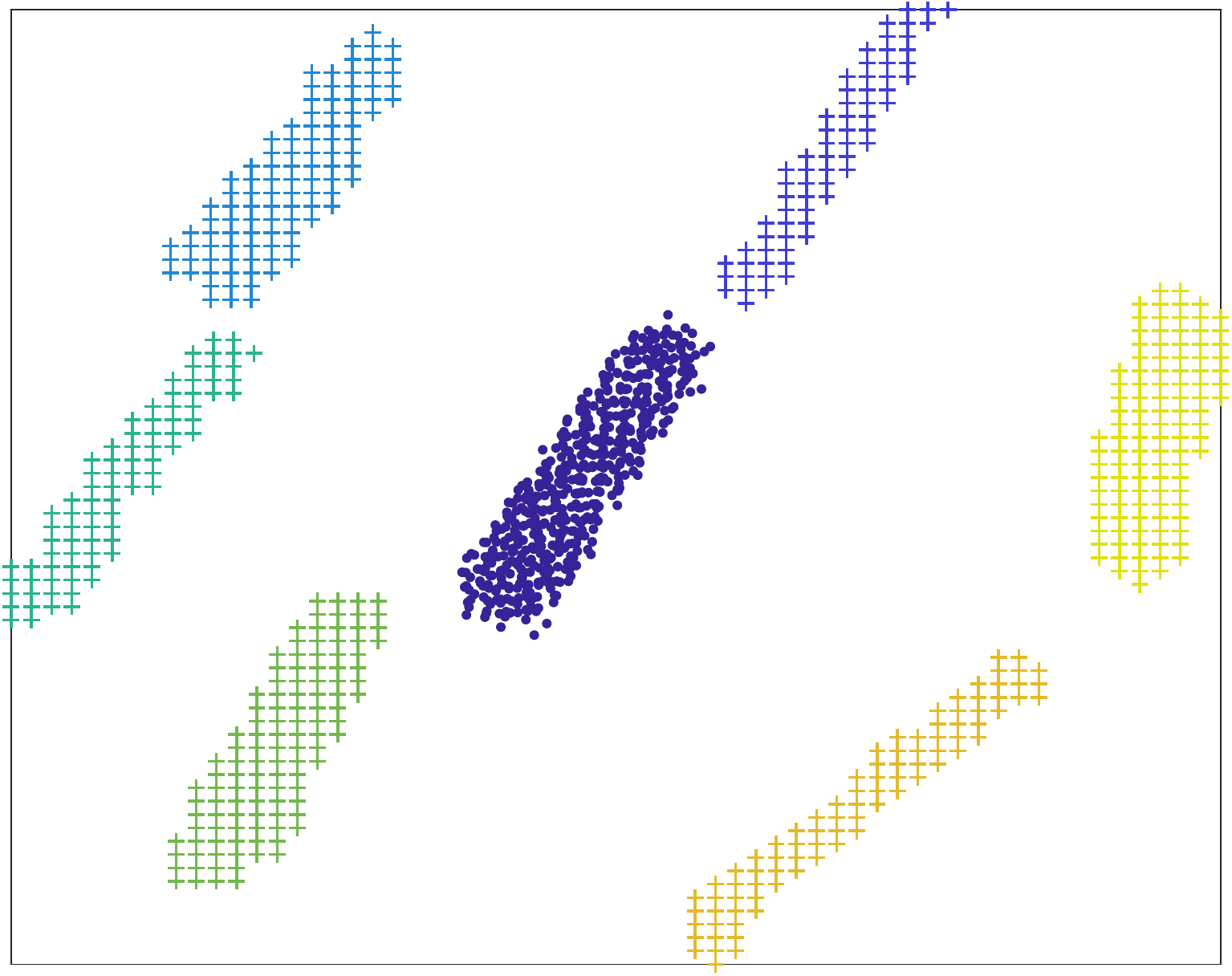} &
\includegraphics[width=0.1\textwidth]{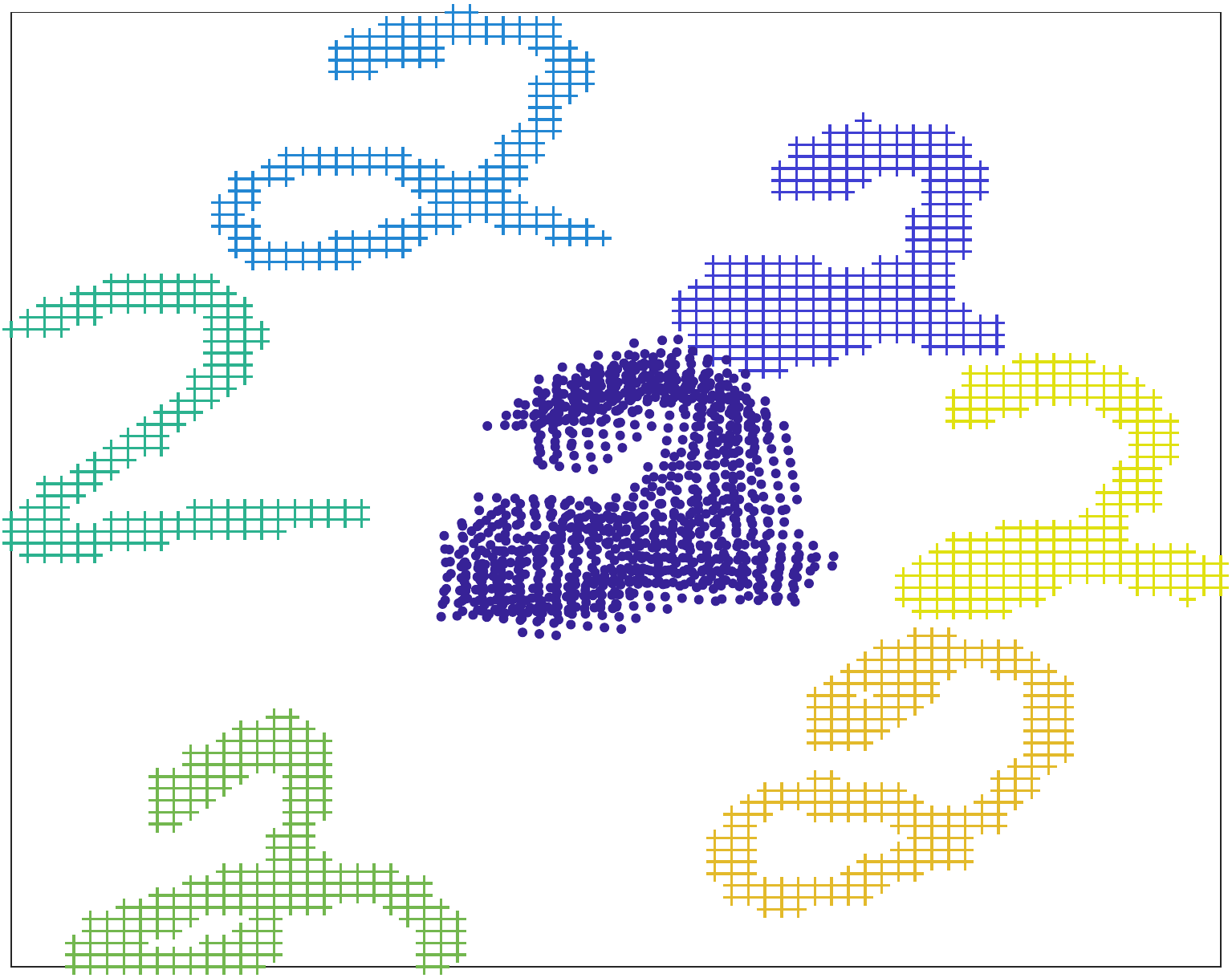} &
\includegraphics[width=0.1\textwidth]{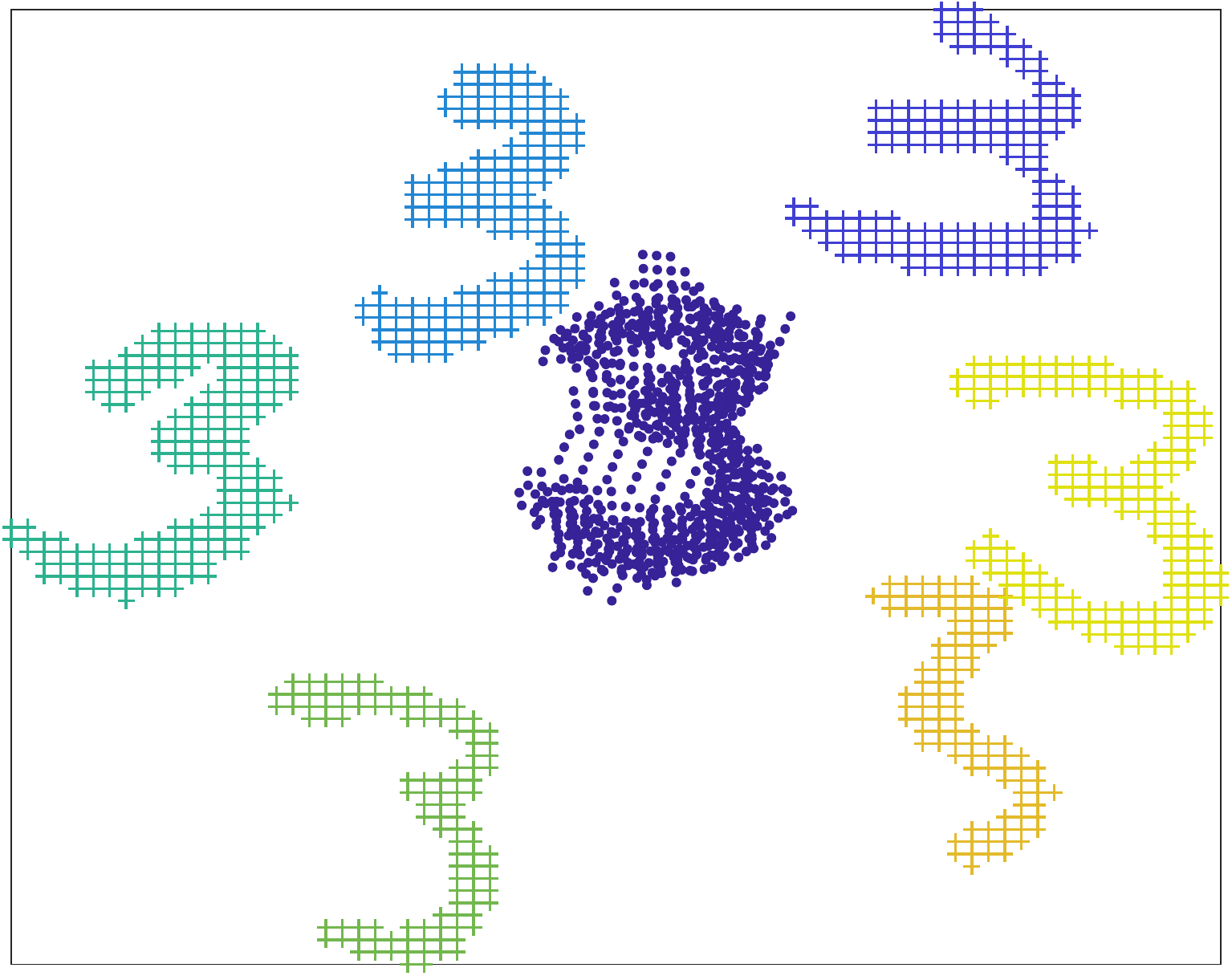} &
\includegraphics[width=0.1\textwidth]{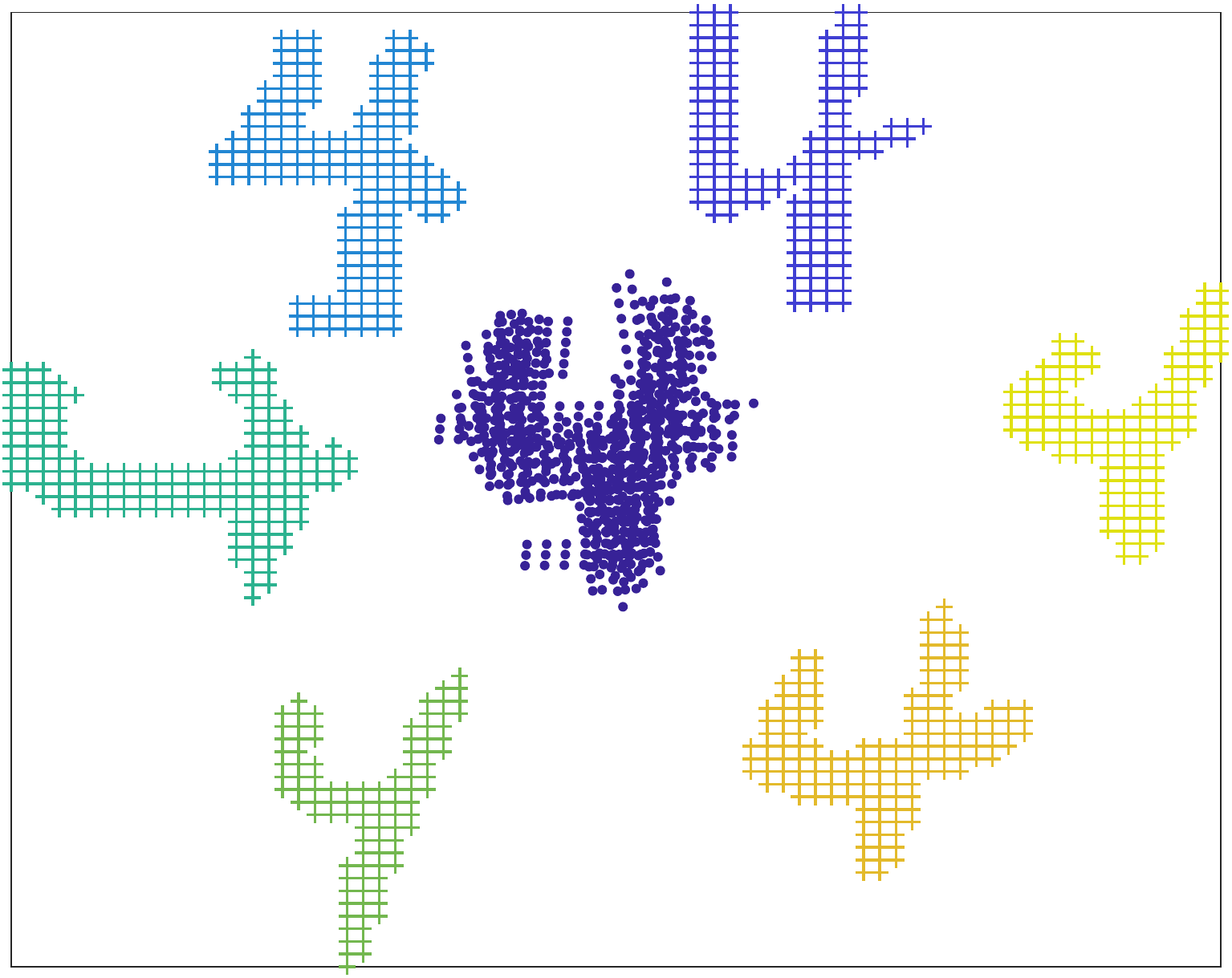} &
\includegraphics[width=0.1\textwidth]{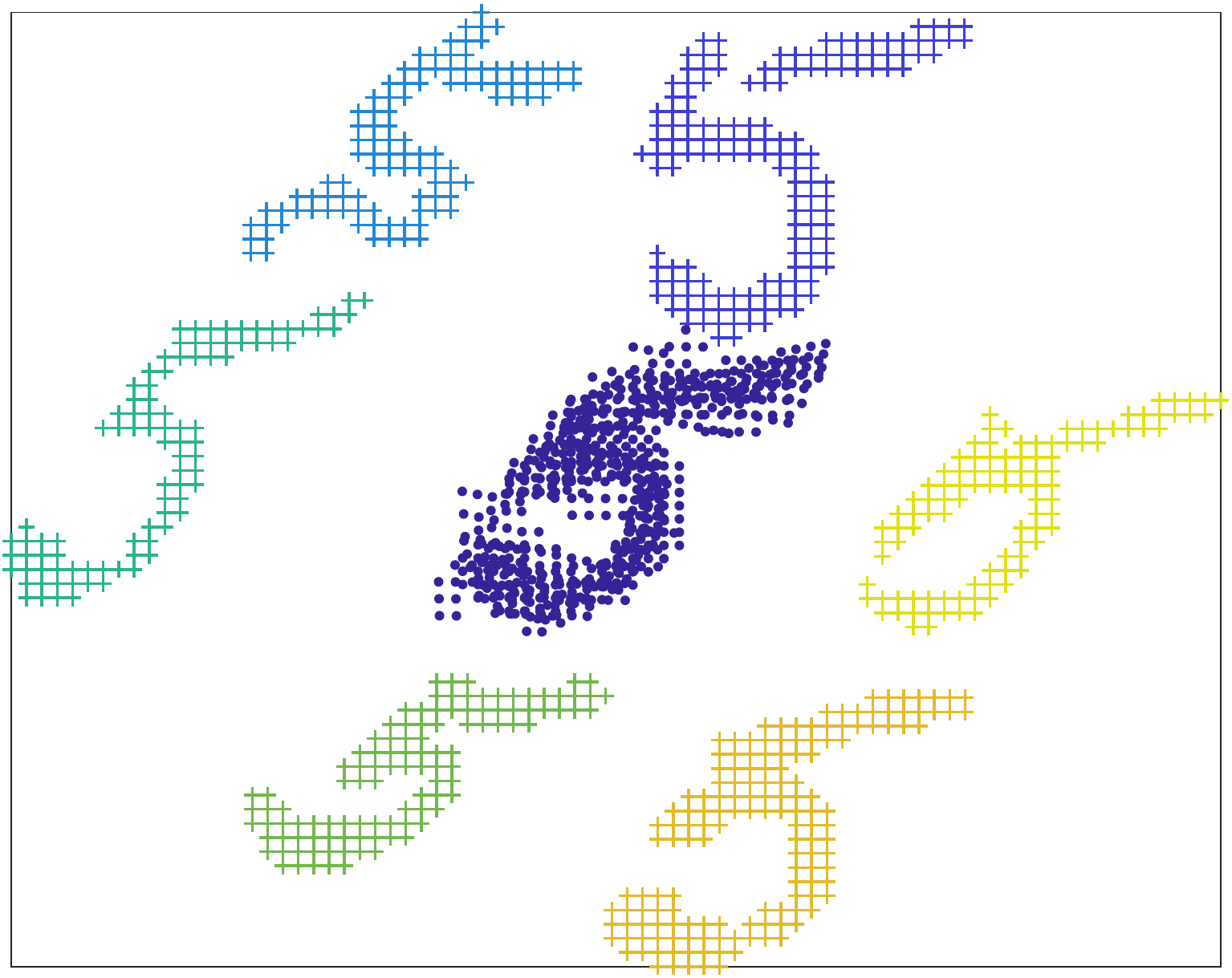} &
\includegraphics[width=0.1\textwidth]{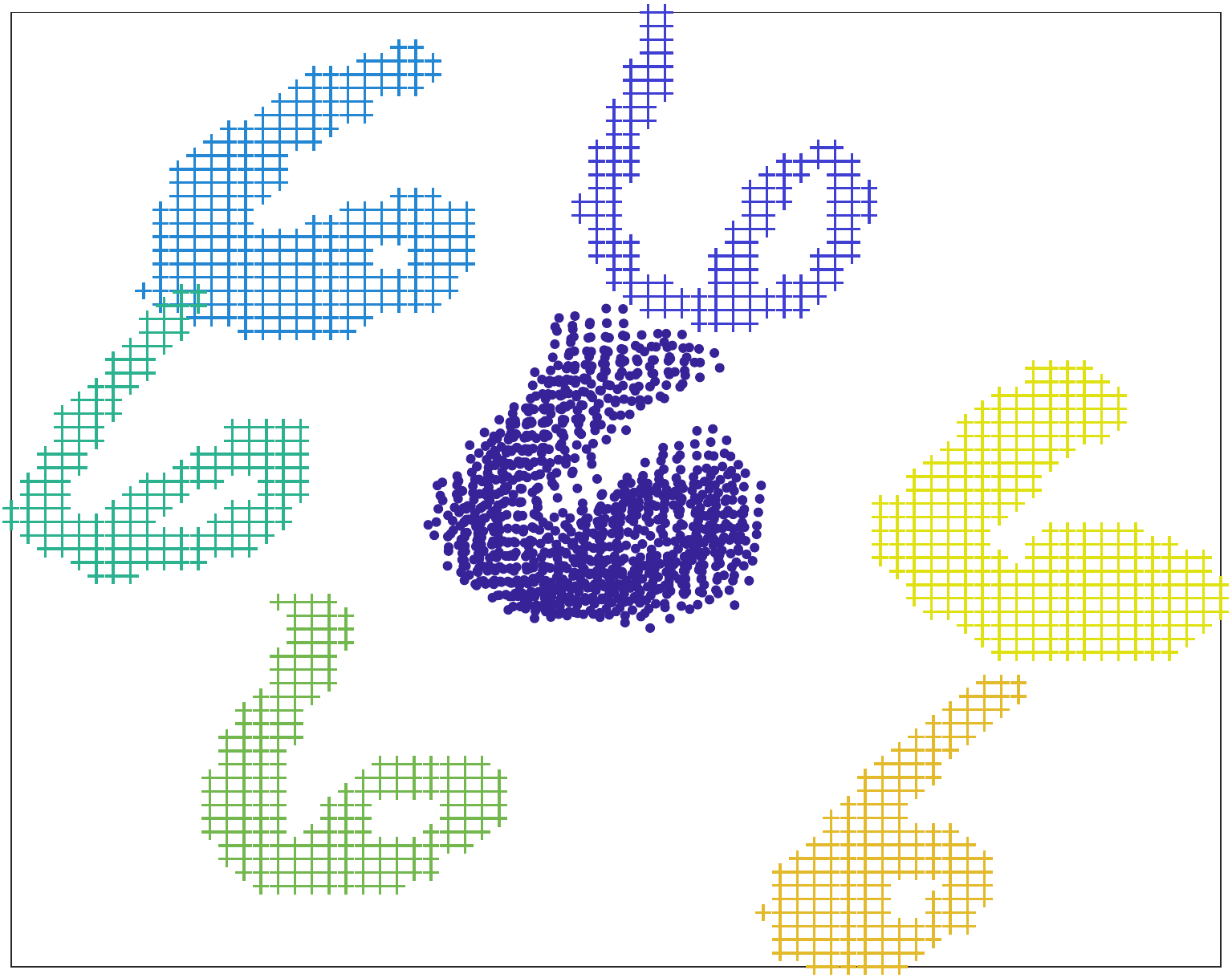} &
\includegraphics[width=0.1\textwidth]{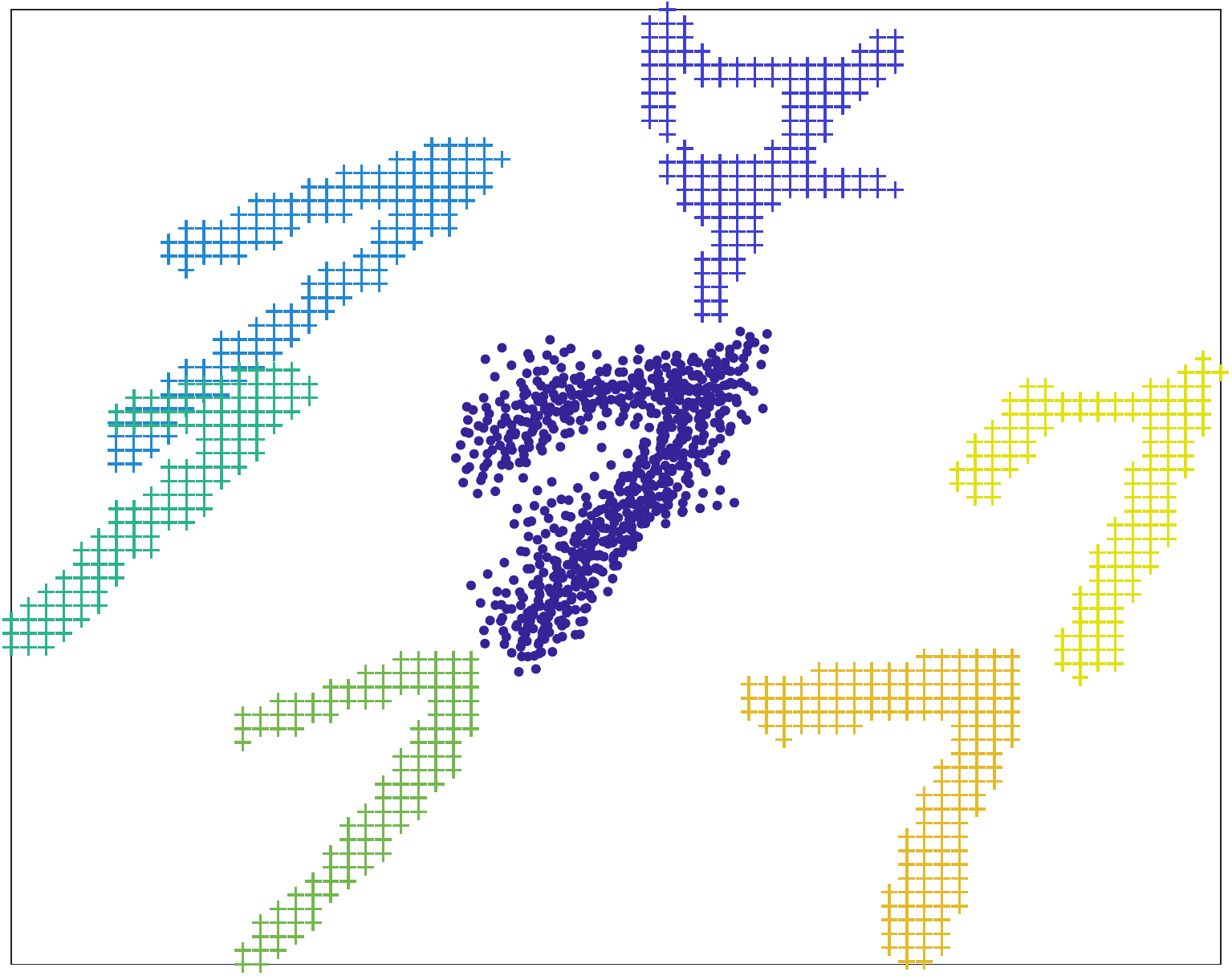} &
\includegraphics[width=0.1\textwidth]{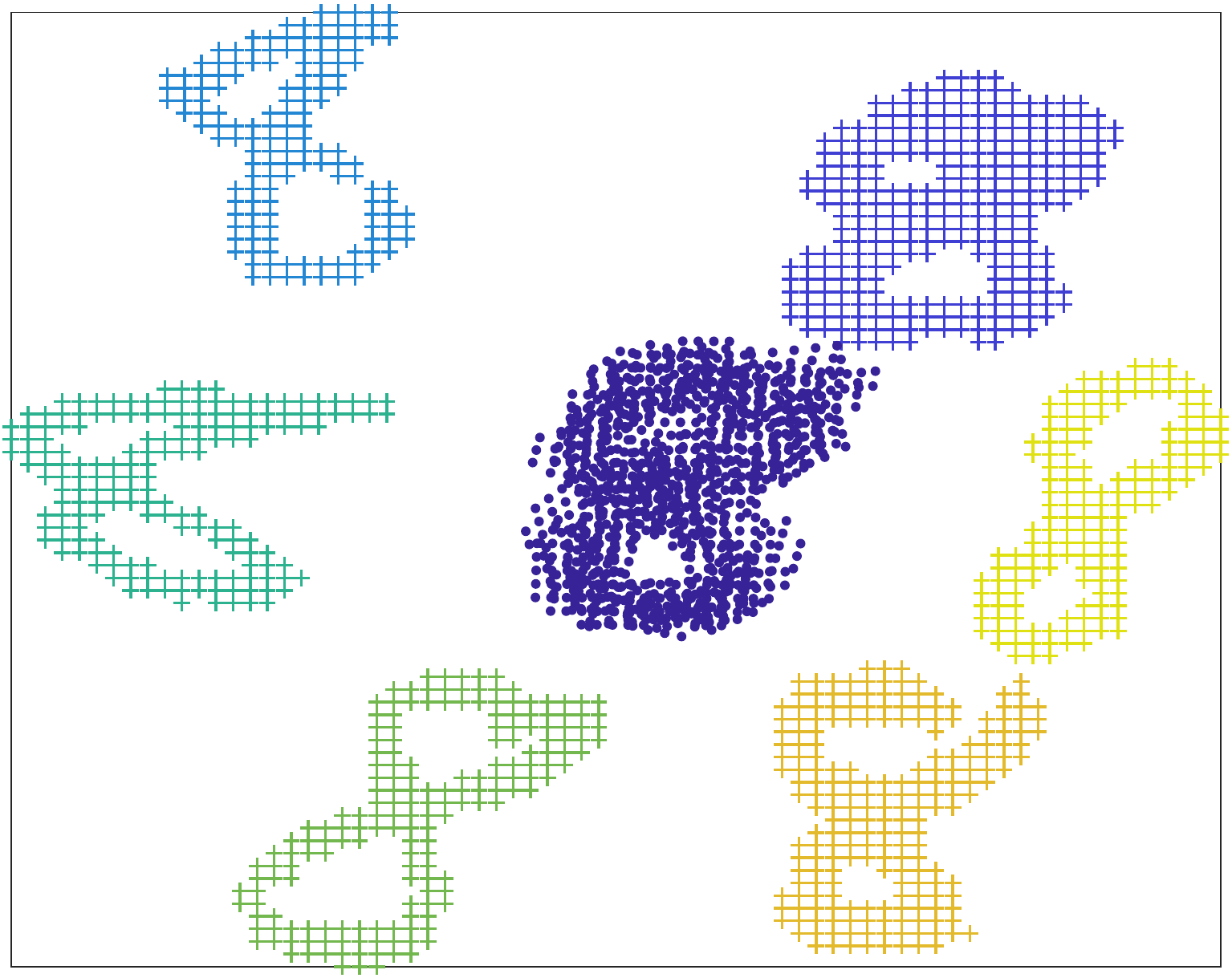} &
\includegraphics[width=0.1\textwidth]{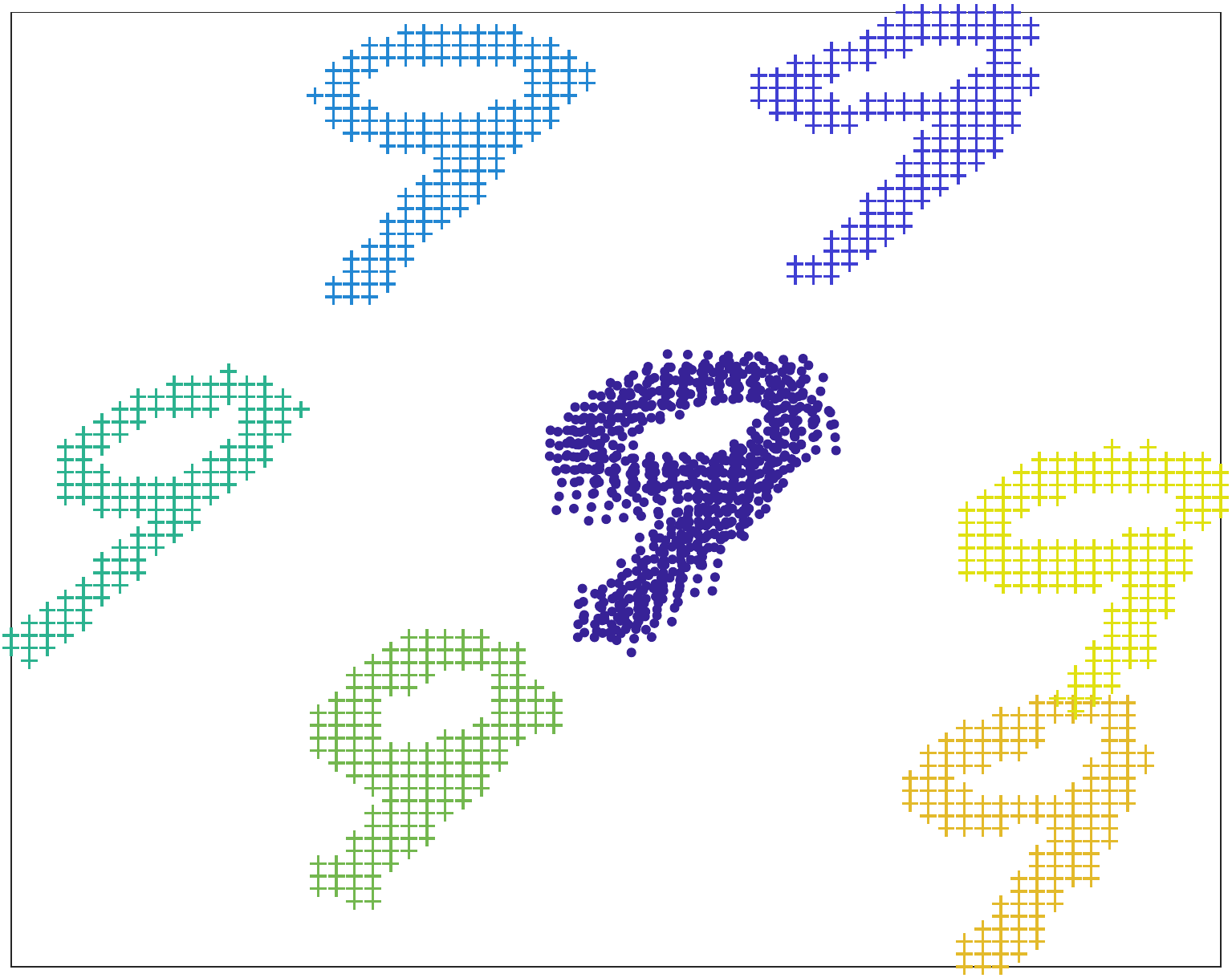}\\
\rotatebox{90}{}&
\includegraphics[width=0.1\textwidth]{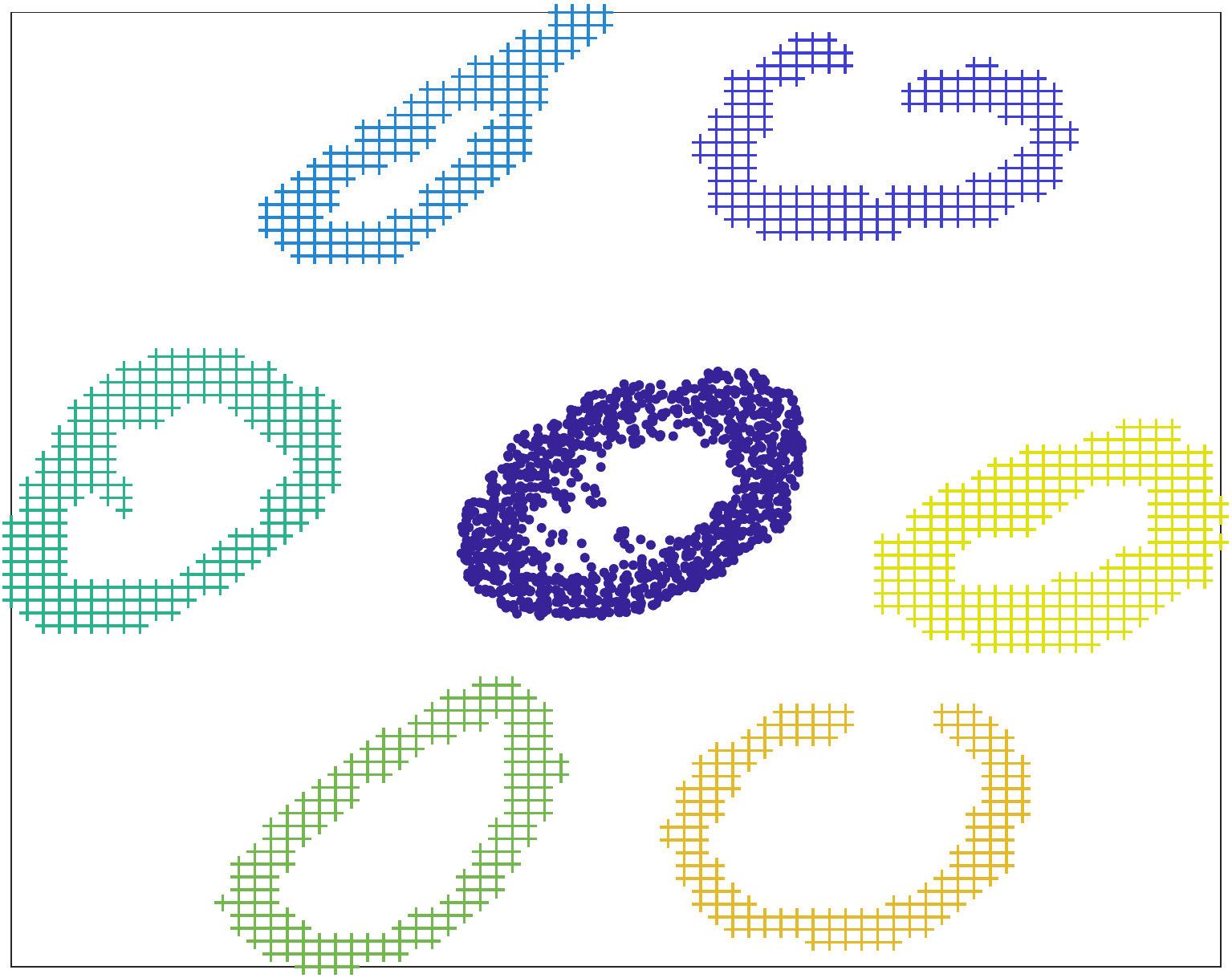} &
\includegraphics[width=0.1\textwidth]{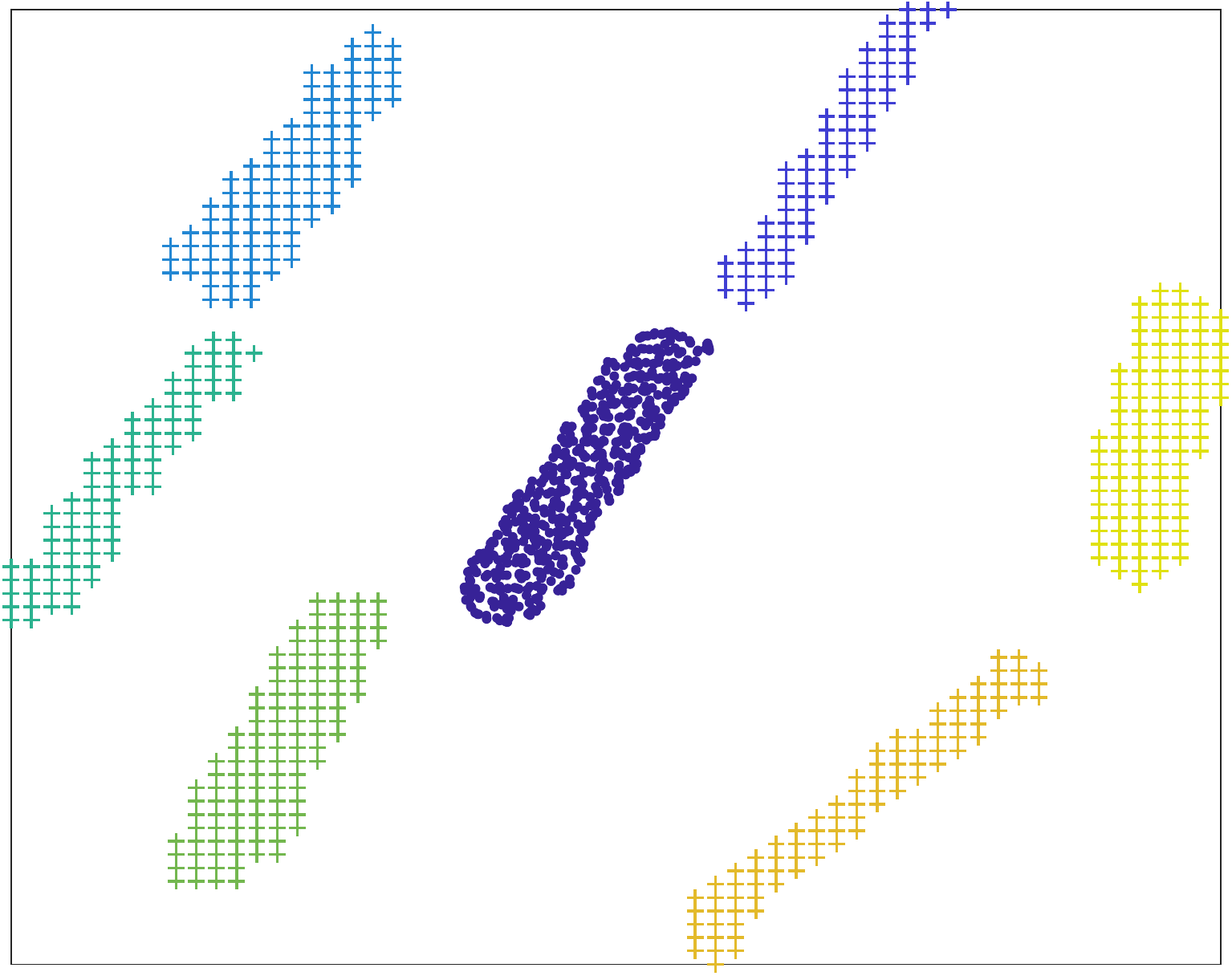} &
\includegraphics[width=0.1\textwidth]{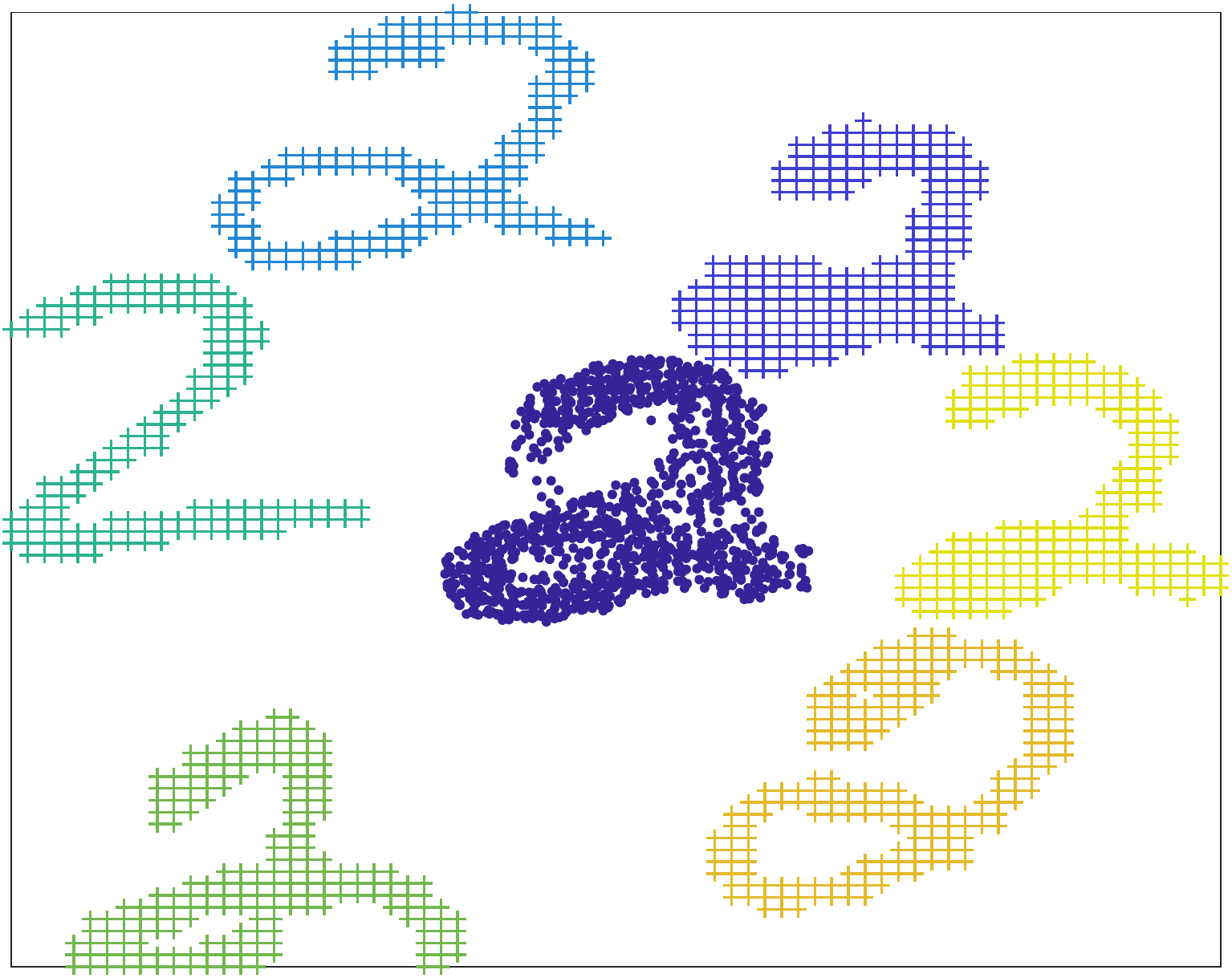} &
\includegraphics[width=0.1\textwidth]{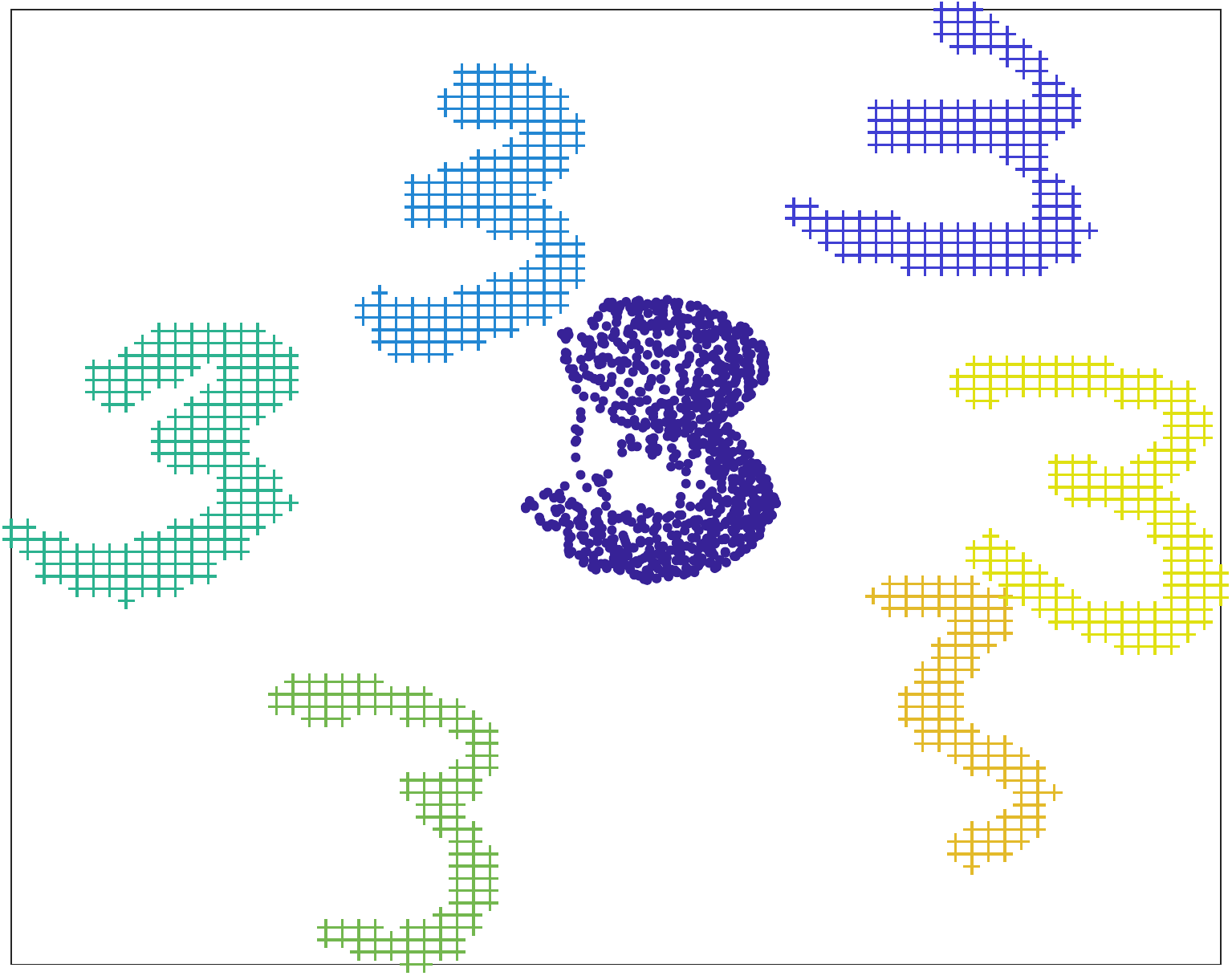} &
\includegraphics[width=0.1\textwidth]{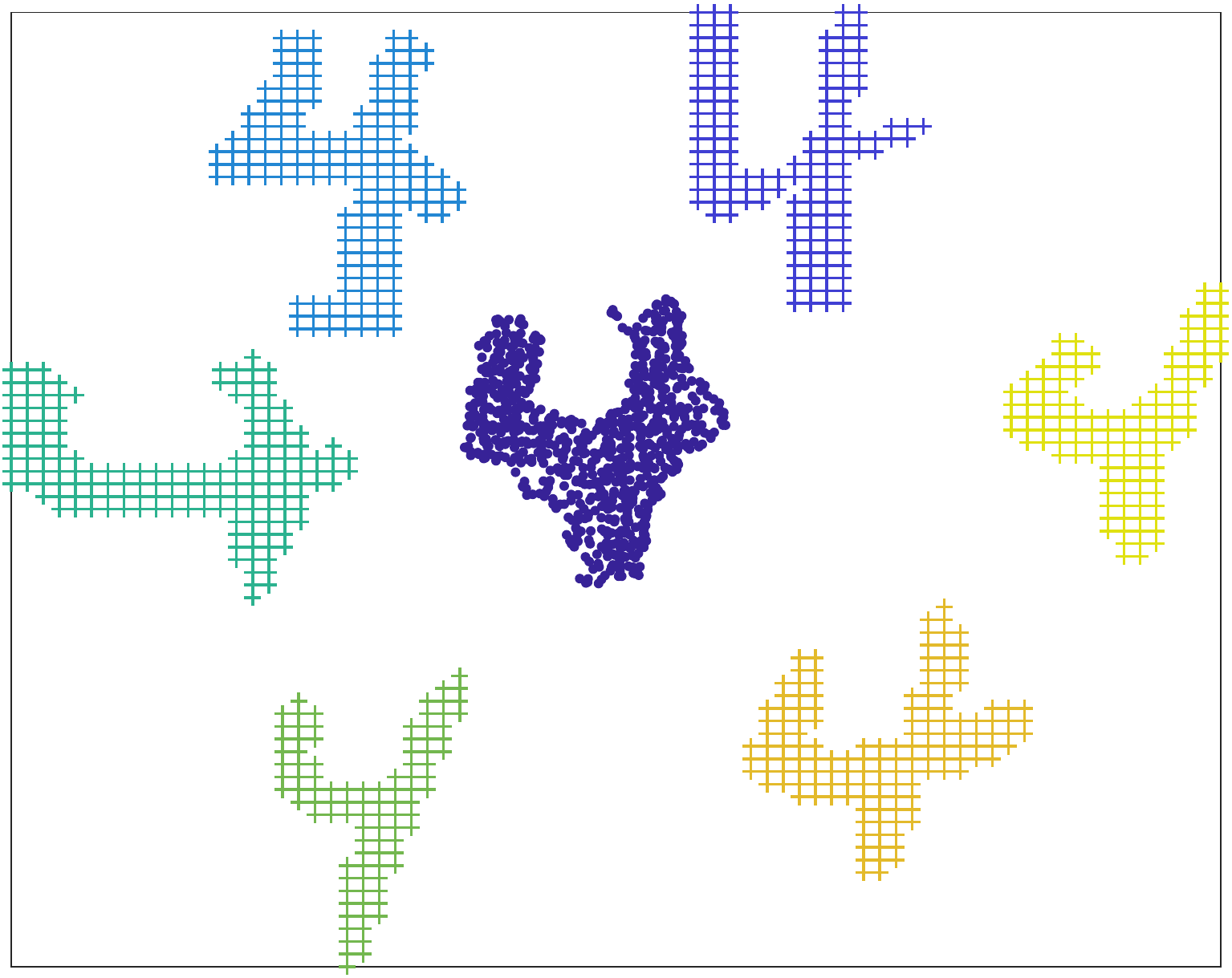} &
\includegraphics[width=0.1\textwidth]{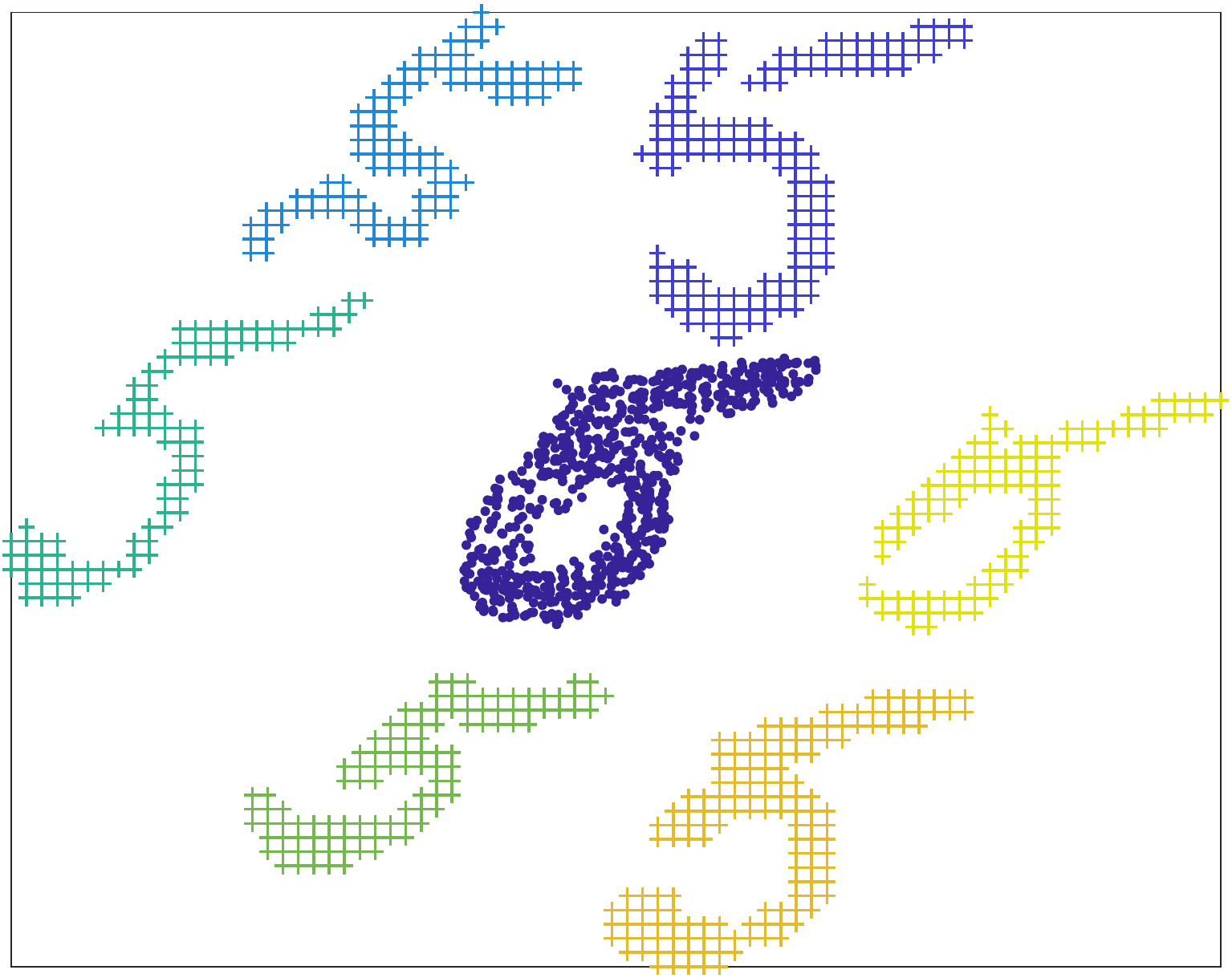} &
\includegraphics[width=0.1\textwidth]{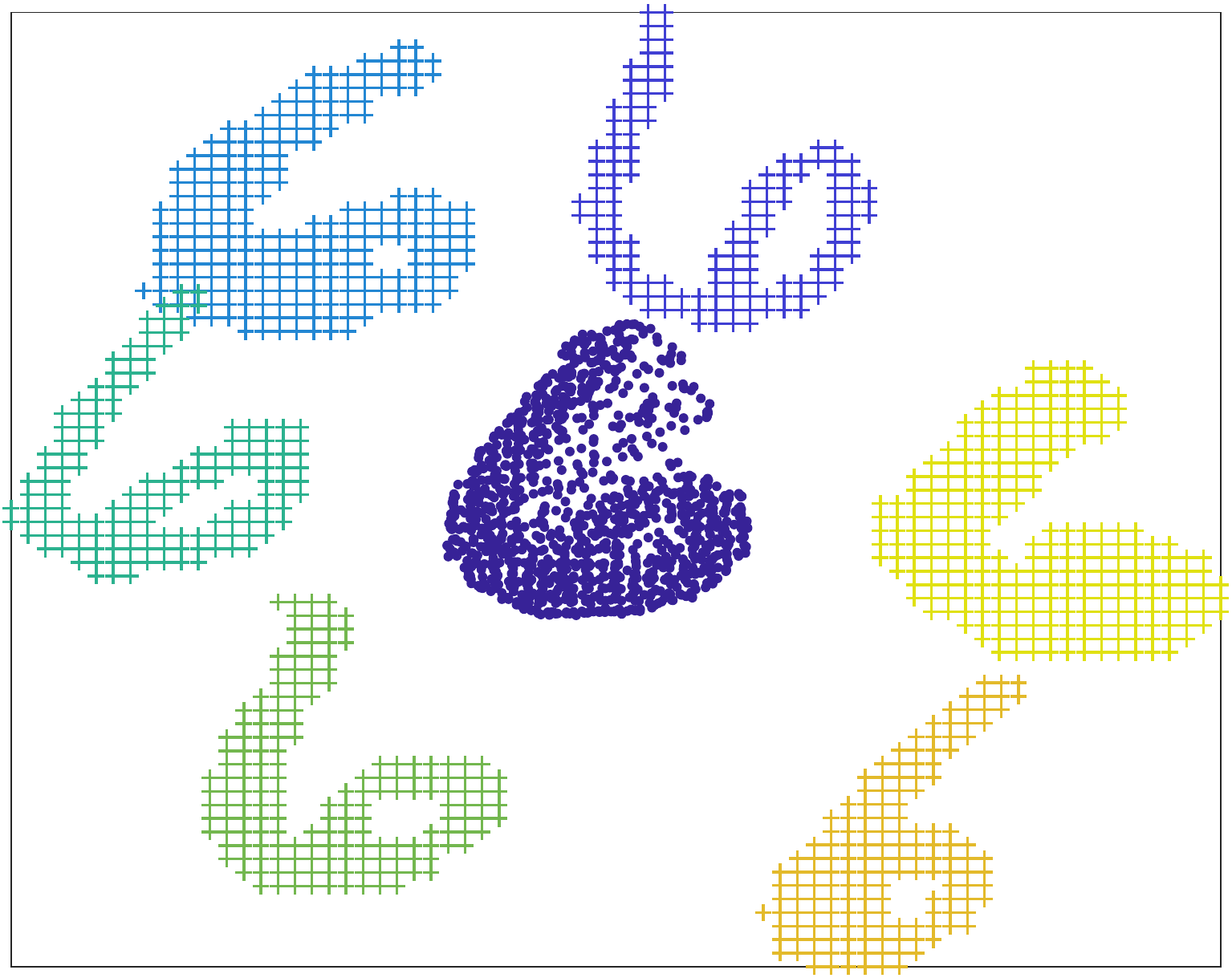} &
\includegraphics[width=0.1\textwidth]{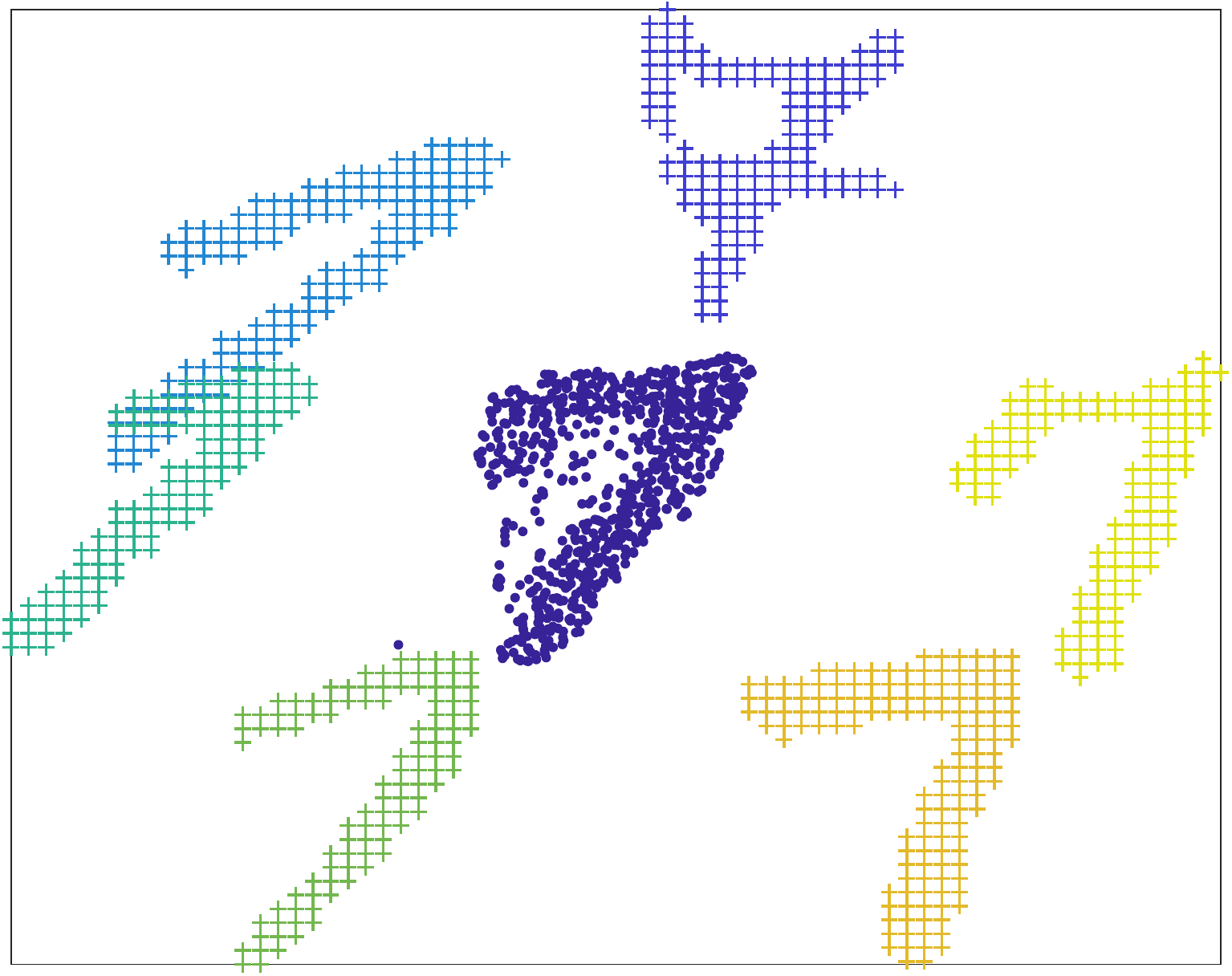} &
\includegraphics[width=0.1\textwidth]{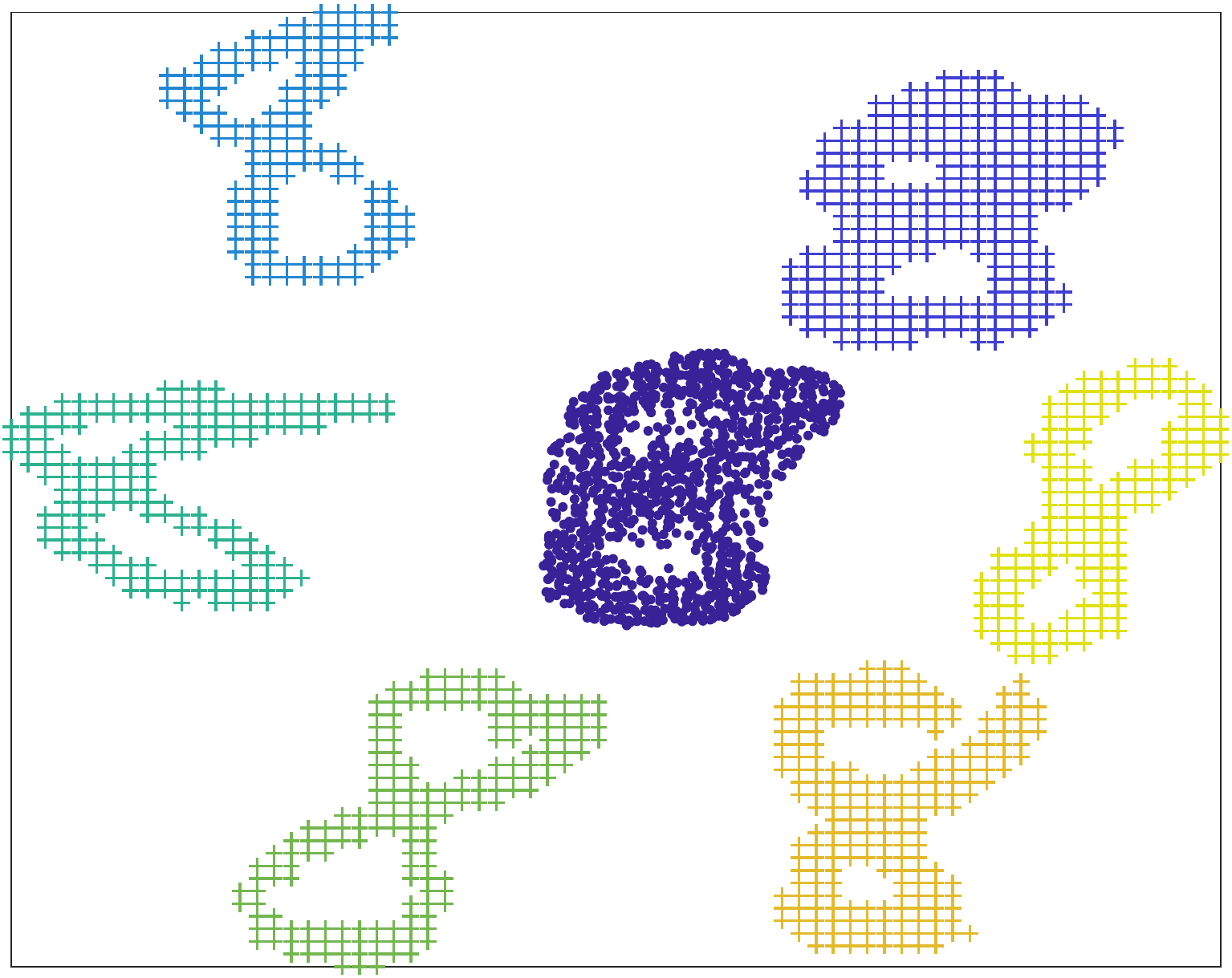} &
\includegraphics[width=0.1\textwidth]{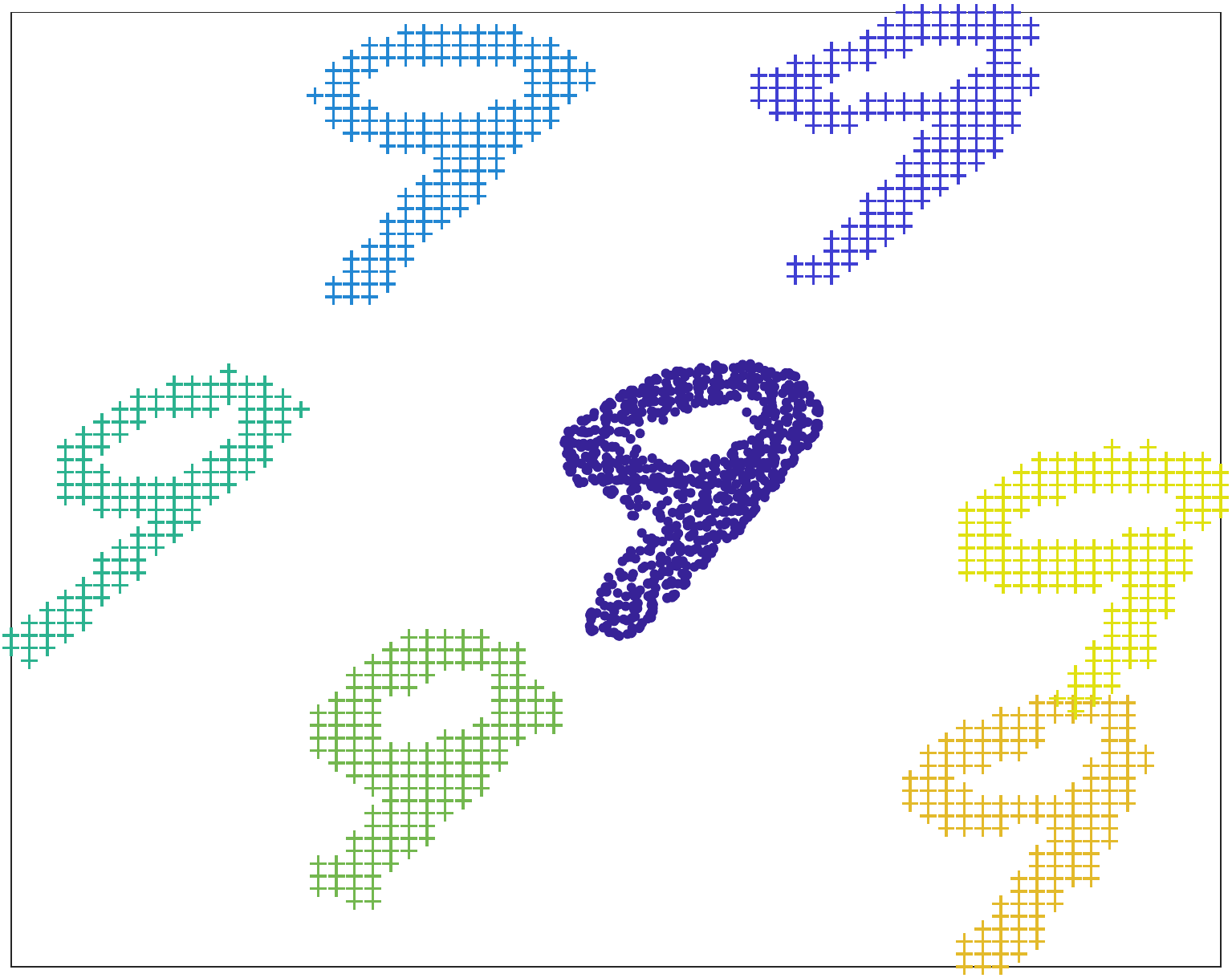}\\\hline

\end{tabular}
\caption{Barycenters of $6$ randomly chosen handwritten digits from MNIST dataset, solved numerically with different sets of test functions. The barycenter $y$ is displayed in deep purple in the center, surrounded by the source data $x$ in different colors indicating the index $z \in \{ 1, 2, 3, 4, 5, 6\}$. The first row contains the solution to Problem \ref{prob:factor} using as test functions only polynomial of first degree, the second row uses polynomials of second degree, and the third row has the solution to Problem 1, where the test function, slaved to the map, evolves through the kernel density estimator in (\ref{eq:Fzcat}).}
\label{table:MNIST}
\end{table}

\subsubsection{ Effect of the cost function }
Looking at Table 1, one may think at first that the barycenter has not been fully resolved, as its contours are not well defined. Figure \ref{fig::distort}, displaying the push forward of each of the six marginals to the barycenter, shows that this is not quite the case, as all push forward measures agree, except when only the preconditioner is used, forcing each of the six  marginals to keep its original shape. Small differences are due to the fact that each marginal contains a different number of sample points. 

The fact that the digit six in Figure \ref{fig::distort} (b) looks somewhat cloudy should not come as a surprize, since nothing in the objective function enforces the notion that the maps to the barycenter should not smear the original digits. 
A way to address this is to adopt a distortion-sensitive cost function in (\ref{eq:DefDBary}), namely
\begin{equation}
C(y(x,k),\rho) =\frac{1}{N^2} \sum_{1\leq i \neq j\leq N}  \left[ \left(   \frac{  || y_i^k - y_j^k ||^2   }{ || x_i^k - x_j^k ||^2 + \epsilon^2  } - 1 \right)^2 \right] +  \omega \frac{1}{N} \sum_{i=1}^N || y_i^k - x_i  ^k||^2. 
\label{eq::distort}
\end{equation}
The first term penalizes the deviation of the map from a conditional isometry, which would have equal pairwise distances in $x$ and $y=T(x, z)$ space for each value of $z$ ($k$ in our discrete setting), with a small parameter $\epsilon$ to prevent division by zero. The second term is a remnant of a regular optimal transport cost, intended to anchor the barycenter in space, with small weight  $\omega = 0.01$ in our numerical example. 
We compare the results obtained with the new cost and with the $L^2$ distance. The barycenters are displayed in Table \ref{table::MNIST2}, and the mapped samples from the digit six with different $z$ values are shown in Figure \ref{fig::distort}. Clearly the adoption of the cost in (\ref{eq::distort}) results
in a barycenter with more defined contours, as the need to preserve pairwise distances prevents the points in the upper branch of the digit six to broaden up when mapped to the barycenter.
\begin{table}[h!]
\centering
\setlength{\tabcolsep}{0.5pt}
\begin{tabular}{ccccccccccc}\\ \hline
& 0 & 1 & 2 & 3 & 4 & 5 & 6 & 7 & 8 & 9 \\ 
\rotatebox{90}{}&
\includegraphics[width=0.1\textwidth]{Figures/PDF_files/Digits/KDE_L2/digit0.pdf} &
\includegraphics[width=0.1\textwidth]{Figures/PDF_files/Digits/KDE_L2/digit1.pdf} &
\includegraphics[width=0.1\textwidth]{Figures/PDF_files/Digits/KDE_L2/digit2.pdf} &
\includegraphics[width=0.1\textwidth]{Figures/PDF_files/Digits/KDE_L2/digit3.pdf} &
\includegraphics[width=0.1\textwidth]{Figures/PDF_files/Digits/KDE_L2/digit4.pdf} &
\includegraphics[width=0.1\textwidth]{Figures/PDF_files/Digits/KDE_L2/digit5.pdf} &
\includegraphics[width=0.1\textwidth]{Figures/PDF_files/Digits/KDE_L2/digit6.pdf} &
\includegraphics[width=0.1\textwidth]{Figures/PDF_files/Digits/KDE_L2/digit7.pdf} &
\includegraphics[width=0.1\textwidth]{Figures/PDF_files/Digits/KDE_L2/digit8.pdf} &
\includegraphics[width=0.1\textwidth]{Figures/PDF_files/Digits/KDE_L2/digit9.pdf}\\\hline
\rotatebox{90}{}&
\includegraphics[width=0.1\textwidth]{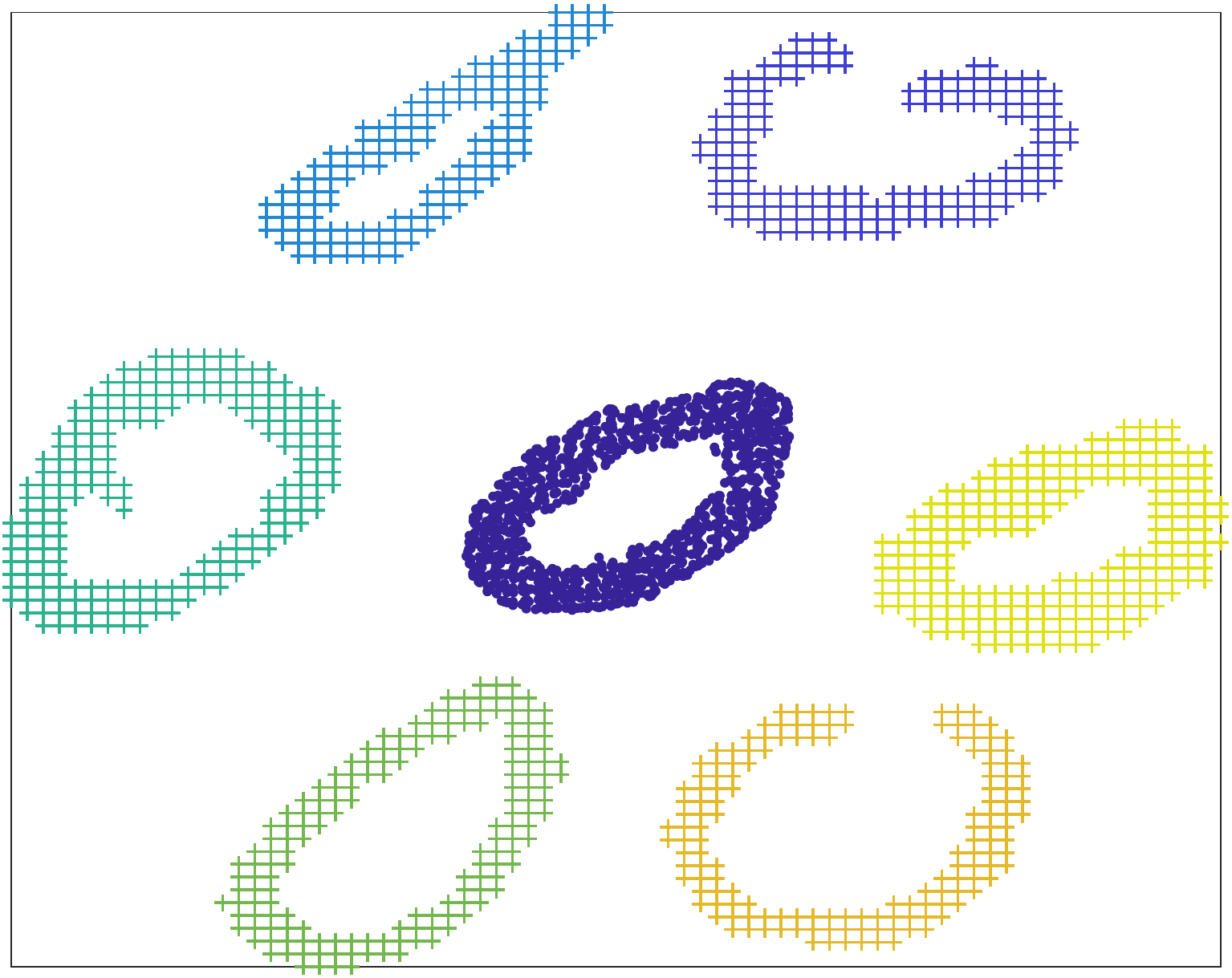} &
\includegraphics[width=0.1\textwidth]{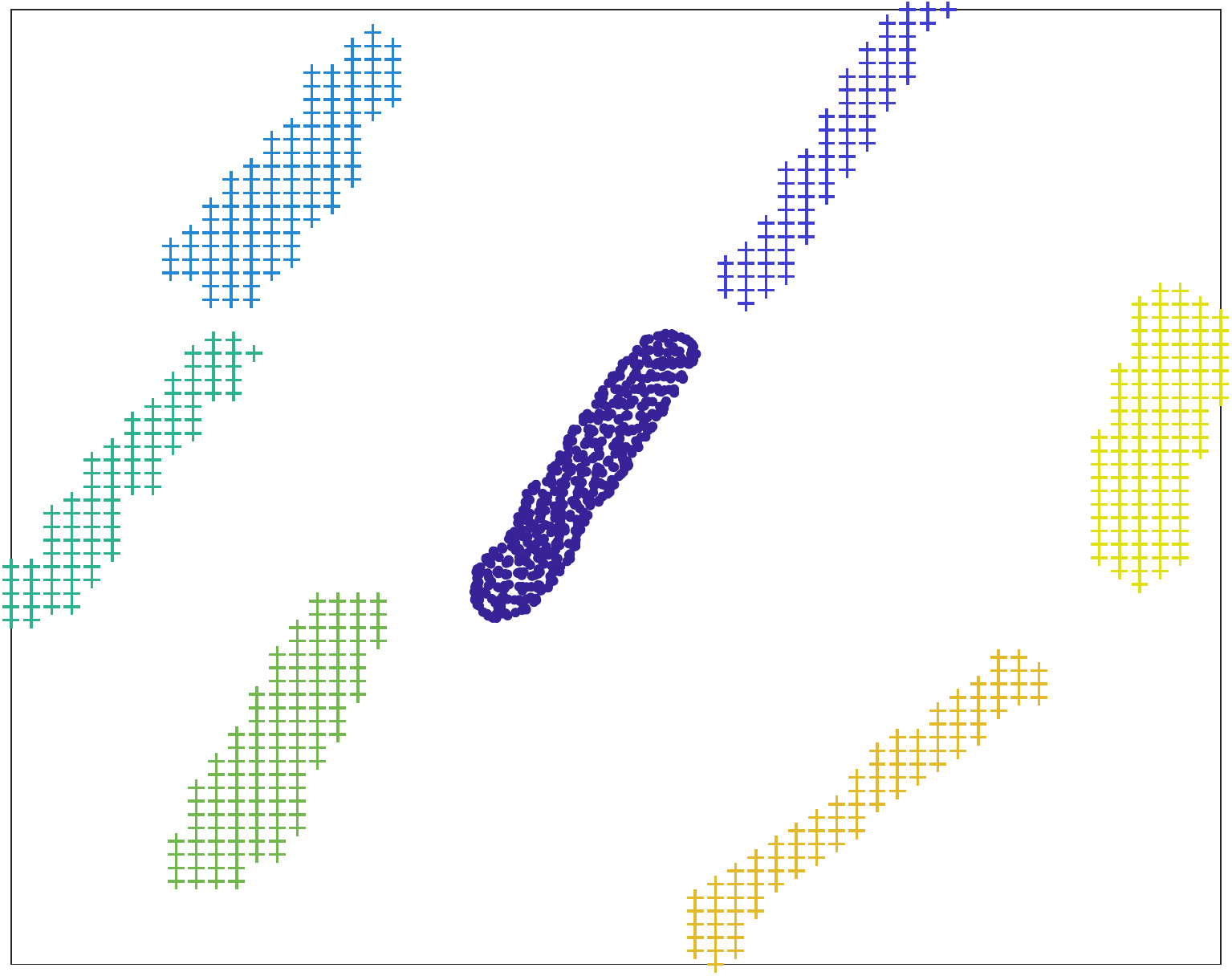} &
\includegraphics[width=0.1\textwidth]{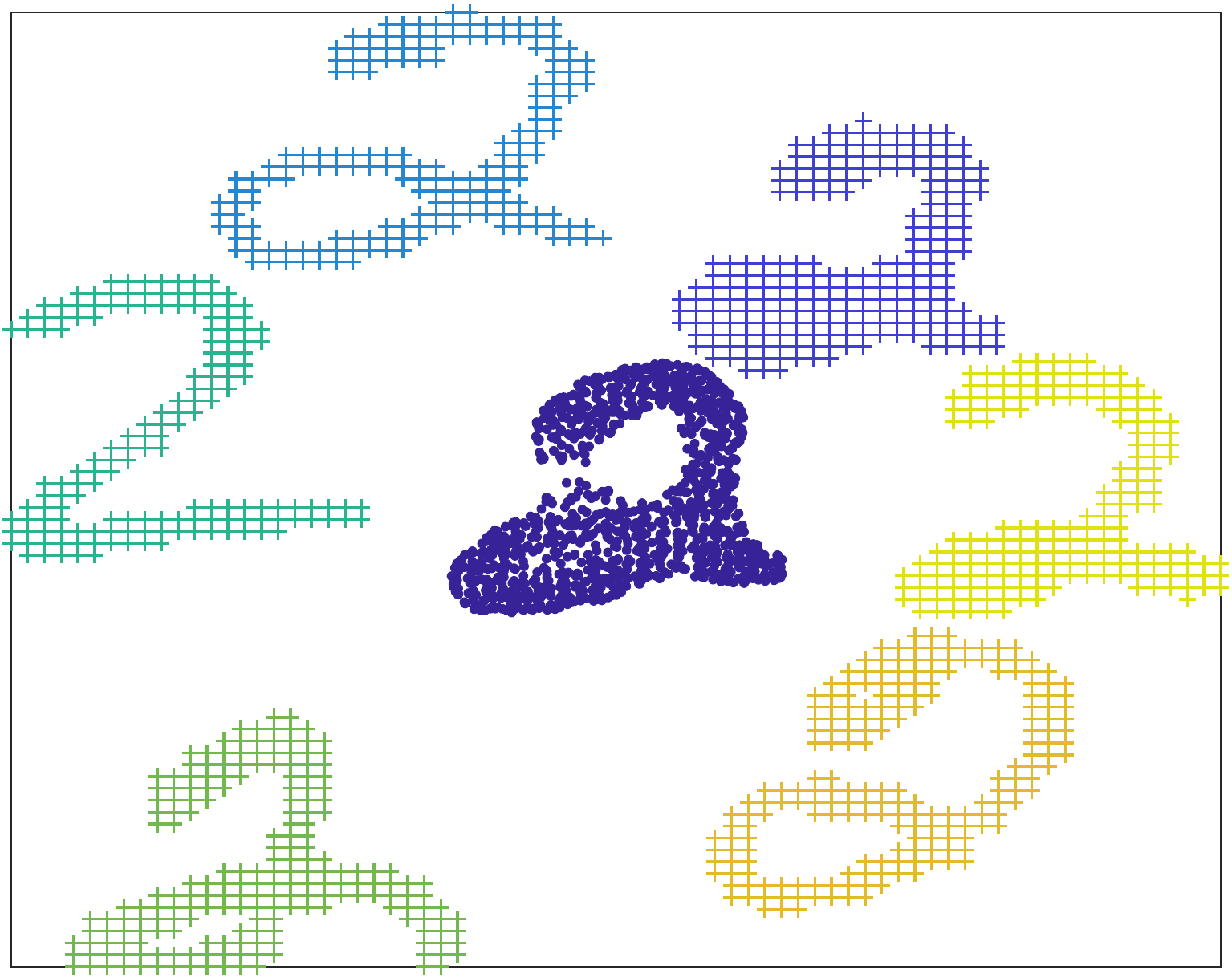} &
\includegraphics[width=0.1\textwidth]{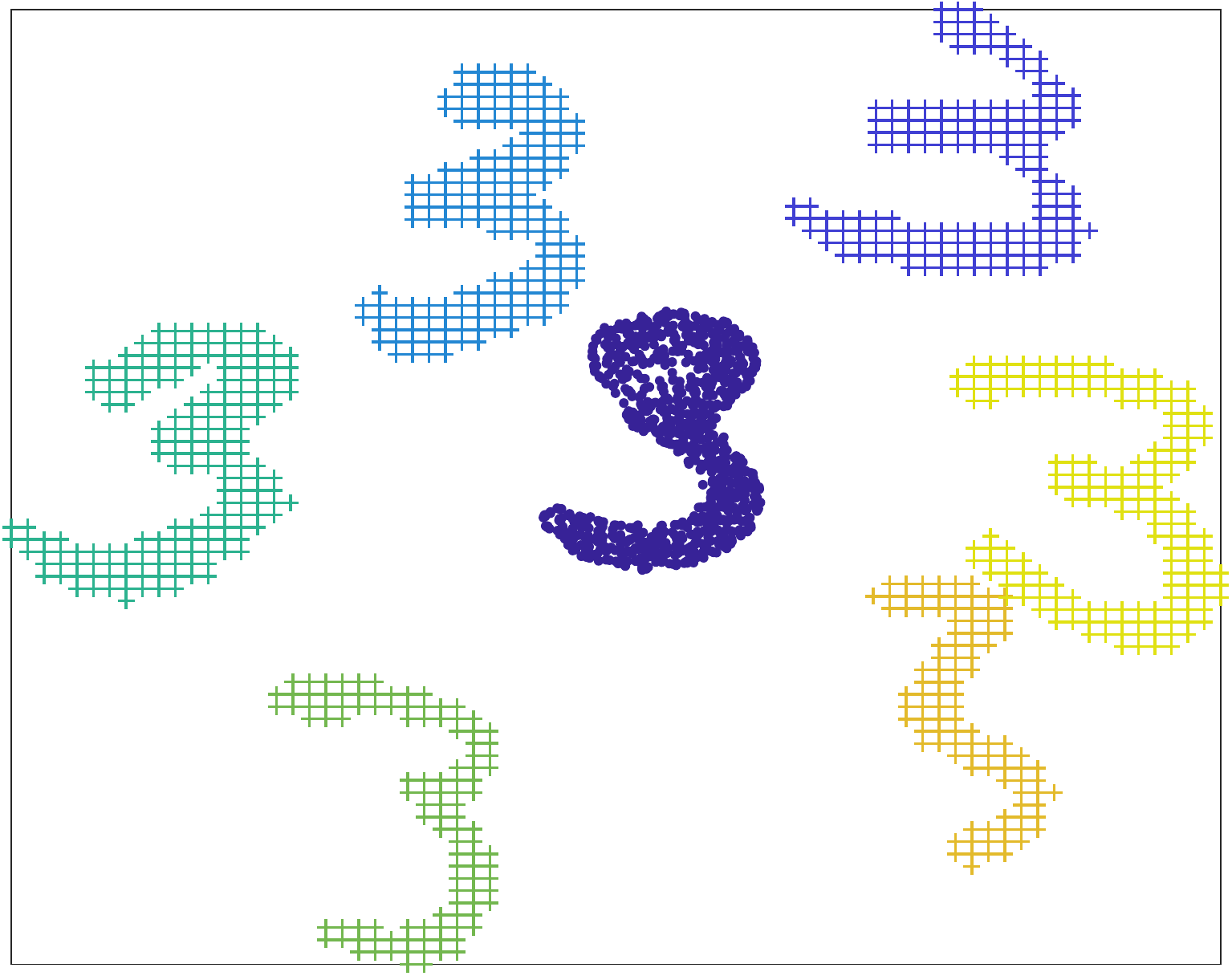} &
\includegraphics[width=0.1\textwidth]{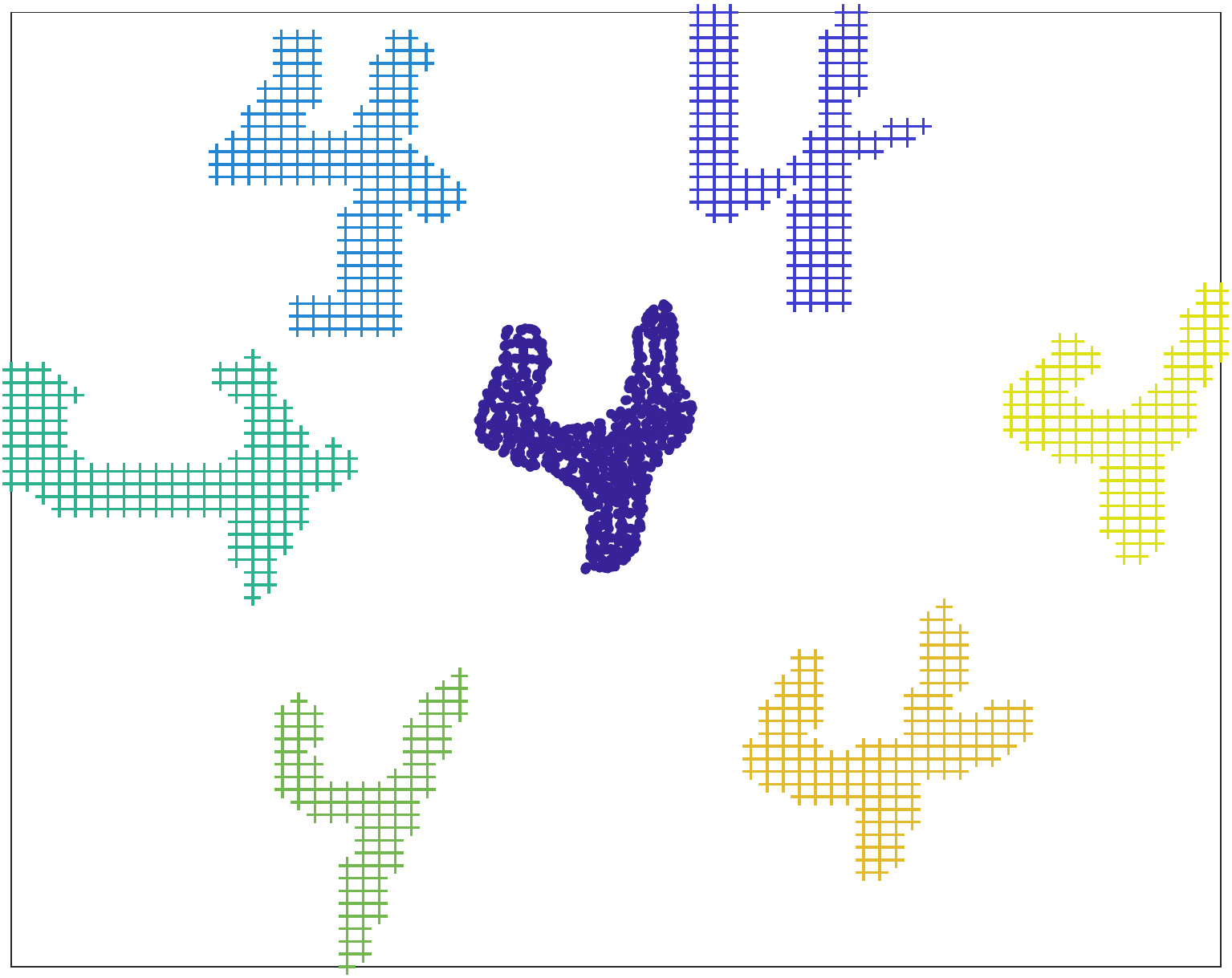} &
\includegraphics[width=0.1\textwidth]{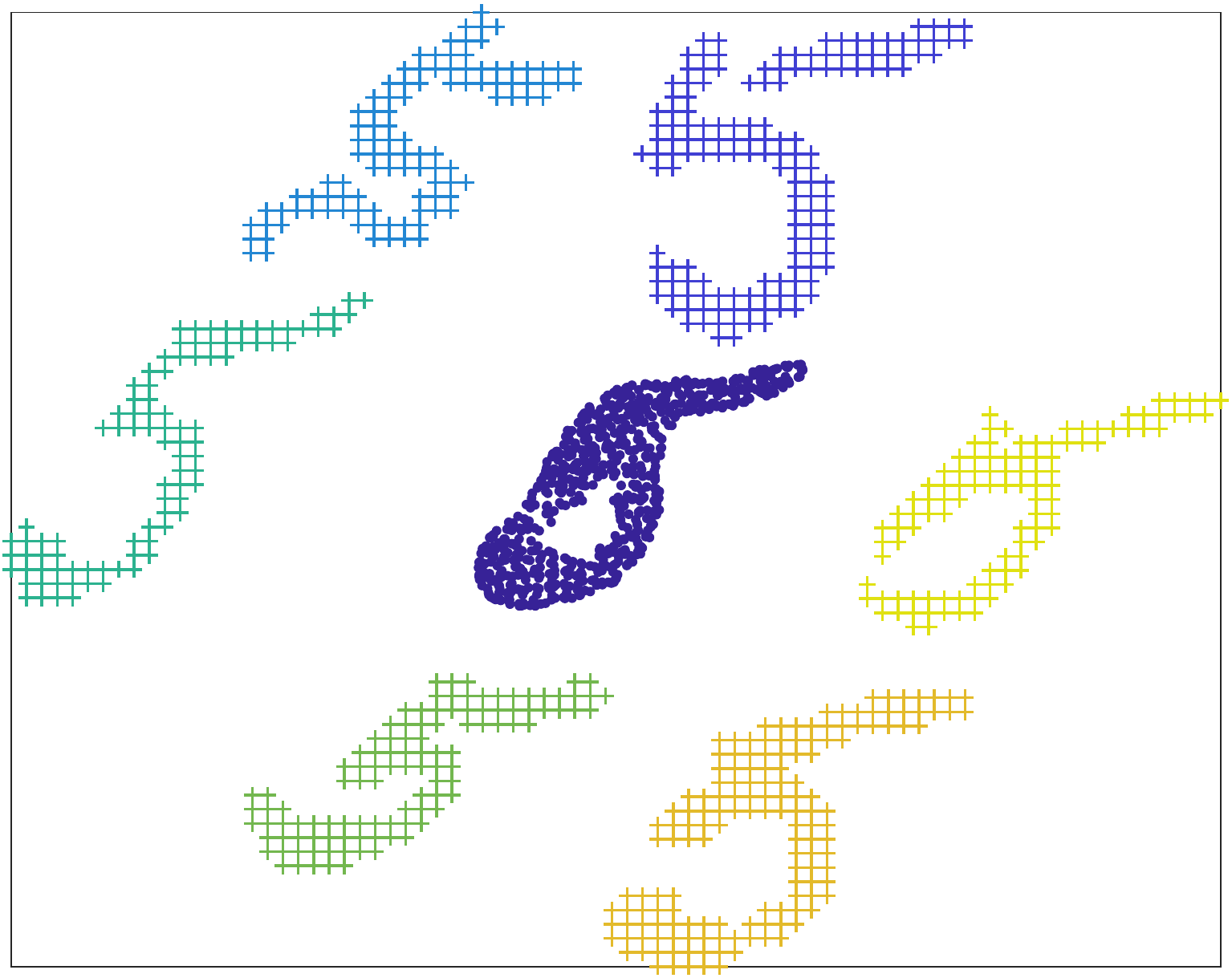} &
\includegraphics[width=0.1\textwidth]{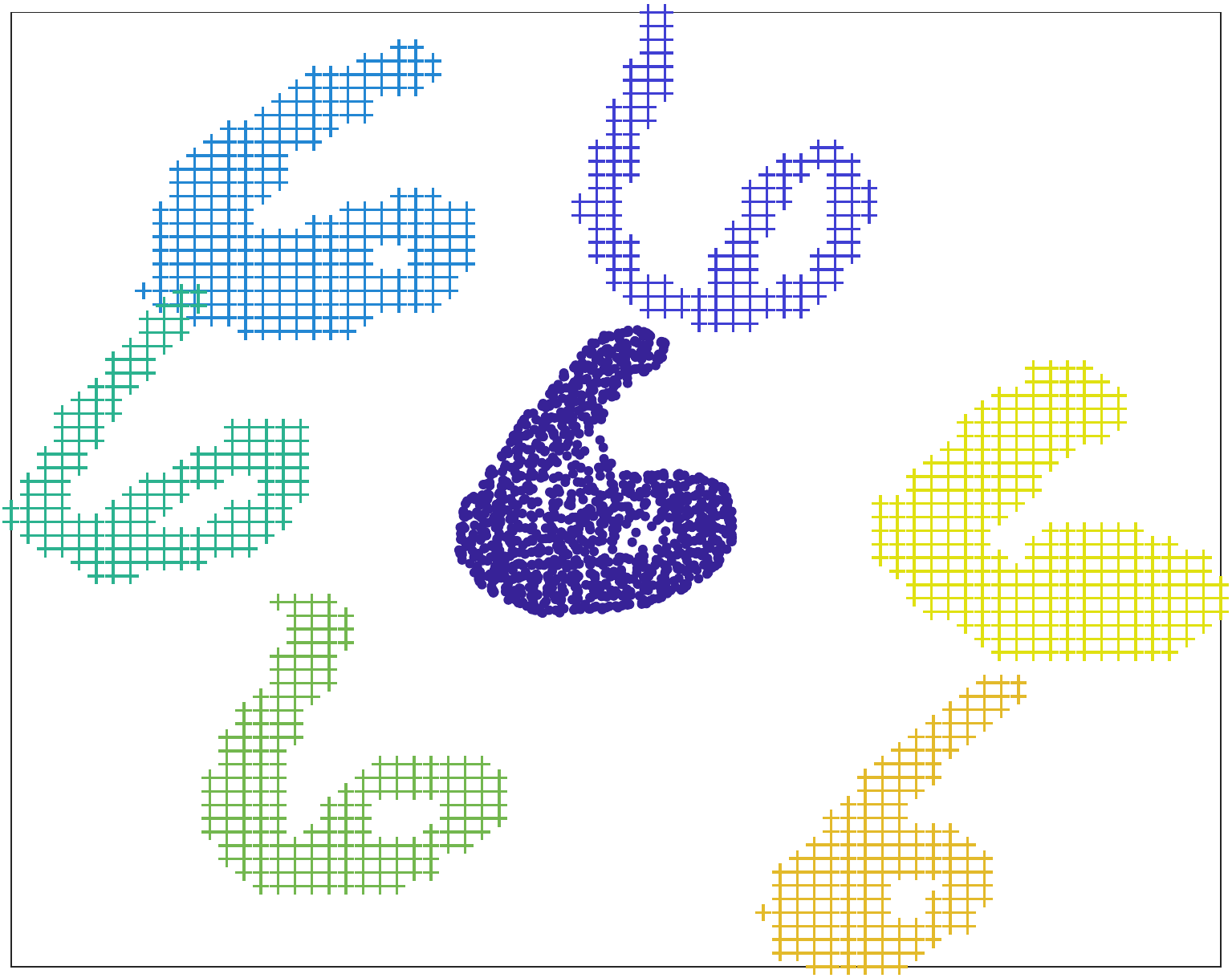} &
\includegraphics[width=0.1\textwidth]{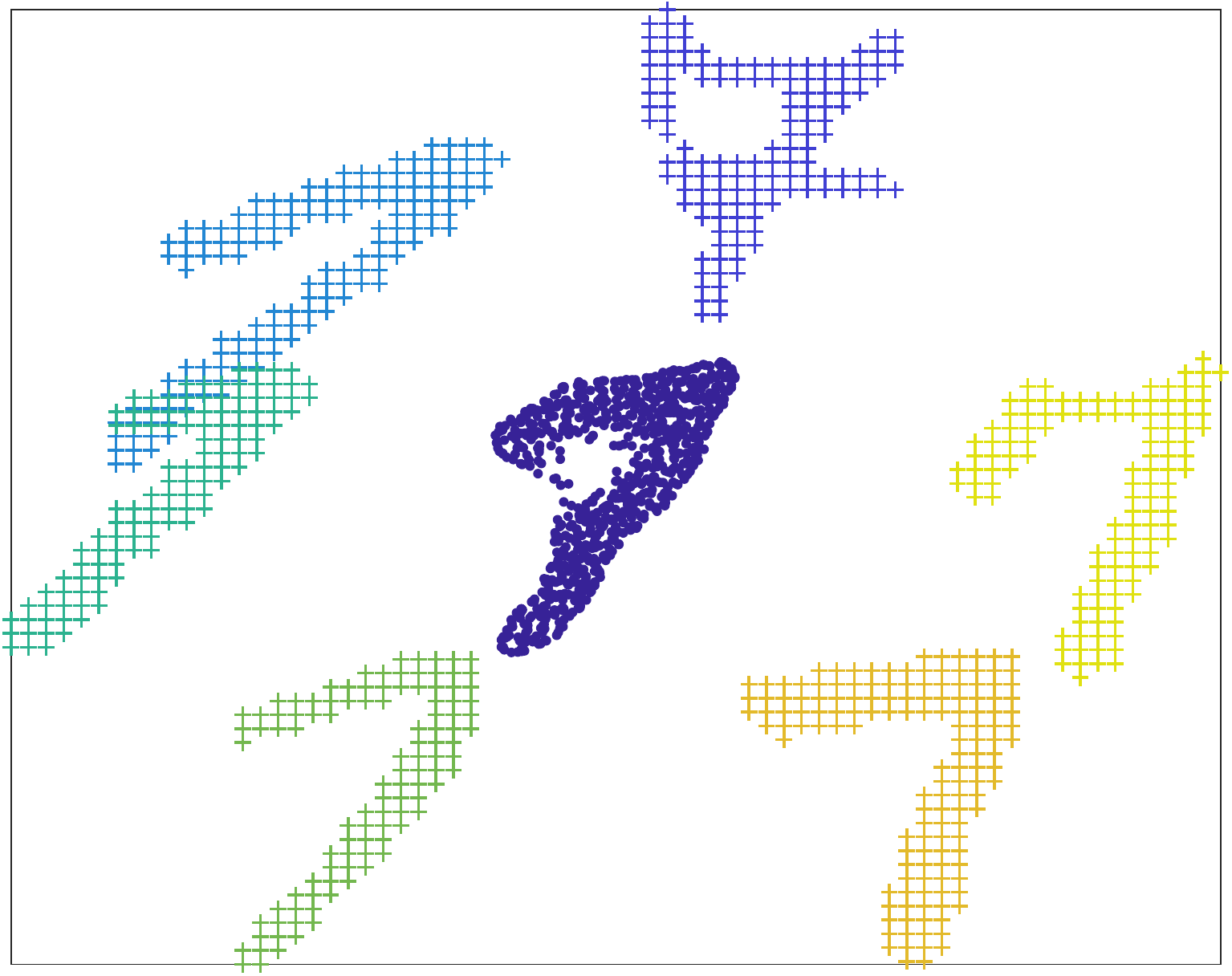} &
\includegraphics[width=0.1\textwidth]{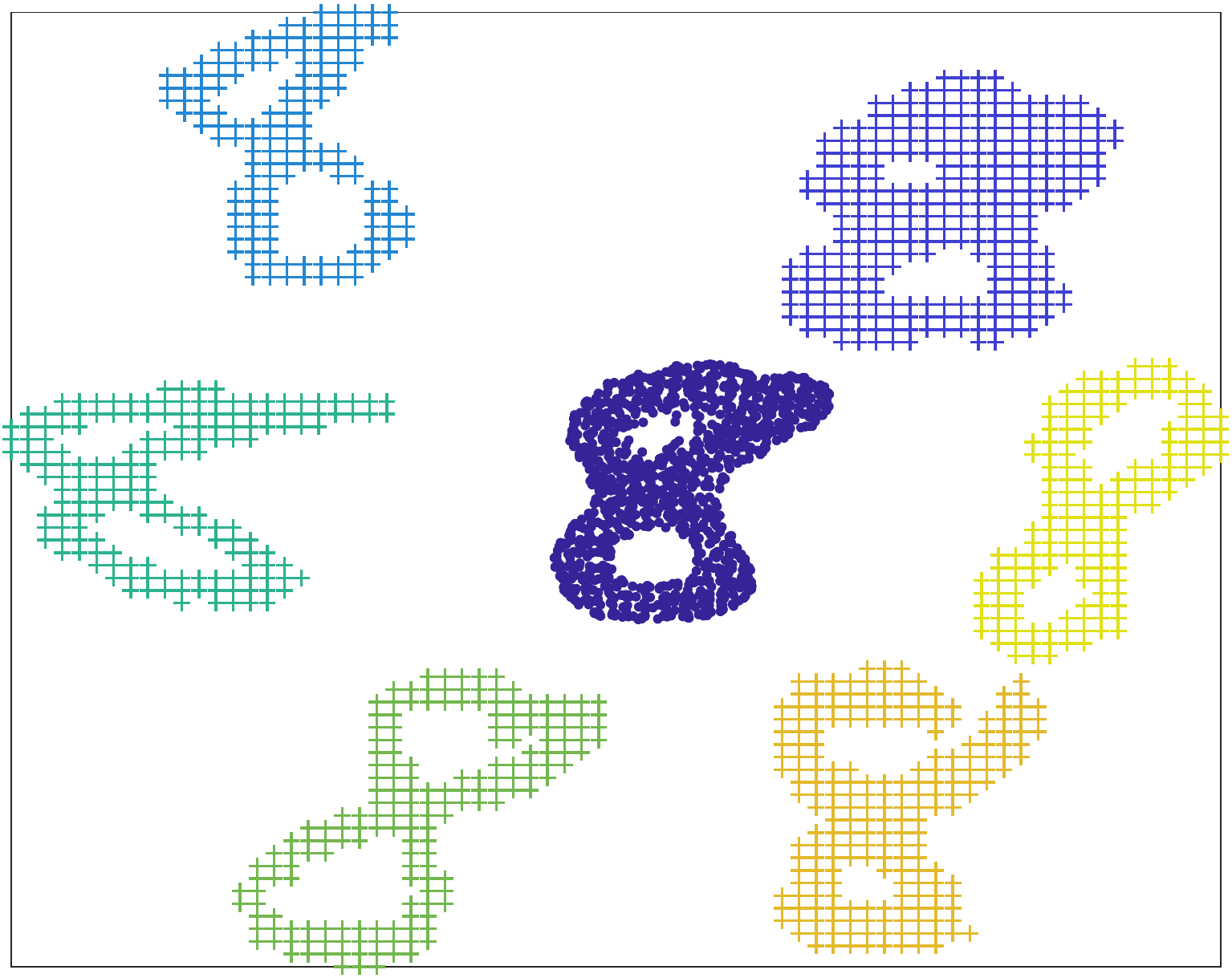} &
\includegraphics[width=0.1\textwidth]{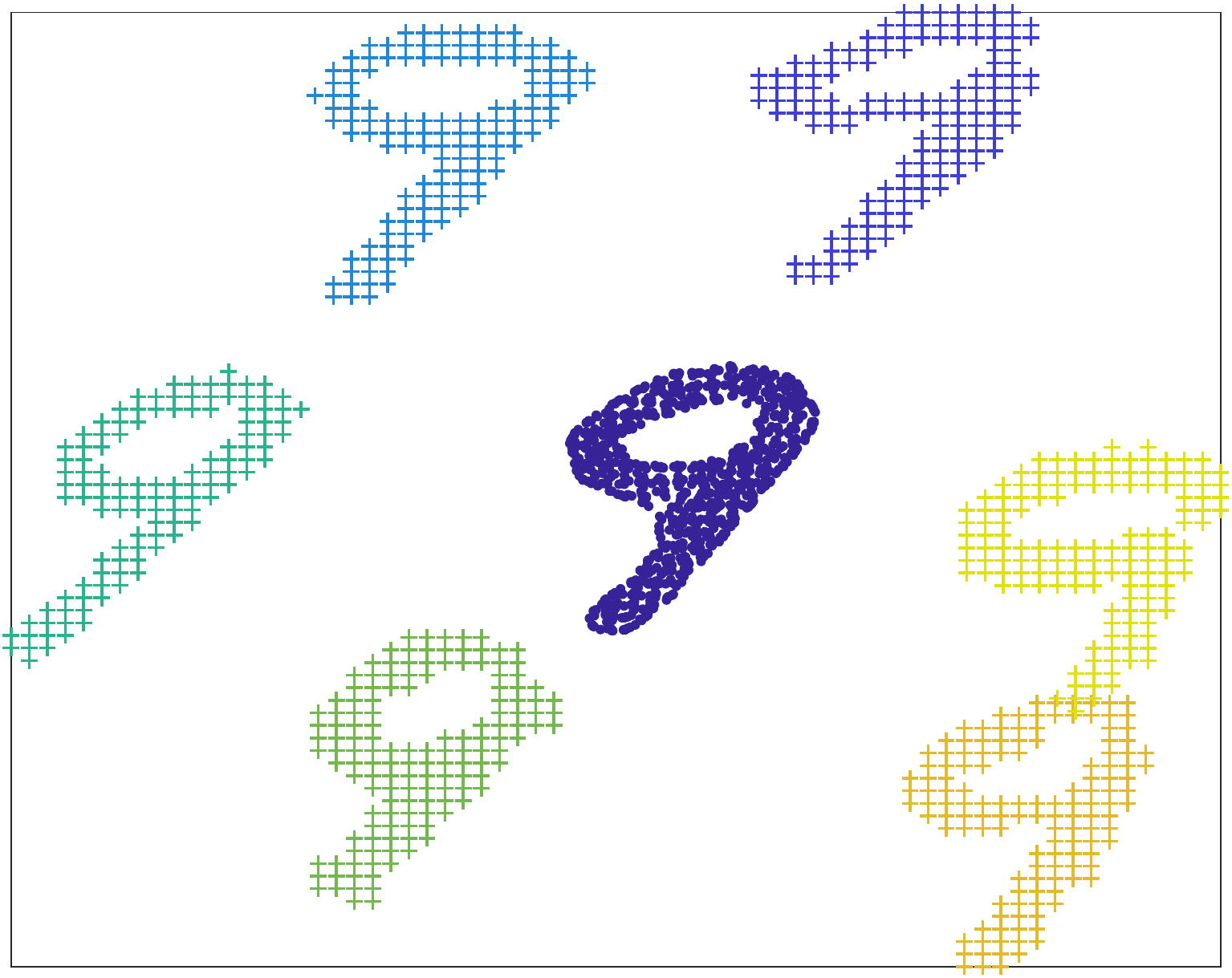}\\\hline
\end{tabular}
\caption{ Barycenters of digits solved non-parametrically,  the first row under a squared distance cost and the second under the distortion-sensitive cost \eqref{eq::distort}.  }
\label{table::MNIST2}
\end{table}

\begin{figure}[h!]
\centering
\subfigure[ Preconditioned ]{
\includegraphics[width=0.3\textwidth]{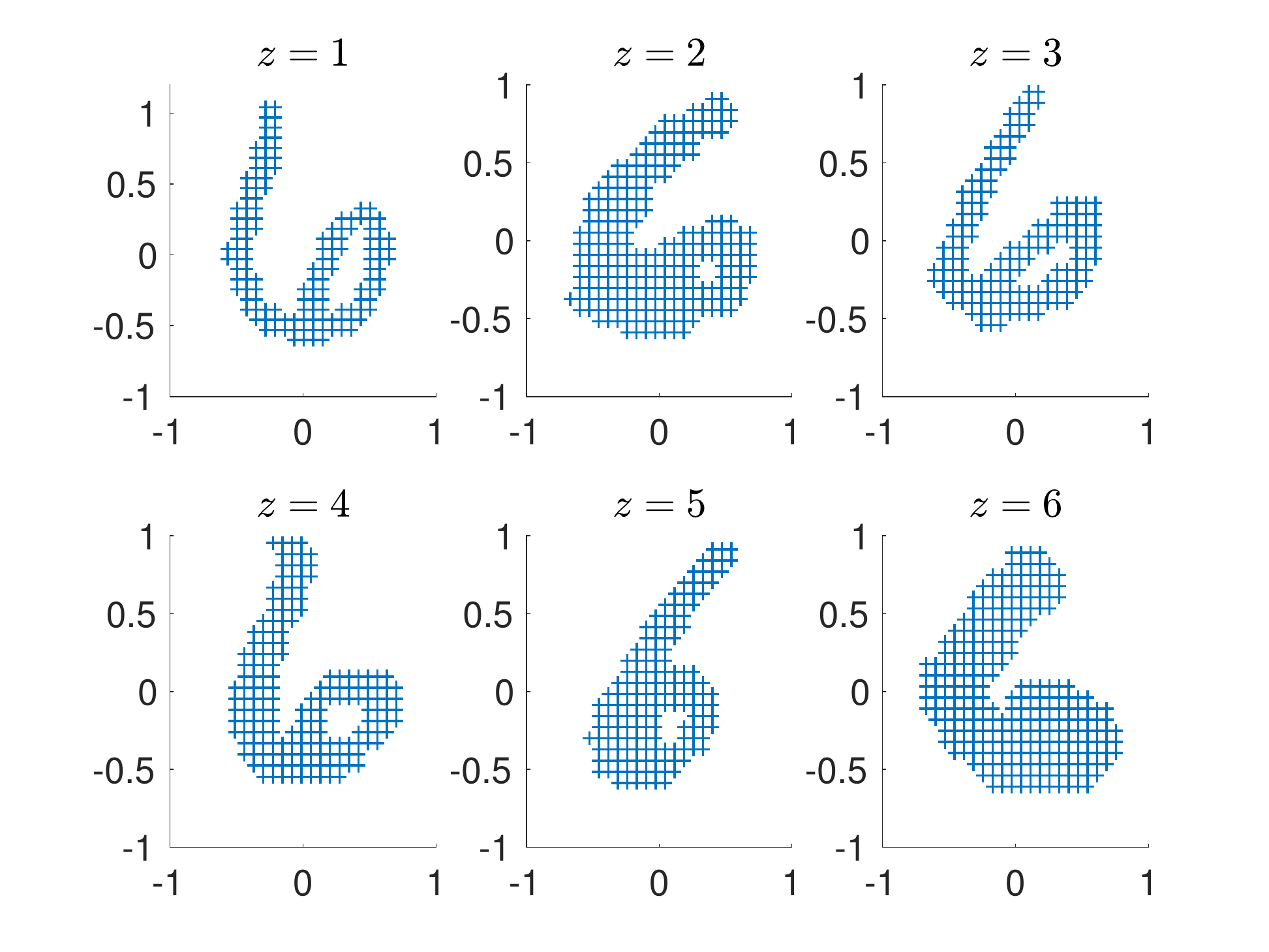}}
\subfigure[ With $L^2$ cost]{
\includegraphics[width=0.3\textwidth]{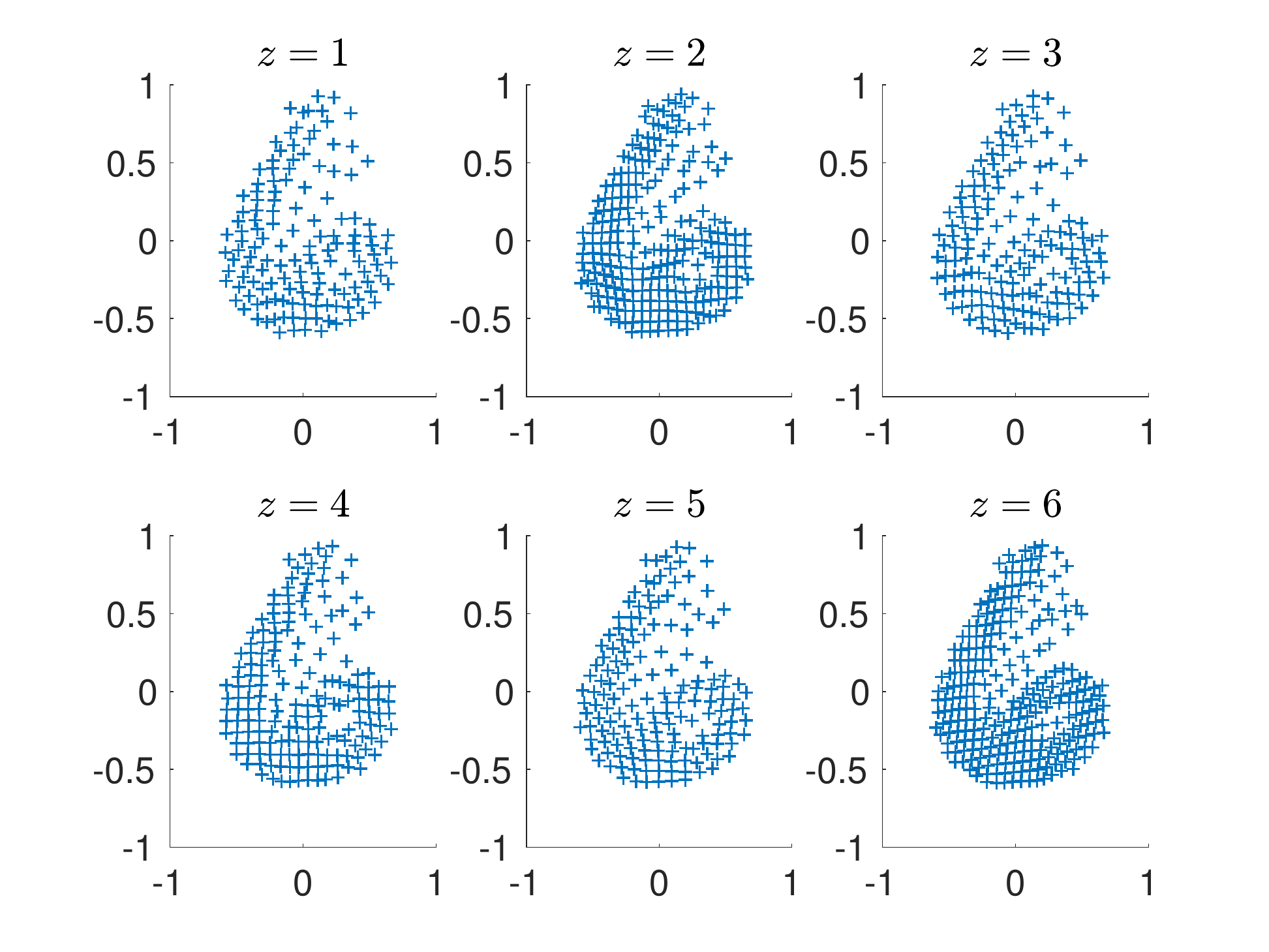}}
\subfigure[With cost in \eqref{eq::distort}]{
\includegraphics[width=0.3\textwidth]{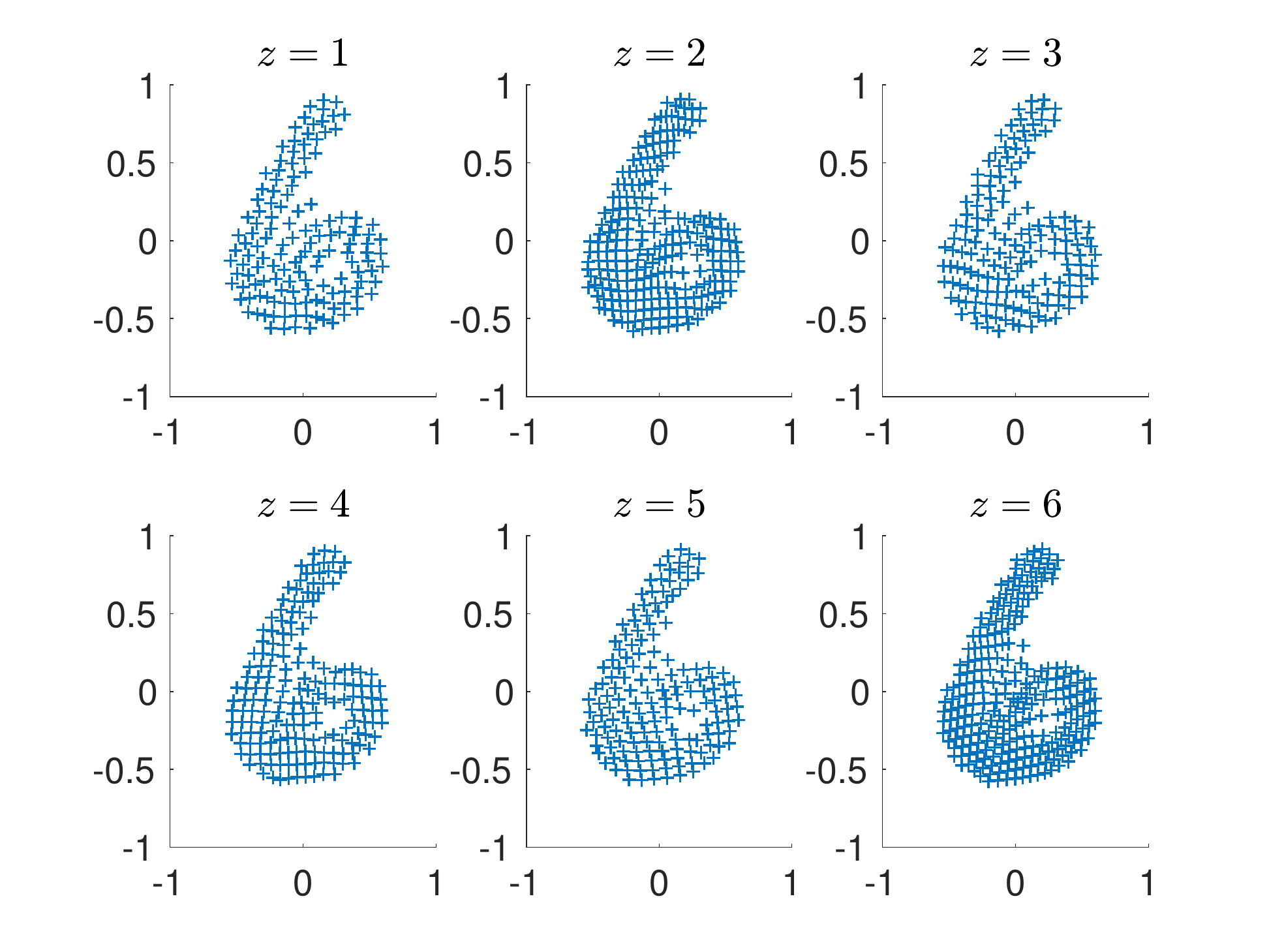}}
\caption{Samples of the digit $6$ push-forwarded to the barycenter for each value of $z$.  }
\label{fig::distort}
\end{figure}

\goodbreak

\subsection{Two patches on the unit sphere}\label{subsec:Patches}
This section performs a numerical experiment on the barycenter of two distributions, with samples shown in Figure \ref{fig:patches}, defined on the unit sphere $\mathbb{S}^2 = \{  x \in \mathbb{R}^3:  ||x||=1 \}$. Because the sample points are defined on the sphere, we can represent them in spherical coordinates,
$$
x_1 = \cos \theta \cos \phi, \quad x_2 = \cos \theta \sin \phi,\quad x_3 = \sin \theta,
$$
where $\theta \in [0,2\pi)$ and $\phi \in [-\pi/2,\pi/2]$ represent longitude and latitude. Then the natural cost is not the canonical Euclidean $L^2$ distance  $c( x, y)=|| x- y||^2$, but the geodesic distance between points:
$$
\tilde c(  x, y ) = 2\arcsin \sqrt{  \sin ^2 \left( \frac{|\theta_x-\theta_y|}{2 } \right) + \cos \theta_x \cos \theta_y \sin ^2 \left( \frac{\phi_x - \phi_y | }{2} \right)     }.
$$

The example illustrates how the barycenters capture essential features of the manifold on which the data are defined. When the two distributions are supported on the same hemisphere (left two panels of Figure \ref{fig:patches}, in red and black), the support of the barycenter (in blue) interpolates between them. By contrast, when the two distributions lie around the north and the south pole respectively, there is no preferred meridian on which the barycenter should lie, resulting in it being supported along the entire equator.

\begin{figure}[!htb]
%  \begin{center}
\hspace{-0.6cm}
      \begin{tabular}{ll}                                              
      \resizebox{66mm}{!}{\includegraphics{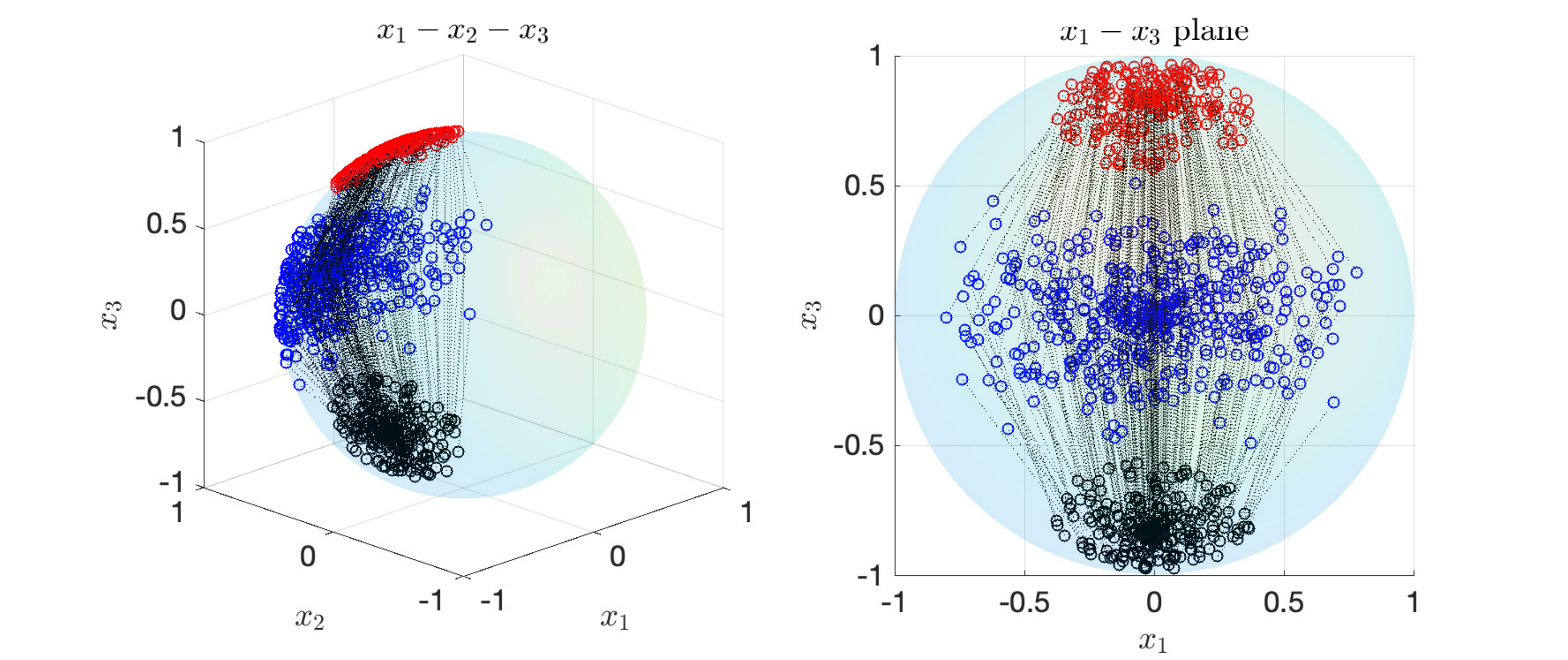}} & \hspace{-0.8cm}                                      
      \resizebox{66mm}{!}{\includegraphics{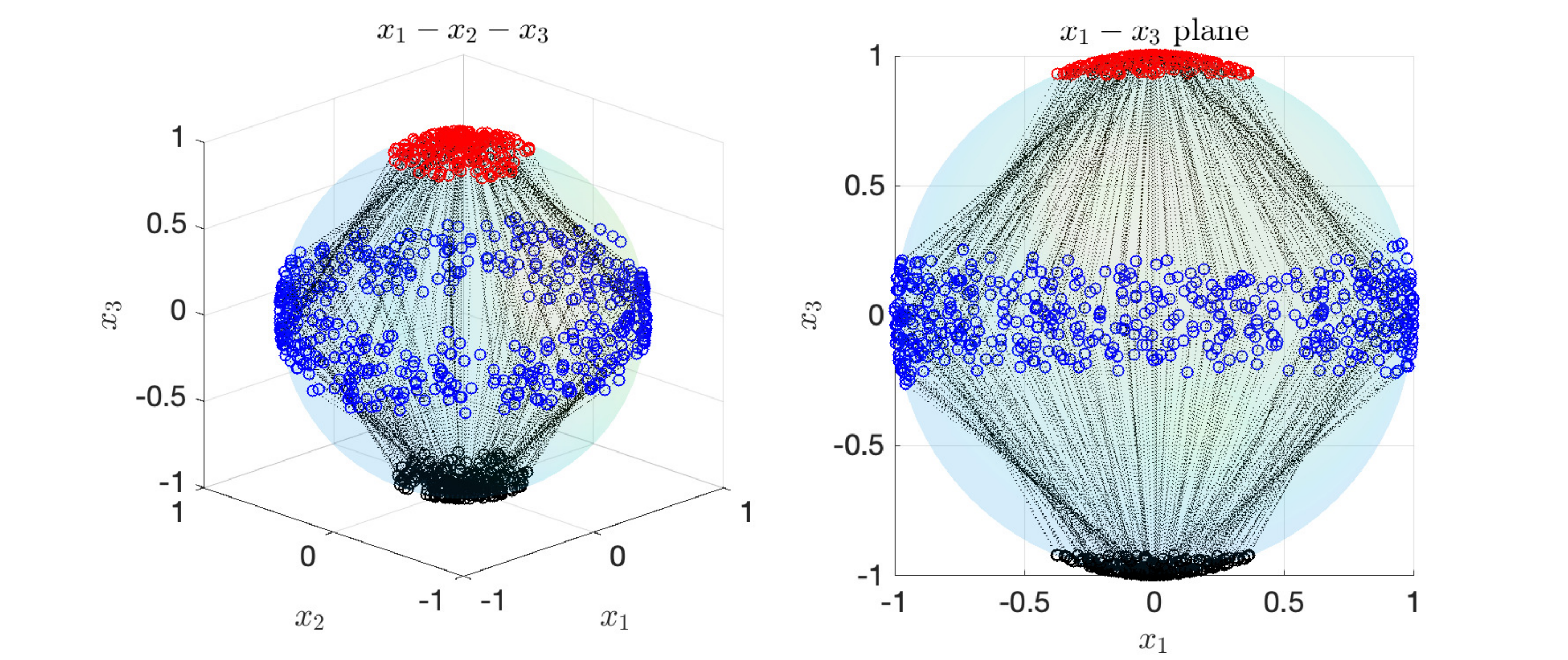}}\\
     \end{tabular}
    \caption{Barycenter (blue) of two distributions with $z=0$ (red) and $z=1$ (black), with 250 points sampled from each. The thin black lines indicate the one-to-one correspondence between the source data and its image under the map. Each distribution has $\theta \sim U_{[0, 2\pi)}$, $\phi \sim U_{[ \frac{3}{8}\pi, \frac{1}{2}\pi ]}$ and $\phi \sim U_{[ -\frac{1}{2}\pi, -\frac{3}{8}\pi ]}$ respectively on the right two panels, while they are shifted to the same hemisphere on the left two.}
    \label{fig:patches}
%  \end{center}
\end{figure}

\subsection{Hidden variability recovery on the unit sphere}
\label{subsec:HiddenSignal}

Time series are often modeled through a Markov model of the form
\begin{equation}\label{eq:MM}
 x^{n+1} = F(x^n, z_{known}^{n+1}, w^{n+1}, t^{n+1} ),
\end{equation}
where $\{x^n\}_{n=0}^T$ is the time series, $t$ is the time, $z_{known}$ represents known factors that influence $x$ and $w$ contains unknown sources of variability. In \cite{TT2}, the authors proposed a method to uncover the hidden variability $w^n$ by removing from $x^{n+1}$ the variability due to $z^{n+1}$, computing the barycenter of $\rho(x^{n+1}|z^{n+1})$ thorough the family of maps
$$
y^n = T(x^n, z^n),\quad z^n = [  x^{n-1}, t^{n}, z_{known}^{n}  ],
$$
so that the ``filtered'' signal $y^n$ is a function of only $w^n$.
This section shows a synthetic example combining this idea with the algorithm described in section \ref{sec:penalty} to study time series defined on Riemannian manifolds. In particular, we consider the time series defined on the $3$D unit sphere, generated as the sum of a deterministic dynamics and random noise:

\begin{itemize}
\item The deterministic dynamics in spherical coordinates is given by:
\begin{equation*}
\begin{bmatrix}
\tilde \phi^{n+1} \\ \tilde \theta^{n+1}
\end{bmatrix} = \begin{bmatrix}
\phi^n \\
\theta^n + \sin (\theta^n) + \frac{1}{2}
\end{bmatrix},
\end{equation*}
where $\theta^n$ and $\phi^n$ are the longitude and latitude at $x^n$ respectively.
In Cartesian coordinates, this becomes
$\tilde x^{n+1} = Sph2Cart( R=1, \tilde \phi^{n+1}, \tilde \theta^{n+1}  ) )$.

\item The hidden factor $w^{n+1}$ is generated by first sampling a 2-dimensional uniform distribution in spherical coordinates, and then transforming the sampled points into Cartesian coordinates on the unit sphere: 
\begin{equation}\label{eq:hs}
 w^{n+1} = Sph2Cart(1, \phi_w^{n+1} , \theta_w^{n+1} ),
 %\quad \phi_w^{n+1} \sim U_{[  \frac{\pi}{2} - 0.45, \frac{\pi}{2} ]}, \, \theta_w^{n+1} \sim U_{[0,2\pi]}. 
\end{equation}
where  $\phi_w^{n+1} \sim U_{[  \frac{\pi}{2} - 0.45, \frac{\pi}{2} ]}$ and $\theta_w^{n+1} \sim U_{[0,2\pi]}.$
This results in the round patch centered at the north pole shown on the left panel of  Figure \ref{fig:signaldata}. In order to add $w^{n+1}$ to the deterministic part $\tilde x^{n+1}$, we define a one to one map between the tangent planes at the north pole and at $\tilde x^{n+1}$, through the reflection with respect to the axis $\tilde x^{n+1}_{1/2} =  Sph2Cart\left(1, \frac 12 ( \tilde \phi^{n+1} + \frac{\pi}{2} ), \tilde \theta^{n+1}  \right) $ bisecting the angle between the north pole and $\tilde x^{n+1}$. Using Rodrigues' rotation formula, this yields
\begin{equation}\label{eq:xn1}
 x^{n+1} = \left(  I + 2K^2(  \tilde x^{n+1}_{1/2} ) \right) w^{n+1},
\end{equation}
where $K\in \R^{3\times3}$ is the cross-product matrix:
$$
K(x) = \begin{bmatrix}
0 & -x_3 & x_2 \\ x_3 & 0 & -x_1 \\ -x_2 & x_1 & 0
\end{bmatrix}.
$$

\end{itemize}
Figure \ref{fig:signaldata} shows a time series of $1000$ steps, starting from the south pole, and the hidden signal $w$.
\begin{figure}[ht!]
\centering
\subfigure[Hidden signal $w$]{\includegraphics[width=0.49\textwidth]{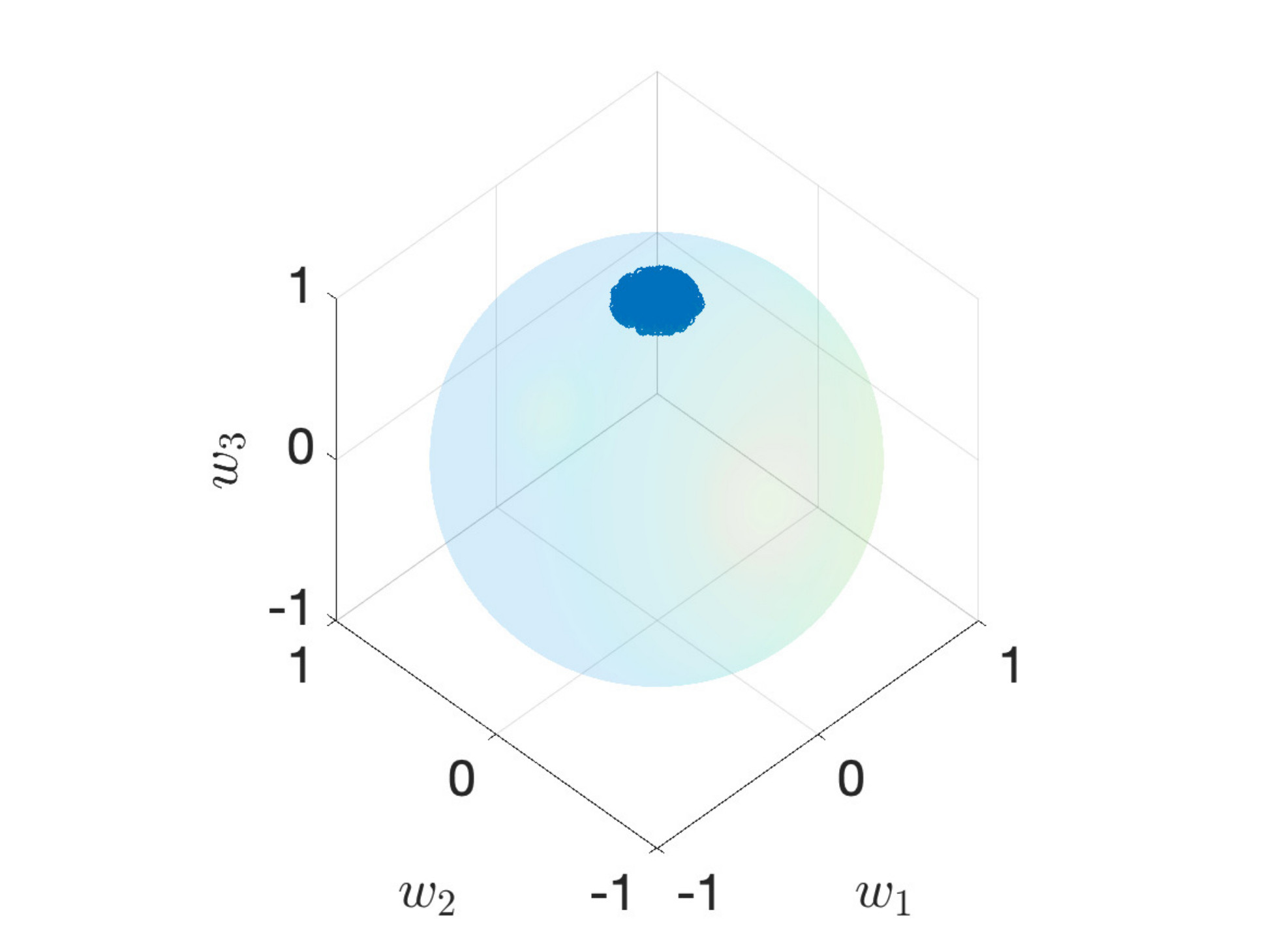}}
\subfigure[Time series $x$]{\includegraphics[width=0.49\textwidth]{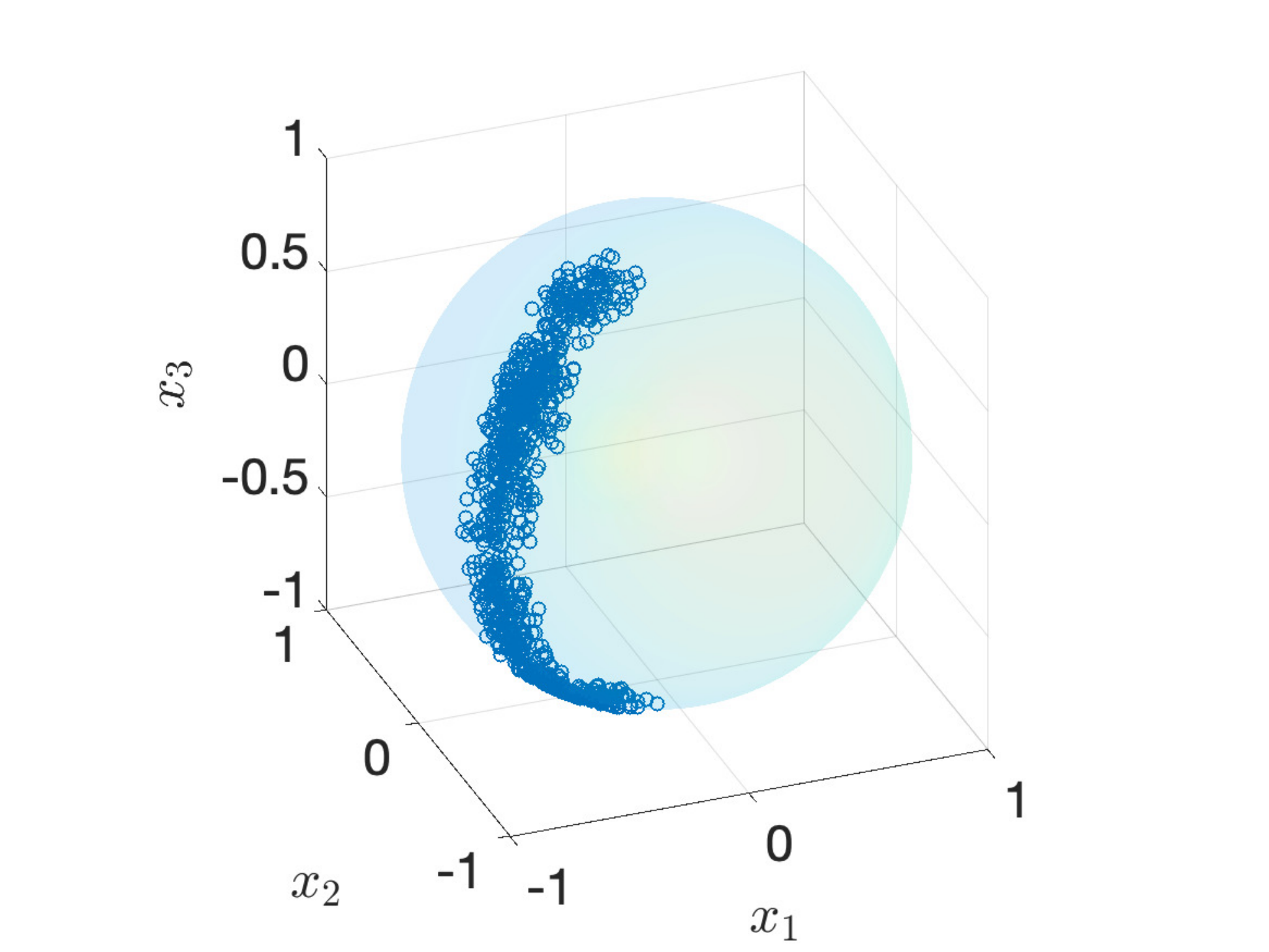}}
\caption{(a) Hidden signal generated using (\ref{eq:hs}). (b) Complete time series generated using (\ref{eq:xn1}).}
\label{fig:signaldata}
\end{figure}
Figure \ref{fig:signalresults} compares the barycenters obtained with two different methods:
\begin{enumerate}
\item Filter $ x^{n+1} $ with $z^{n+1}=x^n$ in $\mathbb{R}^3$, using the Euclidean distance as cost.
\item Filter $[ \phi^{n+1},\theta^{n+1}  ]$ with $z^{n+1} = [\phi^n, \theta^n]$ and great-circle distance as the cost. \end{enumerate}

\begin{figure}
\centering
\includegraphics[width=0.6\textwidth]{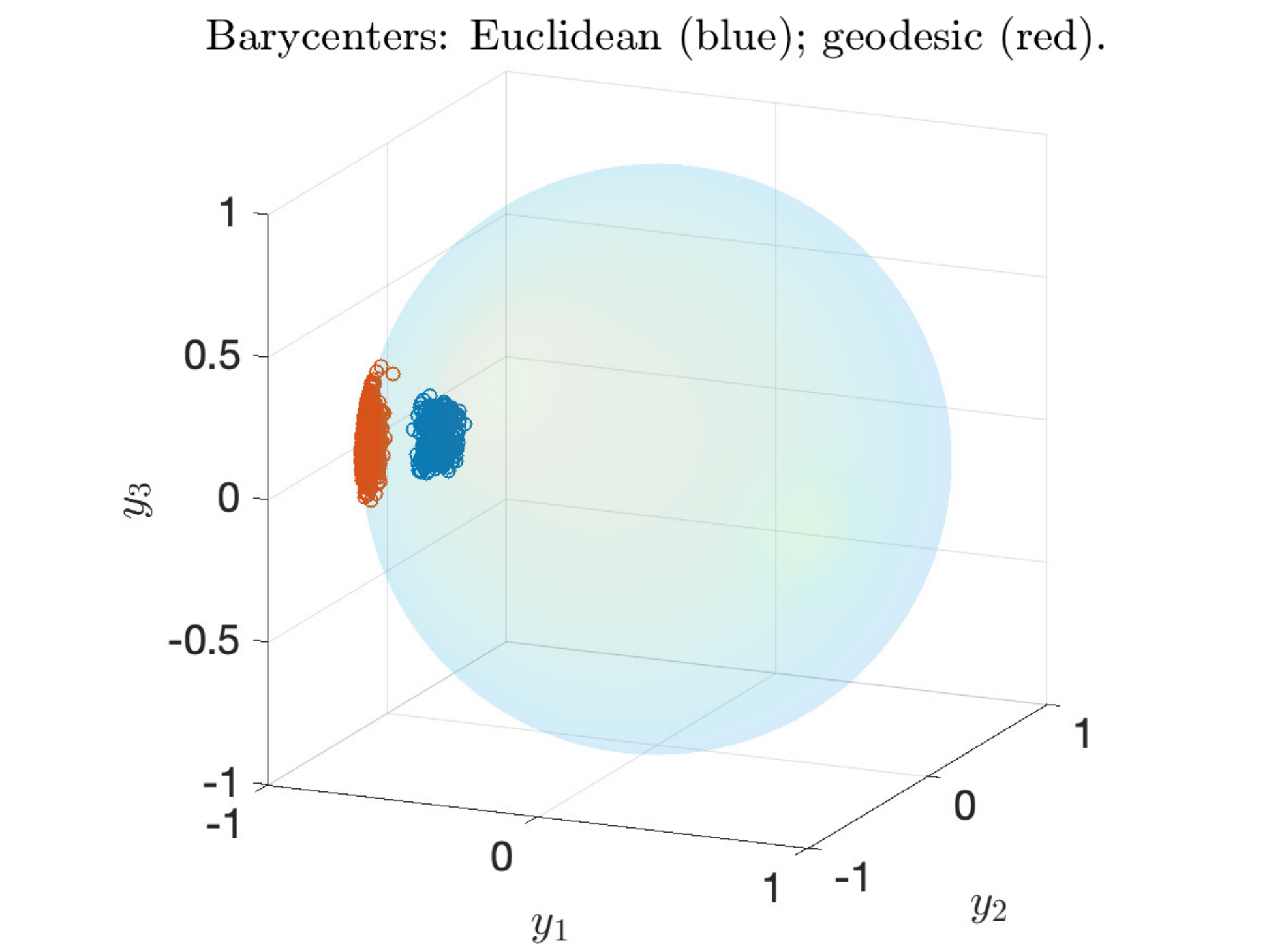} 
\caption{Barycenter solved by two different approaches, in 2D spherical coordinates (red) and 3D Cartesian coordinates (blue)  }
\label{fig:signalresults}
\end{figure}

The first approach ignores the fact that the time series is supported on a lower dimensional manifold of $\R^3$, resulting in a barycenter that is not on the surface of the sphere. The second approach respects the distance metric of the manifold, and the barycenter lays on the same lower dimensional manifold where the marginals are supported. To show that the filtered signal $y^{n}$ is a surrogate for $w^{n}$, it is enough to establish a one to one map between $y^{n}$ and $w^{n}$, a map that depends on the specific form of $F$ in (\ref{eq:MM}). In order to visualize this map, we align the (normalized) barycenter to the hidden noise by means of linear regression.  Figure \ref{fig:signalsurface} shows the resulting smooth dependence between the hidden signal and the barycenter. Figure \ref{fig:MAtime} dsiplayes a moving average of the barycenter and the hidden signal as a functions of time, providing further evidence that the two signal overlap. 

\begin{figure}[h!]
\centering
\includegraphics[width=0.9\textwidth]{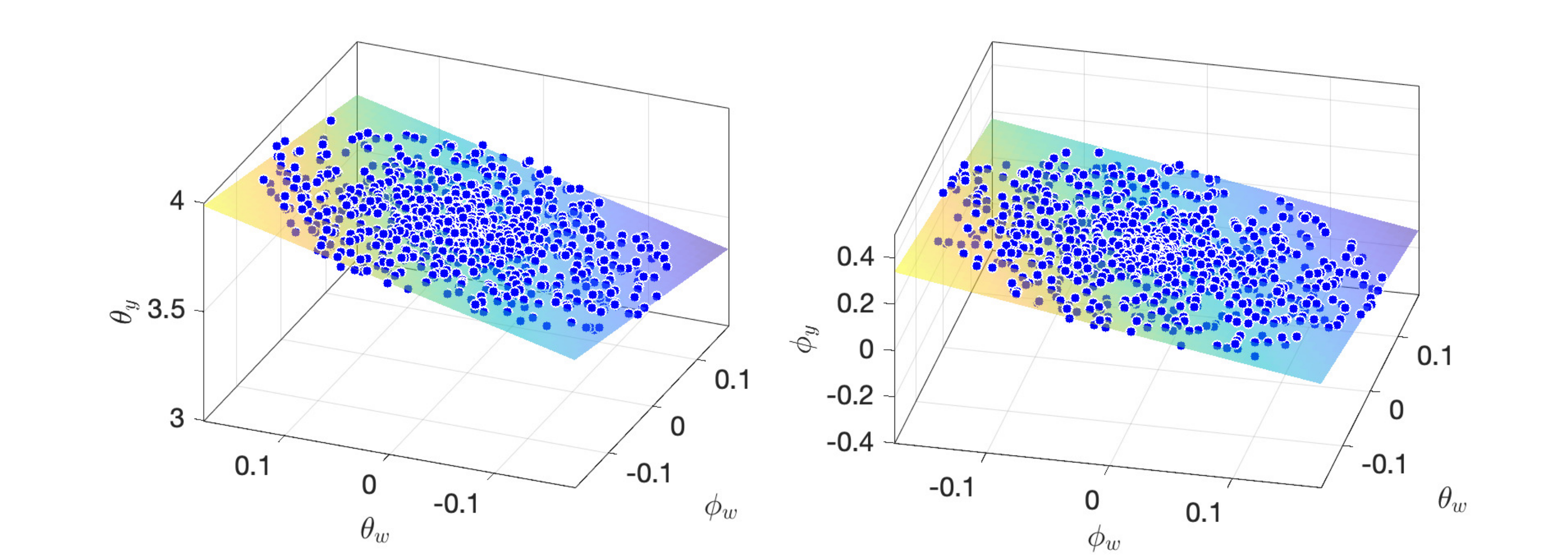}
\includegraphics[width=0.9\textwidth]{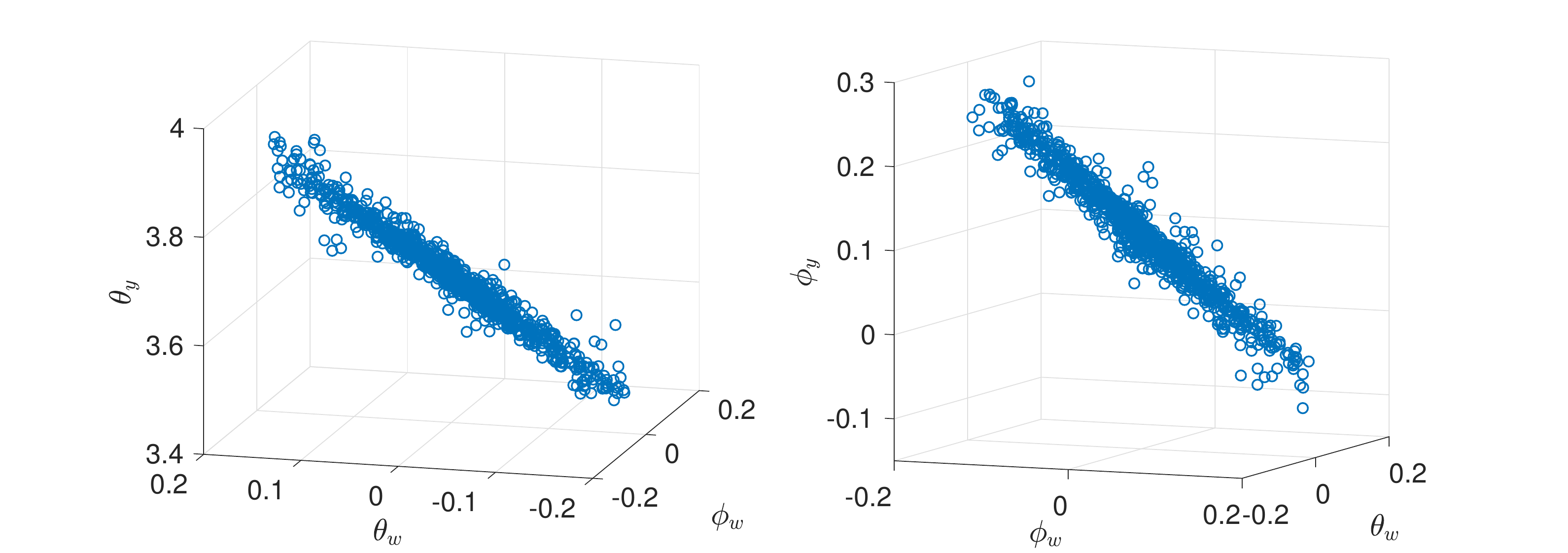}
\caption{The spherical coordinates of the barycenter solved in $2$D as functions of the hidden signal (also in spherical coordinates). Polynomial surfaces of order $5$ are fitted to the data and visualized.}
\label{fig:signalsurface}
\end{figure}

\begin{figure}[!htb]
%  \begin{center}
\hspace{-0.9cm}
      \begin{tabular}{ll}                                           
      \resizebox{69mm}{!}{\includegraphics{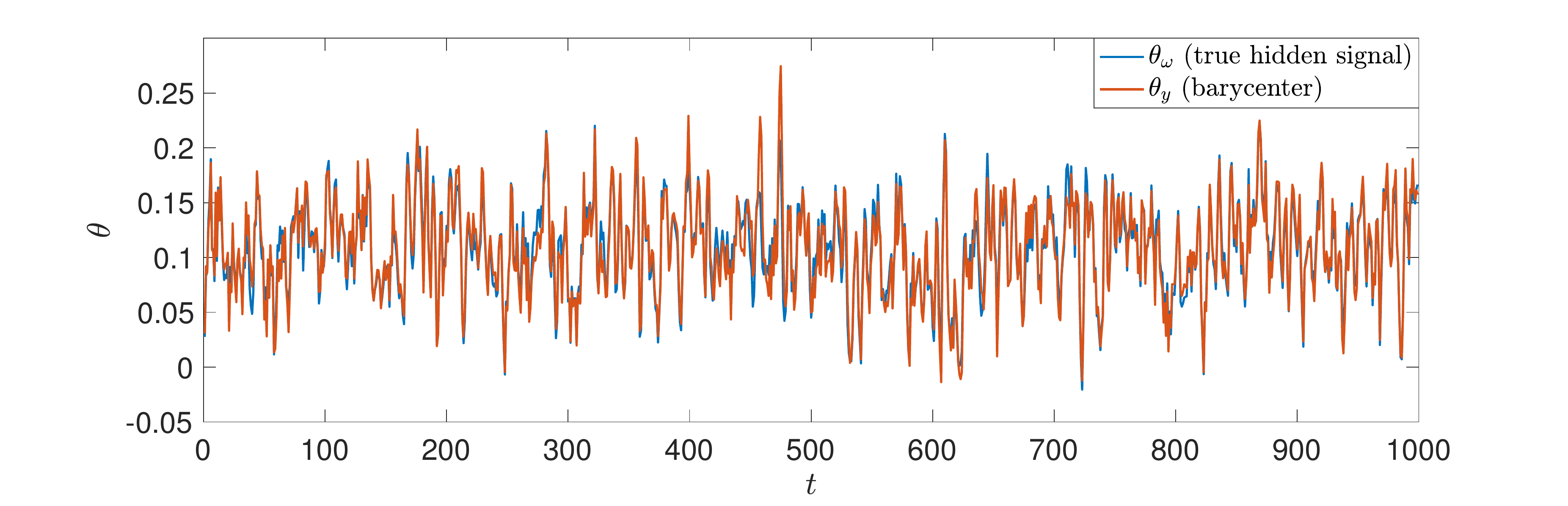}} &       \hspace{-1.1cm}                                
      \resizebox{69mm}{!}{\includegraphics{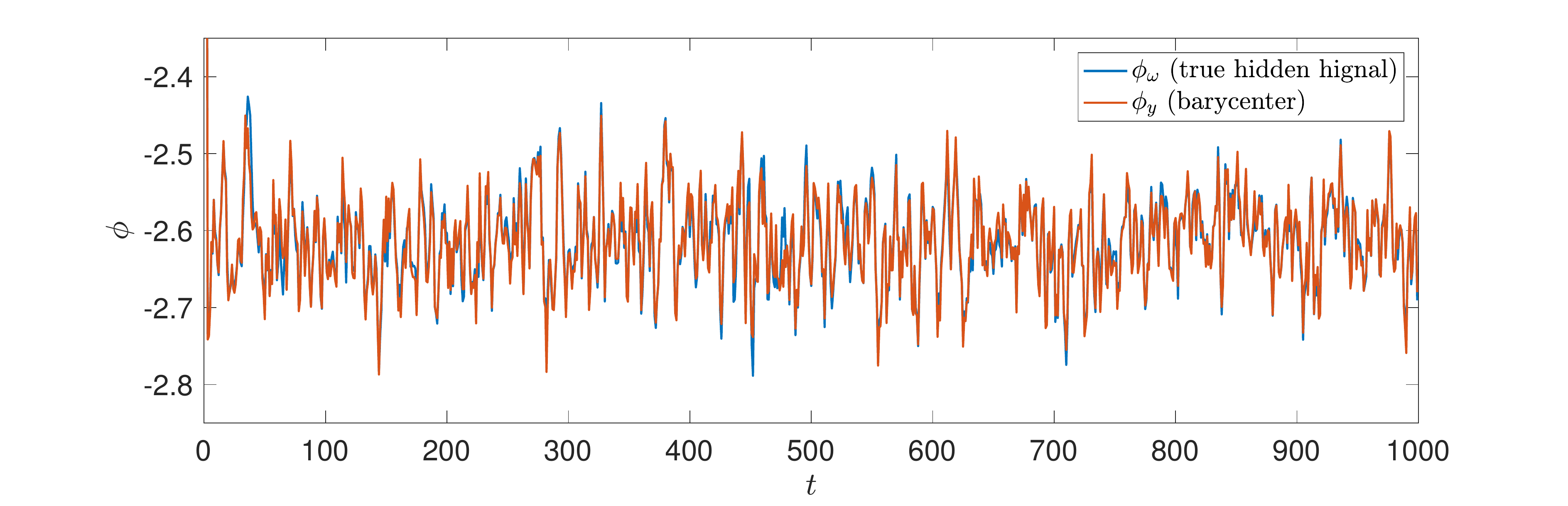}}\\
      \resizebox{69mm}{!}{\includegraphics{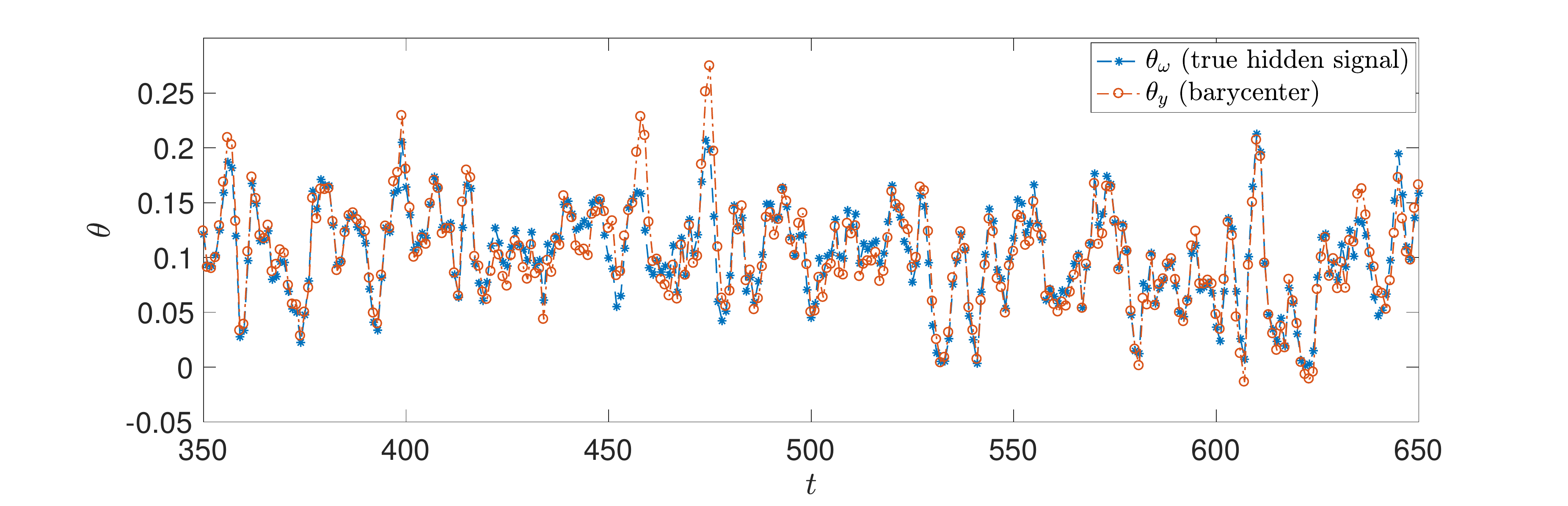}} &       \hspace{-1.1cm}                                
      \resizebox{69mm}{!}{\includegraphics{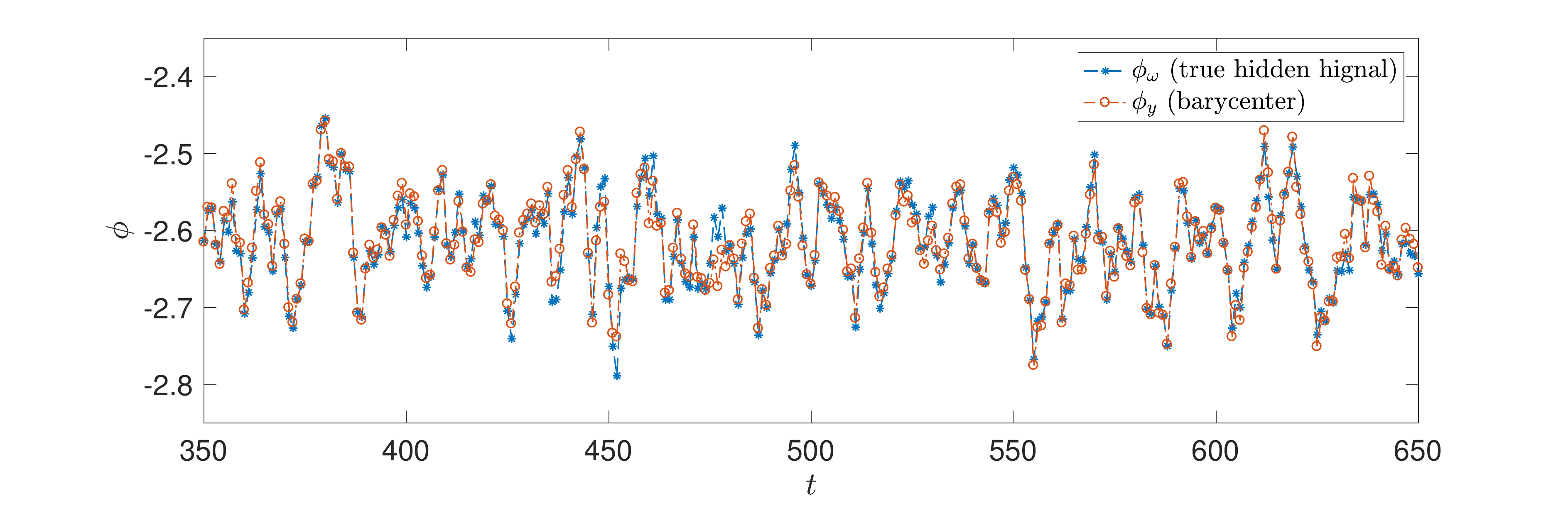}}\\
     \end{tabular}
    \caption{First row: moving average with window size $3$ of the time series of the hidden signal $\omega$  and of the filtered signal $y$ in spherical coordinates (left: longitude $\theta$, right: lattitude $\phi$). Second row: zoom of the first row.}
    \label{fig:MAtime}
%  \end{center}
\end{figure}

\section{Conclusions}\label{conclusions}
This work introduces the distributional barycenter problem, an extension of the optimal transport barycenter problem where the cost needs not be the expected value of a pairwise function, allowing more general costs needed in applications, such as a new cost penalizing non-isometric maps.

A novel numerical algorithm is introduced for the solution of the barycenter problem. The algorithm avoids the difficulties typical of adversarial approaches by slaving the discriminator to the generator. This results in a simpler approach that looks for a minimum rather than a saddle point of the objective function. The approach is essentially non-parametric, as the only parameter of the test functions and maps is the bandwidth of a kernel function.

\appendix

\section{Updating rules (\ref{eq:GD}) and (\ref{eq:firstOrd})}
\label{sec::gradHess}

This appendix calculates the gradient and Hessian of the objective functions $L$ in \eqref{eq:DDprob1} and \eqref{eq:DDprob2}, used to implement the explicit (\ref{eq:firstOrd}) and implicit (\ref{eq:SecondOrd}) schemes for updating the current position $y_i^n$ of the original sample points $x_{i}$. 

The update $y_{i}^{n}\rightarrow y_{i}^{n+1}$ is subtle, as both of the kernel $\cKa(y, w)$'s  arguments are evaluated at the sample points $y_{i}$, yet they play very different roles in the Lagrangian $L$: while $y= y_i$ represents the map $T(x_i, z_i)$ that $L$ is to be minimized over, the $w = y_i$ are the Kernels' centers, characterizing the test function $F = \rho(y|z)$ over which $L$ was originally to be maximized! Our methodology replaced this maximization by a slaving of $F$ to $T$, hence the appearance of the $y_i$ in $F$, yet $L$ must be minimized only over its first argument, not the second.  Thus, for gradient descent, one must use terms such as
$$  \frac{\partial \mathcal K_a(y, y^{n}_k)}{\partial y}\Big|_{y=y_{i}^{n}} $$
and, for implicit gradient descent,
\begin{multline}
\frac{\partial \mathcal K_a(y, y^{n+1}_k)}{\partial y} \Big|_{y=y_{i}^{n+1}} \approx 
\frac{\partial \mathcal K_a(y, y^{n}_k)}{\partial y} \Big|_{y=y_{i}^{n}} + \\
 \left. \frac{\partial^2 \cKa(y, w)}{\partial y^2} \right|_{\substack{y=y_{i}^{n}\\  w=y_{k}^n}} \left(y_{i}^{n+1} - y_{i}^{n} \right)
+  \frac{\partial^2 \cKa(y, w)}{\partial y \partial w} \Big|_{\substack{y=y_{i}^{n} \\ w=y_{k}^n}} \left(y_{k}^{n+1} - y_{k}^{n} \right).
\end{multline}

Though formulas below are developed for regular pairwise cost functions, their extension to the general case should be clear. The objective functions for problems \ref{prob:CKDE} and \ref{prob:factor} are:
\begin{eqnarray*}
\text{Kernel density estimation: } &&\mathcal{L}_1 =  \sum_i c(x_i,y_i) + \lambda  \sum_{i,k}  \mathcal{K}_a(y_{i},y_{k})C_{ik}.\\
\text{Parametric: }&&\mathcal{L}_2 = \sum_i c(x_i,y_i) + \lambda  \sum_{i,k}  f(y_{i})f(y_{k})C_{ki}.
\end{eqnarray*}

\noindent
\emph{Explicit: }
Formula (\ref{eq:firstOrd}) is equivalent to forward Euler for ODEs. 
We can update the position of each point $y_i\in \R^{d}$ independently, through the update rule $y^{n+1}_{i}=y^{n}_{i}-\eta\nabla_{y}L|_{y=y_{i}^{n}}$, where
$$
\left.\nabla_{y} \mathcal{L}_{1}\right|_{y=y_{i}^{n}} =\left[ \frac{\partial c(x_i,y)}{\partial  y} + \lambda \sum_{k} \frac{\partial \mathcal K_a(y, y^{n}_k)}{\partial y} C_{ik}\right]_{y=y_{i}^{n}},
$$
and 
$$
\left.\nabla_{y} \mathcal{L}_{2}\right|_{y=y_{i}^{n}}  =\left[ \frac{\partial c(x_i,y)}{\partial  y} + \lambda \sum_{k} \frac{\partial f(y_i)}{\partial y} f(y_k^{n})C_{ki}\right]_{y=y_{i}^{n}}.
$$

\noindent
\emph{Implicit:} 
This scheme, when applied to minimize the generic function $f(y,w)$ is obtained by the following approximation:
\begin{multline}
y^{n+1}=y^{n}-\eta f_{y}(y^{n+1},y^{n+1}) \\
\approx y^{n}-\eta\lbrace f_{y}(y^{n},y^{n}) +(y^{n+1}-y^{n})\left[f_{yy}(y^{n},y^{n})+f_{yw}(y^{n},y^{n})\right] \rbrace
\end{multline}
that, once rearranged, results in the scheme in (\ref{eq:SecondOrd}) \cite{essid2019implicit}:
\begin{equation}
 y^{n+1}=y^{n}-\eta\left[I +\eta(f^{n}_{yy}+f^{n}_{yw})\right]^{-1}f_{y}^{n}
\end{equation}
The Hessian matrix $\nabla_{yy}L_1$ in (\ref{eq:SecondOrd})  is therefore given by 
%$$
%\nabla_{yy}\mathcal{L}_{1}=H_{1,ii}+\sum_{k}H_{1,ik}
%$$
$$
\nabla_{yy}\mathcal{L}_1 = \mathcal{L}_1^{yy} + \mathcal{L}_1^{yw}.
$$
The matrix $\mathcal{L}_{1,yy}$ is diagonal, and we have:
$$
\mathcal{L}_{1,ii}^{yy} %= \frac{\partial^2 \mathcal L_1}{\partial y_i^2} 
=\left[ \frac{\partial^2 c(x_i,y)}{\partial  y^2} + \lambda \sum_{k} \frac{\partial^2 \cKa(y, y^{n}_{k})}{\partial y^2} C_{ik} \right]_{y=y_{i}^{n}} ,
 \quad \mathcal{L}_{1,ik}^{yw} %=  \frac{\partial^2 \mathcal L_1}{\partial y_i\partial y_k} 
= \lambda \left.\frac{\partial^2 \cKa(y,w)}{\partial y \partial w}\right|_{\substack{y=y^{n}_{i} \\ w=y^{n}_{k}}} C_{ik}.
$$
This calculation applies to pairwise cost functions, where the only non diagonal $\R^{d \times d}$ blocks arise from the $L_F$ in ${\mathcal L_1}$. One needs to adjust accordingly for more general costs.

Similarly for $\mathcal{L}_{2}$ we have $\nabla_{yy}\mathcal{L}_{2}= \mathcal{L}_2^{yy} + \mathcal{L}_2^{yw}$ with
$$
\mathcal{L}_{2,ii}^{yy} %= \frac{\partial^2 \mathcal L_2}{\partial y_i^2} 
= \left[ \frac{\partial^2 c(x_i,y)}{\partial  y^2} + \lambda \sum_{k} \frac{\partial^2 f(y) }{\partial y^2}f(y^{n}) C_{ki} \right]_{y=y_{i}^{n}},
\quad
\mathcal{L}_{2,ik}^{yw} %=  \frac{\partial^2 \mathcal L_2}{\partial y_i\partial y_k}
= \lambda \left.\frac{\partial f(y)}{\partial y}\right|_{y=y^{n}_{i}} \left.\frac{\partial f(y)}{\partial y}\right|_{y=y^{n}_{k}} C_{ki}.
$$
When $f$ is a vector, the gradient and Hessian have an additional sum over its components.

\section*{Acknowledgments}
Tabak's work was partially supported by NSF grant DMS-1715753 and ONR grant N00014-15-1-2355.

\bibliographystyle{plain}
\bibliography{TabakTrigilaZhao}

\begin{thebibliography}{10}

\bibitem{AguehBarycenters}
M~Agueh and G~Carlier.
\newblock Barycenter in the {W}asserstein space.
\newblock {\em SIAM J. MATH. ANAL.}, 43(2):094--924, 2011.

\bibitem{AHT}
Sigurd Angenent, Steven Haker, and Allen Tannenbaum.
\newblock Minimizing flows for the monge--kantorovich problem.
\newblock {\em SIAM journal on mathematical analysis}, 35(1):61--97, 2003.

\bibitem{BB}
Jean-David Benamou and Yann Brenier.
\newblock A computational fluid mechanics solution to the monge-kantorovich
  mass transfer problem.
\newblock {\em Numerische Mathematik}, 84(3):375--393, 2000.

\bibitem{carlier2015numerical}
Guillaume Carlier, Adam Oberman, and Edouard Oudet.
\newblock Numerical methods for matching for teams and wasserstein barycenters.
\newblock {\em ESAIM: Mathematical Modelling and Numerical Analysis},
  49(6):1621--1642, 2015.

\bibitem{de2003conditional}
Jan~G De~Gooijer and Dawit Zerom.
\newblock On conditional density estimation.
\newblock {\em Statistica Neerlandica}, 57(2):159--176, 2003.

\bibitem{essid2019implicit}
Montacer Essid, Esteban Tabak, and Giulio Trigila.
\newblock An implicit gradient-descent procedure for minimax problems.
\newblock {\em Submitted to Machine Learning (Springer)}, 2019.

\bibitem{ELT}
Montecer Essid, Debra Laefer, and Esteban~G Tabak.
\newblock Adaptive optimal transport.
\newblock {\em Submitted to Information and Inference}, 2018.

\bibitem{feldman2002monge}
Mikhail Feldman and Robert McCann.
\newblock Monge’s transport problem on a riemannian manifold.
\newblock {\em Transactions of the American Mathematical Society},
  354(4):1667--1697, 2002.

\bibitem{galichon2018optimal}
Alfred Galichon.
\newblock {\em Optimal transport methods in economics}.
\newblock Princeton University Press, 2018.

\bibitem{kantorovich1948v}
Leonid~V Kantorovich.
\newblock On a problem of monge.
\newblock {\em Uspekhi Mat. Nauk, 3, No. 2}, pages 225--226, 1948.

\bibitem{kolouri2017optimal}
Soheil Kolouri, Se~Rim Park, Matthew Thorpe, Dejan Slepcev, and Gustavo~K
  Rohde.
\newblock Optimal mass transport: Signal processing and machine-learning
  applications.
\newblock {\em IEEE signal processing magazine}, 34(4):43--59, 2017.

\bibitem{kolouri2016continuous}
Soheil Kolouri, Akif~B Tosun, John~A Ozolek, and Gustavo~K Rohde.
\newblock A continuous linear optimal transport approach for pattern analysis
  in image datasets.
\newblock {\em Pattern recognition}, 51:453--462, 2016.

\bibitem{kuang2017preconditioning}
Max Kuang and Esteban~G Tabak.
\newblock Preconditioning of optimal transport.
\newblock {\em SIAM Journal on Scientific Computing}, 39(4):A1793--A1810, 2017.

\bibitem{kuang2019sample}
Max Kuang and Esteban~G Tabak.
\newblock Sample-based optimal transport and barycenter problems.
\newblock {\em Communications on Pure and Applied Mathematics},
  72(8):1581--1630, 2019.

\bibitem{lavenant2018dynamical}
Hugo Lavenant, Sebastian Claici, Edward Chien, and Justin Solomon.
\newblock Dynamical optimal transport on discrete surfaces.
\newblock {\em ACM Transactions on Graphics (TOG)}, 37(6):1--16, 2018.

\bibitem{lecun-98}
Y.~LeCun, L.~Bottou, Y.~Bengio, and P.~Haffner.
\newblock Gradient-based learning applied to document recognition.
\newblock {\em Proceedings of the IEEE}, 86(11):2278--2324, November 1998.

\bibitem{monge1781memoire}
Gaspard Monge.
\newblock {\em M{\'e}moire sur la th{\'e}orie des d{\'e}blais et des remblais}.
\newblock De l'Imprimerie Royale, 1781.

\bibitem{nocedal2006numerical}
Jorge Nocedal and Stephen Wright.
\newblock {\em Numerical optimization}.
\newblock Springer Science \& Business Media, 2006.

\bibitem{pavon2018data}
Michele Pavon, Esteban~G Tabak, and Giulio Trigila.
\newblock The data-driven schroedinger bridge.
\newblock {\em To appear in Communication of Pure and Applied Mathematics},
  2020.

\bibitem{rabin2014adaptive}
Julien Rabin, Sira Ferradans, and Nicolas Papadakis.
\newblock Adaptive color transfer with relaxed optimal transport.
\newblock In {\em 2014 IEEE International Conference on Image Processing
  (ICIP)}, pages 4852--4856. IEEE, 2014.

\bibitem{rosenblatt1969conditional}
Murray Rosenblatt.
\newblock Conditional probability density and regression estimators.
\newblock {\em Multivariate analysis II}, 25:31, 1969.

\bibitem{santambrogio2015optimal}
Filippo Santambrogio.
\newblock Optimal transport for applied mathematicians.
\newblock {\em Birk{\"a}user, NY}, 55(58-63):94, 2015.

\bibitem{sapienza2018weighted}
Facundo Sapienza, Pablo Groisman, and Matthieu Jonckheere.
\newblock Weighted geodesic distance following fermat's principle.
\newblock {\em 6th International Conference on Learning Representations}, 2018.

\bibitem{schiebinger2019optimal}
Geoffrey Schiebinger, Jian Shu, Marcin Tabaka, Brian Cleary, Vidya Subramanian,
  Aryeh Solomon, Joshua Gould, Siyan Liu, Stacie Lin, Peter Berube, et~al.
\newblock Optimal-transport analysis of single-cell gene expression identifies
  developmental trajectories in reprogramming.
\newblock {\em Cell}, 176(4):928--943, 2019.

\bibitem{solomon2015convolutional}
Justin Solomon, Fernando De~Goes, Gabriel Peyr{\'e}, Marco Cuturi, Adrian
  Butscher, Andy Nguyen, Tao Du, and Leonidas Guibas.
\newblock Convolutional wasserstein distances: Efficient optimal transportation
  on geometric domains.
\newblock {\em ACM Transactions on Graphics (TOG)}, 34(4):66, 2015.

\bibitem{TT3}
Esteban~G Tabak and Giulio Trigila.
\newblock Conditional expectation estimation through attributable components.
\newblock {\em Information and Inference: A Journal of the IMA}, 128(00), 2018.

\bibitem{TT2}
Esteban~G Tabak and Giulio Trigila.
\newblock Explanation of variability and removal of confounding factors from
  data through optimal transport.
\newblock {\em Communications on Pure and Applied Mathematics}, 71(1):163--199,
  2018.

\bibitem{tabak2020conditional}
Esteban~G Tabak, Giulio Trigila, and Wenjun Zhao.
\newblock Conditional density estimation and simulation through optimal
  transport.
\newblock {\em Machine Learning}, pages 1--24, 2020.

\bibitem{tenetov2018fast}
Evgeny Tenetov, Gershon Wolansky, and Ron Kimmel.
\newblock Fast entropic regularized optimal transport using semidiscrete cost
  approximation.
\newblock {\em SIAM Journal on Scientific Computing}, 40(5):A3400--A3422, 2018.

\bibitem{TT1}
Giulio Trigila and Esteban~G Tabak.
\newblock Data-driven optimal transport.
\newblock {\em Communications on Pure and Applied Mathematics}, 69(4):613--648,
  2016.

\bibitem{wang2010optimal}
Wei Wang, John~A Ozolek, Dejan Slep{\v{c}}ev, Ann~B Lee, Cheng Chen, and
  Gustavo~K Rohde.
\newblock An optimal transportation approach for nuclear structure-based
  pathology.
\newblock {\em IEEE transactions on medical imaging}, 30(3):621--631, 2010.

\bibitem{wang2013linear}
Wei Wang, Dejan Slep{\v{c}}ev, Saurav Basu, John~A Ozolek, and Gustavo~K Rohde.
\newblock A linear optimal transportation framework for quantifying and
  visualizing variations in sets of images.
\newblock {\em International journal of computer vision}, 101(2):254--269,
  2013.

\bibitem{yair2019optimal}
Or~Yair, Felix Dietrich, Ronen Talmon, and Ioannis~G Kevrekidis.
\newblock Optimal transport on the manifold of spd matrices for domain
  adaptation.
\newblock {\em arXiv preprint arXiv:1906.00616}, 2019.

\bibitem{yang2019conditional}
Hongkang Yang and Esteban~G Tabak.
\newblock Conditional density estimation, latent variable discovery and optimal
  transport.
\newblock {\em arXiv preprint arXiv:1910.14090}, 2019.

\bibitem{yang2018complex}
Yang Yang, Yi-Feng Wu, De-Chuan Zhan, Zhi-Bin Liu, and Yuan Jiang.
\newblock Complex object classification: A multi-modal multi-instance
  multi-label deep network with optimal transport.
\newblock In {\em Proceedings of the 24th ACM SIGKDD International Conference
  on Knowledge Discovery \& Data Mining}, pages 2594--2603, 2018.

\end{thebibliography}

\end{document}